\documentclass[11pt, english, reqno]{amsart}

\usepackage{amsmath, amsfonts, amssymb, amscd, enumerate, color, graphics, mathtools, wasysym, empheq, stmaryrd, slashed, babel, mathtools, tikz, hyperref}
\usetikzlibrary{arrows,decorations.markings,positioning}
\tikzset{inner sep=0pt, node distance=5mm,
  root/.style={circle,draw,minimum size=5pt,thick},
  broot/.style={circle,draw,minimum size=5pt,thick,fill},
  xroot/.style={circle,draw,minimum size=5pt,thick,label=below:$\times$},
  doublearrow/.style={postaction={decorate},   decoration={markings,mark=at position .6 with {\arrow[line width=1.2pt]{>}}},double distance=1.6pt,thick},
  rdoublearrow/.style={postaction={decorate},   decoration={markings,mark=at position .4 with {\arrowreversed[line width=1.2pt]{>}}},double distance=1.6pt,thick},
	curvedline/.style={bend=right}
} 
\usepackage{here}
\usepackage{pdflscape}
\usepackage[margin=3cm]{geometry}
\usepackage[graphicx]{realboxes}
\usepackage[figuresright]{rotating}
\theoremstyle{plain}
\newtheorem{theo}{Theorem}[section]

\theoremstyle{definition}

\newtheorem{example}[theo]{Example}
\newtheorem{definition}[theo]{Definition}
\theoremstyle{plain}
\newtheorem{lemma}[theo]{Lemma}
\newtheorem{theorem}[theo]{Theorem}

\newtheorem{proposition}[theo]{Proposition}

\theoremstyle{definition}

\newtheorem{remark}[theo]{Remark}

\newenvironment{colored}{\color{red}}{}
\newcommand{\bc}{\begin{colored}}
\newcommand{\ec}{\end{colored}}

\newcommand{\beq}{\begin{equation}}
\newcommand{\eeq}{\end{equation}}
\renewcommand{\a}{\alpha}
\renewcommand{\b}{\beta}

\newcommand{\e}{\epsilon}

\renewcommand{\l}{\lambda}

\newcommand{\gJ}{\mathfrak{J}}

\newcommand{\bC}{\mathbb{C}}

\newcommand{\bR}{\mathbb{R}}
\newcommand{\bZ}{\mathbb{Z}}

\newcommand{\bK}{\mathbb{K}}

\newcommand{\gb}{\mathfrak{b}}
\newcommand{\gc}{\mathfrak{c}}

\renewcommand{\gg}{\mathfrak{g}}
\newcommand{\gh}{\mathfrak{h}}
\newcommand{\gk}{\mathfrak{k}}

\newcommand{\gm}{\mathfrak{m}}
\newcommand{\gn}{\mathfrak{n}}

\newcommand{\gq}{\mathfrak{q}}

\newcommand{\gu}{\mathfrak{u}}

\newcommand{\so}{\mathfrak{so}}
\newcommand{\su}{\mathfrak{su}}

\newcommand{\ggl}{\mathfrak{gl}}

\newcommand\GL{\mathrm{GL}}
\newcommand\SL{\mathrm{SL}}
\newcommand\SO{\mathrm{SO}}
\newcommand\SU{\mathrm{SU}}

\newcommand\Sp{\mathrm{Sp}}
\renewcommand\sp{\mathfrak{sp}}
\renewcommand\sl{\mathfrak{sl}}

\newcommand{\cC}{\mathcal{C}}
\newcommand{\cD}{\mathcal{D}}

\newcommand{\cF}{\mathcal{F}}

\newcommand{\cI}{\mathcal{I}}
\newcommand{\cJ}{\mathcal{J}}

\newcommand{\cL}{\mathcal{L}}

\newcommand{\cU}{\mathcal{U}}

\renewcommand{\square}{\kern1pt\vbox
{\hrule height 0.6pt\hbox{\vrule width 0.6pt\hskip 3pt
\vbox{\vskip 6pt}\hskip 3pt\vrule width 0.6pt}\hrule height0.6pt}\kern1pt}
\DeclareMathOperator\tr{tr}

\DeclareMathOperator\Aut{Aut}

\DeclareMathOperator\ad{ad}

\DeclareMathOperator\Id{Id}
\DeclareMathOperator{\sgn}{sgn}

\renewcommand\Re{\operatorname{\mathfrak{Re}}}
\renewcommand\Im{\operatorname{Im}}

\newcommand{\wt}{\widetilde}
\newcommand{\wh}{\widehat}

\newcommand{\be}{\begin{equation}}
\newcommand{\ee}{\end{equation}}

\def\<#1,#2>{\langle\,#1,\,#2\,\rangle}
\newcommand{\arr}{\begin{array}{rlll}}
\newcommand{\ea}{\end{array}}
\newcommand{\bea}{\begin{eqnarray}}
\newcommand{\eea}{\end{eqnarray}}
\newcommand{\bean}{\begin{eqnarray*}}
\newcommand{\eean}{\end{eqnarray*}}

\def\sideremark#1{\ifvmode\leavevmode\fi\vadjust{%            The remark
\vbox to0pt{\hbox to 0pt{\hskip\hsize\hskip1em%               will appear only
\vbox{\hsize3cm\tiny\raggedright\pretolerance10000%          on the side
\noindent #1\hfill}\hss}\vbox to8pt{\vfil}\vss}}}%           in 3cm
\newcounter{ssig}
\newcounter{ttig}
\newcommand{\under}[1]{{\underline{#1}\,}}

\title[Homogeneous models for Levi degenerate CR manifolds]
{Homogeneous models for  \\Levi degenerate CR manifolds}
\address{Andrea Santi, School of Mathematics of the University of Edinburgh,
Kings Buildings, EH$9$ $3$JZ, Edinburgh, Scotland (UK)}
\email{asanti.math@gmail.com}
\thanks{EMPG-15-21}
\author{Andrea Santi}
\begin{document}
\begin{abstract} 
We extend the notion of a fundamental negatively $\bZ$-graded Lie algebra $\gm_x=\bigoplus_{p\leq -1}\gm_x^p$ associated to any point of a Levi-nondegenerate CR manifold
to the class of $k$-nondegenerate CR manifolds $(M,\cD,\cJ)$ for all $k\geq 2$ and call this invariant the core at $x\in M$. It consists of a $\bZ$-graded vector space $\gm_x=\bigoplus_{p\leq k-2}\gm_x^p$ of height $k-2$ endowed with the natural algebraic structure induced by the Tanaka and Freeman sequences
of $(M,\cD,\cJ)$ and the Levi forms of higher order. In the case of CR manifolds of hypersurface type we propose a definition of a homogeneous model of type $\gm$, that is, a homogeneous $k$-nondegenerate CR manifold $M=G/G_o$ with core $\gm$ associated with an appropriate $\bZ$-graded Lie algebra $Lie(G)=\gg=\bigoplus\gg^p$ and subalgebra $Lie(G_o)=\gg_o=\bigoplus\gg_o^p$ of the nonnegative part $\bigoplus_{p\geq 0}\gg^p$. It generalizes the classical notion of Tanaka of homogeneous model for Levi-nondegenerate CR manifolds and the tube over the future light cone, the unique (up to local CR diffeomorphisms) maximally homogeneous $5$-dimensional $2$-nondegenerate CR manifold. We investigate the basic properties of cores and models and study the $7$-dimensional CR manifolds of hypersurface type from this perspective. We first classify cores of $7$-dimensional $2$-nondegenerate CR manifolds up to isomorphism and then construct homogeneous models for seven of these classes. We finally show that there exists a unique core and homogeneous model in the $3$-nondegenerate class.
\end{abstract}
\maketitle
\null \vspace*{-.50in}
\tableofcontents
\section{Introduction}
\setcounter{equation}{0}
\setcounter{section}{1}
A CR structure on a manifold $M$ of dimension $m=2d+c$ is a pair $(\cD, \mathcal J)$
given by a rank $2d$ distribution $\cD\subset TM$ and a smooth family of complex structures $\cJ_x:\cD|_x\longrightarrow\cD|_x$ with the property that the $\cJ$-eigenspace distribution $\cD^{10}\subset T^\bC M$ corresponding to the eigenvalue $+i$ is involutive.
The integers $d,c$ are respectively called  the CR dimension and codimension of the CR structure $(\cD,\cJ)$ and
the complex vector bundles $\cD^{10}$ and  $\cD^{01} = \overline{\cD^{10}}$   are the holomorphic and anti-holomorphic bundles.
We refer the reader to e.g. \cite{BER, DT, Tu} for a general introduction to CR manifolds.

The equivalence problem of CR manifolds, and the associated problem of constructing a full system of invariants which distinguishes one CR manifold from another, is well-understood, at least in the case of strongly regular CR manifolds with \emph{nondegenerate} Levi form (see \cite{CM, Ta0, Ta1}). We recall that the Levi form of a CR manifold $(M,\cD,\cJ)$ is the $\cJ$-Hermitian skew-symmetric bundle map
\beq
\label{intro:lf}
\cL:\cD\times\cD\longrightarrow TM/\cD
\eeq
defined by $\cL(v,w):=[X^{(v)},X^{(w)}]|_x\mod\cD|_x$, where $v,w\in\cD|_x$ and $X^{(v)}$ and $X^{(w)}$ are sections of $\cD$ that extends respectively $v$ and $w$ around $x\in M$. It is the first main invariant of $(M,\cD,\cJ)$ and it is called nondegenerate if for any nonzero $v\in\cD|_x$ there is an element $w\in\cD|_x$ with $\cL(v,w)\neq 0$. The equivalence problem of strongly regular Levi-nondegenerate CR manifolds of any CR dimension and codimension was solved by N.\ Tanaka in \cite{Ta0, Ta1}  using a generalization of the usual prolongation procedure of $G$-structures \cite{St}. He observed that, under fairly general assumptions, any distribution $\cD$ on a manifold $M$ determines a filtration in negative degrees of each tangent space (see \S \ref{sec:2.1} for the definition of this sequence)
\beq
\label{fil:T}
T_xM=\cD_{-\mu}|_{x}\supset\cD_{-\mu+1}|_x\supset\cdots\supset\cD_{-2}|_x\supset\cD_{-1}|_x=\cD|_x\,,
\eeq
and that the associated $\bZ$-graded vector space
$$
\gm_x=\operatorname{gr}(T_xM)=\gm^{-\mu}_x\oplus\gm^{-\mu+1}_x\oplus\cdots\oplus\gm^{-2}_x\oplus\gm^{-1}_x
$$
inherits, by the commutators of vector fields, a natural structure of $\bZ$-graded Lie algebra $\gm_x=\bigoplus_{p\in\bZ}\gm_x^p$  which enjoys the following properties:
\begin{itemize}
\item[(i)] $\gm_x$ is \emph{negatively} $\bZ$-graded, that is $\gm^p_x=0$ for all $p\geq 0$,
\item[(ii)] $\gm_x$ is of \emph{depth $\mu$}, that is $\gm_x^p=0$ for all $p<-\mu$,
\item[(iii)] $\gm_x$ is \emph{fundamental}, that is it is generated by $\gm_x^{-1}$.
\end{itemize} 
He assumed that $(M,\cD)$ is strongly regular of type $\gm$, i.e., with all $\gm_x$ isomorphic to a fixed fundamental $\bZ$-graded Lie algebra $\gm=\gm^{-\mu}\oplus\cdots\oplus\gm^{-1}$, and noted that the presence of an additional geometric datum supported on the distribution $\cD$ corresponds to the assignment of a subalgebra $\gg^0$ of the Lie algebra $\mathfrak{der}(\gm)$ of all zero-degree derivations of $\gm$. In the case of CR manifolds the existence of $\cJ$ corresponds to the existence of a complex structure $J:\gm^{-1}\longrightarrow\gm^{-1}$ satisfying $[Jv,Jw]=[v,w]$ for all $v,w \in\gm^{-1}$, $$\gg^0=\mathfrak{der}(\gm,J)=\left\{X\in\mathfrak{der}(\gm)\,|\,X|_{\gm^{-1}}\circ J=J\circ X|_{\gm^{-1}} \right\}$$
 and the CR structure can be encoded in an appropriate principal bundle $\pi:P\longrightarrow M$ of ``graded frames'' on $M$ with structure group $G^0$, $Lie(G^0)=\gg^0$.

He then showed that any pair $(\gm,\gg^0)$ admits a unique maximal transitive prolongation to positive degrees, a (possibly infinite-dimensional) $\bZ$-graded Lie algebra 
\beq
\label{eq:max}
\gg=\bigoplus_{p\in \bZ} \gg^{p}
\eeq
satisfying:
\begin{enumerate}
\item[(i)]
$\gg^{p}$ is finite-dimensional for every $p\in\bZ$,
\item[(ii)]
$\gg^{p}=\gm^p$ for every $-\mu\leq p\leq -1$, $\gg^{0}$ is equal to the given subalgebra of $\mathfrak{der}(\gm)$ and $\gg^{p}=0$ for every $p<-\mu$,
\item[(iii)]
for all $p\geq 0$, if $X\in\gg^{p}$ is an element such that $[X,\gg^{-1}]=0$, then $X=0$ (\emph{transitivity}),
\item[(iv)]
$\gg$ is \emph{maximal} with these properties.
\end{enumerate}
In the case of CR manifolds and $\gg^0=\mathfrak{der}(\gm,J)$ one has that $\gg$ is of finite type, that is there is a nonnegative integer $\ell$ such that $\gg^{p}=0$ for all $p> \ell$, if and only if for any nonzero $v\in\gm^{-1}$ there is $w\in\gm^{-1}$ with $[v,w]\neq 0$ (see \cite{Ta1}; see also \cite{MeNa1, MeNa2} for further properties of \eqref{eq:max}). This condition clearly corresponds to the nondegeneracy of the Levi form \eqref{intro:lf}. 
In this case the nilpotent Lie group associated with $\gm$ has a natural left-invariant CR structure 
and is locally identifiable with the homogeneous model $M=G/G_o$, where $Lie(G)=\gg$, $Lie(G_o)=\bigoplus_{p\geq 0}\gg^p$.
Moreover there is a finite tower
$$
\cdots\longrightarrow P^{i}\overset{\pi_i}{\longrightarrow} P^{i-1}\longrightarrow\cdots\longrightarrow P^{1}\overset{\pi_1}{\longrightarrow} P\overset{\pi}{\longrightarrow} M
$$
of bundles $\pi_i:P^{i}\longrightarrow P^{i-1}$ with abelian structure groups $\gg^{i}$ with the property that any CR automorphism of $(M,\cD,\cJ)$ admits a unique lift to each term of the tower and thus ending with an automorphism of the absolute parallelism $\pi_{\ell+1}:P^{\ell+1}\longrightarrow P^{\ell}$ (in some notable cases a Cartan connection modeled on $M=G/G_o$, cf. \cite{Ta2, Ta3, AS}).
This reduces the equivalence problem of strongly regular Levi-nondegenerate CR manifolds to that of the associated absolute parallelisms, which can be dealt with classical results \cite{St}. 

Note that an important result of the prolongation procedure is given by the fact that both $\gm$ and $\gg$ are $\bZ$-graded Lie algebras, with compatible $\bZ$-gradings. Indeed the idea of considering
the filtration \eqref{fil:T} was also independently proposed
by B.\ Weisfeiler in \cite{Weis} and the associated $\bZ$-graded Lie algebras of depth $\mu>1$ were used for filling a gap in E.\ Cartan classification of the infinite
primitive Lie algebras of vector fields.

Unfortunately, for \emph{degenerate} CR structures,
the Tanaka approach does not work since the prolongation \eqref{eq:max} is infinite. In this paper we propose
a way to overcome this difficulty.
We recall that the first examples of Levi-degenerate CR manifolds are the products $M=\overline{M}\times \bC^s$ where $\overline M$ is Levi-nondegenerate. More generally any CR manifold $(M,\cD,\cJ)$ is endowed also with a second natural filtration. It is a decreasing filtration of $\cD^{10}|_x$ (see \cite{Fr} and \S \ref{sec:2.1})
\beq
\label{fil:F}
\cD^{10}|_x=\cF_{-1}^{10}|_x\supset\cF^{10}_{0}|_{x}\supset\cF^{10}_{1}|_x \supset \cdots\supset\cF^{10}_{p}|_x\supset\cF^{10}_{ p+1}|_x
\supset\cdots
\eeq
which stabilizes at a certain point $k-1\geq -1$, i.e., with $\cF^{10}_{k-1+p}=\cF^{10}_{k-1}$ for all $p\geq 0$, and with the property that the first stabilizing distribution $\cF^{10}_{k-1}$ is nonzero if and only if $(M,\cD,\cJ)$ is locally a product as above. The CR manifolds $(M,\cD,\cJ)$ with $\cF_{k-1}^{10}=0$ are called k-nondegenerate (for $k=1$ they reduce to the Levi-nondegenerate CR manifolds) and they are not, even locally, products of the form $M=\overline M\times \bC^s$. Our starting point is very simple and can be shortly summarized  as follows: we combine filtrations \eqref{fil:T} and \eqref{fil:F} into a single filtration and consider the resulting
pointwise invariant $\gm_x$ (called \emph{core} throughout the paper). In a sense the main idea of the paper is  complementary to the one of Tanaka and Weisfeiler: the invariant $\gm_x$ not only has a depth $\mu>1$ but, since the CR
manifold $(M,\cD,J)$ is k-nondegenerate, it  also has a nonnegative height. Indeed $\gm_x^p\neq 0$
exactly if $-\mu\leq p\leq k-2$, with case $k = 1$ reducing to the case of nondegenerate CR manifolds and negatively graded Lie algebras.

More precisely the {\bf main aim} of this paper is to introduce a notion of \emph{homogeneous model $M=G/G_o$ of type $\gm$} for $k$-nondegenerate CR manifolds when $k\geq 2$. This has a double motivation. On a one hand it provides with a method to construct CR manifolds with a uniformly degenerate Levi form. 
We remark that the construction of CR manifolds which are uniformly $k$-nondegenerate at all points with $k\geq 2$ is a difficult problem and that homogeneous CR manifolds have been constructed in \cite{FK1, FK2, Fel, KZ, MeNa3}, using various techniques. Our definition of a model is new and  gives another means to build up $k$-nondegenerate homogeneous CR manifolds. We stress that our results are effective and in principle can be  applied  without any restriction on the dimension $m$. Moreover, since the construction of models is tightly related to the cores $\gm$, we also automatically have
that two models $M=G/G_o$ and $M'=G'/G_o'$ of different types $\gm$ and $\gm'$ are not, even locally, CR diffeomorphic.

On the other hand the Lie algebra $\gg=Lie(G)$ of infinitesimal CR automorphisms of $M=G/G_o$ satisfies properties which directly generalize those of the maximal prolongation \eqref{eq:max} and we expect these homogeneous manifolds as natural candidates for solutions of equivalence problems, especially for the construction of Cartan connections. On this regard, we remark that the equivalence problem for all $5$-dimensional $2$-nondegenerate CR manifolds has been recently settled in 
\cite{IZ, MS}. In particular the Cartan connection  ``\'a la Tanaka'' determined  in \cite{MS} is modeled on the projective completion $M=\mathrm{SO}^o(3,2)/G_o$ of the tube over the future light cone (see \cite{FK2, SV} for its main properties) and the tower of fiber bundles constructed as the geometric counterpart of a special $\mathbb Z$-grading of the Lie algebra $Lie(G)=\mathfrak{so}(3,2)$, which is indeed of the form we study in this paper (see \cite[\S 3.2]{MS} and Example \ref{example01}). The solution of the same equivalence problem presented in \cite{IZ} uses different methods and the associated set of invariants does not correspond to a Cartan connection modeled on $M=\mathrm{SO}^o(3,2)/G_o$. The equivalence problem of 
$7$-dimensional $2$-nondegenerate CR manifolds has been recently considered in \cite{Por} under certain additional constraints. The solution is given by an absolute parallelism taking values in $\gg=\mathfrak{su}(2, 2)$ or $\gg=\mathfrak{su}(1,3)$, the algebras of infinitesimal CR automorphisms of two of the seven $7$-dimensional $2$-nondegenerate homogeneous models we obtain in this paper.
\vskip0.15cm\par
We now give the detailed description of the contents of the paper. 
\vskip0.15cm\par
In \S \ref{sec:2.1} we recall the relevant definitions and combine the Tanaka \eqref{fil:T} and Freeman \eqref{fil:F} sequences to construct the core $\gm_x=\bigoplus_{p\leq k-2} \gm_x^p$ associated with a $k$-nondegenerate CR manifold $(M,\cD,\cJ)$ at a point $x\in M$ (Definition \ref{acore} and Lemma \ref{lemmino}). It is an invariant which generalizes the notion of a fundamental algebra of a nondegenerate CR manifold but in sharp contrast with this case it 
turns out that the collection of all Levi forms of higher order does not induce a structure of a Lie algebra on $\gm_x$.
Sections \S \ref{sinizia} and \S\ref{sec:4.1} are therefore entirely dedicated to constructing homogeneous CR manifolds with a given (abstract) core $\gm$. To this aim, we first recall in \S \ref{sec:2.2} the correspondence developed in \cite{MeNa3, Fel} between germs of homogeneous CR manifolds $M=G/G_o$ and CR algebras $(\gg,\gq)$, i.e., Lie algebras of local infinitesimal CR automorphisms, and note that the core can be easily read off from $(\gg,\gq)$. In other words we consider homogeneous CR manifolds from a local point of view. In particular  we will not address in this paper the question of their global existence in full generality, which would involve criteria for the closedness of the isotropy subgroup, but we will easily check that on a case by case basis.

Starting from \S \ref{sinizia} we restrict to the case of CR manifolds of hypersurface type, that is with $c=1$.
We recognize higher order Levi forms as
defining components of the Tanaka-Weisfeiler grading of the real contact algebra $\gc$ (see \cite{MT, SS})
and describe in Proposition \ref{bed} a CR algebra $(\gc,\gu)$ which is \emph{universal} in the sense that 
any abstract core $\gm$ (of hypersurface type) has a natural immersion $\varphi:\gm\longrightarrow\mathfrak{M}$ into the core $\mathfrak{M}$ of $(\gc,\gu)$. We remark that the Lie algebra $\gc$ is infinite-dimensional, consistent with the fact that it has to accomodate for abstract cores $\gm=\bigoplus_{p\leq k-2}\gm^p$ of any possible height $\operatorname{ht}(\gm)=k-2$. In Proposition \ref{noioso} we describe the Lie algebra structure of $\gc$, generalizing a description of Morimoto and Tanaka \cite{MT}  (the recursive expression of the brackets is rather involved but it is fully needed in the construction of the $7$-dimensional $3$-nondegenerate homogeneous model in \S \ref{exexamplesII}).

In Definition \ref{ukip} we give the definition of a \emph{model of type $\gm$} as an appropriate $\bZ$-graded subalgebra $\gg=\bigoplus\gg^p$ of the universal CR algebra tightly related to $\gm$ and prove in Theorem \ref{primoteorema} of \S\ref{sec:4.1} that any model $\gg$ determines a CR algebra $(\gg,\gq)$ with core $\gm$ in a canonical way. We note that the real isotropy algebra $\gg_o$ of $(\gg,\gq)$ satisfies $\gg/\gg_o\simeq\gm$ and that it is a \emph{proper} subalgebra of the nonnegative part $\bigoplus_{p\geq 0}\gg^p$ of $\gg$ whenever $k\geq 2$. Section \ref{ukip} ends with Example \ref{example00} and Example \ref{example01}.  

In \S \ref{sec:7cores}, \S \ref{exexamples} and \S \ref{exexamplesII} we consider applications of the  general theory
to the case of CR manifolds of dimension $m=7$. We remark that this is the smallest possible dimension for a CR hypersurface to be $3$-nondegenerate and that a full classification  of homogeneous $2$-nondegenerate hypersurfaces in dimension $m=5$ up to local CR equivalence has been given in \cite{FK2}. In order to construct models which are inequivalent, we first classify in \S \ref{sec:7cores} the $7$-dimensional abstract cores up to isomorphism, see Theorem 
 \ref{firstcores} and Theorem \ref{secondcores}. This result might be of independent interest, as the core is a basic invariant of any $7$-dimensional CR manifold. The proof relies on a down-to-earth description of representatives for the orbit spaces of the actions of $\mathrm{S0}_{3}(\bR)$ and $\mathrm{S0}^+(2,1)$ on the complex projective plane $\mathbb{P}^2(\bC)$ (see Proposition \ref{firstorbitprop} and Proposition \ref{secondorbitprop}). Finally we introduce a notion of admissibility for an orbit and the associate abstract core in Definition \ref{def:admissibility}.

In \S \ref{exexamples} and \S\ref{exexamplesII} we construct models $\gg$ corresponding to the abstract cores $\gm$ determined in \S \ref{sec:7cores}. In Theorem \ref{mainI} of \S\ref{exexamples} we prove the existence of $\gg$ for all the admissible $2$-nondegenerate abstract cores. 
More precisely we obtain three homogeneous models given by some $\bZ$-gradings $\gg=\bigoplus\gg^p$ of the simple Lie algebras $\gg=\sl_4(\bR)$, $\su(1,3)$ and $\su(2,2)$ in Theorem \ref{thmsimple} and four other models of the form $\gg=\gg^{-2}\oplus\gg^{-1}\oplus\gg^{0}$ in Theorem \ref{mainII}. We finally show in Theorem \ref{thmstrange} that there exists a unique model in the $3$-nondegenerate class.

Before concluding, we recall that a $7$-dimensional $2$-nondegenerate (resp. $3$-nondegenerate) CR manifold locally homogeneous with respect to an algebra $\gg$ isomorphic to $\mathfrak{su}(2,2)$ and $\mathfrak{su}(1,3)$ (resp. with $\dim(\gg)=8$) has also appeared in \cite{KZ, Por} (resp. \cite{FK2}). We also remark that a concept of model has been introduced already in \cite{Bel}, in a more analytic context. Analogies and differences with \cite{Bel} will be discussed elsewhere.
\vskip0.15cm\par\noindent
{\it Notations.} Given a real vector space $V$ we set $V^\times=\left\{v\in V\,|\,v\neq0\right\}$ and $V^\bC=V\otimes \bC$. In the case of models and cores we also use the shortcuts $\wh\gg=\gg\otimes\bC$ and $\wh\gm=\gm\otimes\bC$. 
We denote by $\underline{\mathcal K}$ the space of sections of a bundle $\mathcal K$ on $M$ and decompose any section $X$ of $\cD^{\mathbb C}$ into the sum $X=X^{10}+X^{01}$ of its holomorphic $X^{10}\in\underline\cD^{10}$ and antiholomorphic $X^{01}\in\underline\cD^{01}$ parts.
\vskip0.15cm\par\noindent
{\it Acknowledgments.} The author is supported by a Marie-Curie
research fellowship of the ``Istituto Nazionale di Alta Matematica'' (Italy). 
Part of this work was done while the author was a
post-doc at the University of Parma and he would like to thank the
Mathematics Department and in particular A. Tomassini and C. Medori for
support and ideal working conditions.
\vskip0.1cm\par\noindent
\section{Main definitions and preliminaries}\hfill\par
\subsection{The sequences of Tanaka and Freeman and the abstract cores}\label{sec:2.1}
\setcounter{section}{2}
\setcounter{equation}{0}
The {\it Tanaka sequence} of a CR manifold $(M, \cD, \cJ)$ (see  \cite{Ta1}) is the nested sequence of $\cC^{\infty}(M)$-modules of real vector fields
$$\cdots  \supset  \under\cD_{p-1} \supset  \under\cD_{p} \supset  \under\cD_{p+1}\supset \cdots \supset   \under \cD_{-3} \supset \under \cD_{-2} \supset \under \cD_{-1} 
$$
iteratively defined for any $p\leq-1$ by 
$$
\begin{array}{l}\under\cD_{-1}:=\under\cD
\qquad \text{and}\qquad \under\cD_{p}:=\under\cD_{p+1}+[\under\cD_{-1},\under\cD_{p+1}]
\ \, .
\end{array}
$$
On the other hand, the  {\it  Freeman sequence} (see \cite{Fr})
is a sequence
$$\under\cF_{-1}\supset \under\cF_{0}\supset \under\cF_{1} \supset \cdots\supset \under\cF_{p-1}\supset\under\cF_{ p}\supset\under \cF_{ p+1}
\supset\cdots$$
of complex vector fields given for any $p\geq -1$ by 
\begin{equation*}
\begin{array}{l}
\under \cF_{p}=\under \cF_{p}^{10}\oplus
\overline{\under \cF_{p}^{10}}\qquad \text{where} 
\qquad  \under{\cF}_{-1}^{10}:= \under{\cD}^{10}\qquad\text{and}
\\ \,\\
\!\!\under{\cF}_{p}^{10}:=\left\{X\in \under{\cF}_{p-1}^{10} :[X, \under{\cD}^{01}] = 0^{\phantom{P^C}}\!\!\!\!\!\! \mod  \under{\cF}_{p-1}^{10} \oplus \under{\cD}^{01} \right\}\;;
\end{array}
\end{equation*}
note that $\under\cF_{-1}=\under\cD_{-1}^\bC=\under\cD^\bC$ whereas  $\under{\cF}_{p}^{10}$
for $p\geq 0$ coincides with the left kernel of the Levi form of higher order
\begin{align*}
\cL^{p+1}:&\;\; \under{\cF}_{p-1}^{10}\times \under{\cD}^{01}\longrightarrow \mathfrak{X}(M)^\bC/ (\under{\cF}_{p-1}^{10}\oplus\under{\cD}^{01})\\
&\;(X,Y)\longrightarrow [X,Y]\!\!\!\mod \under{\cF}_{p-1}^{10}\oplus\under{\cD}^{01}\,.
\end{align*}
Each term $\under\cF_p$ of the Freeman sequence is a $\cC^{\infty}(M)$-module and, if $p\geq 0$,
also a Lie subalgebra of $\mathfrak{X}(M)^\mathbb C$. We say that $(M, \cD, \cJ)$ is {\it  regular} if the vector fields in $\under\cD_{p}$ and  $\under{\cF}_{p}^{10}$ are the sections of corresponding distributions
$\cD_{p}\subset TM$ and $\cF_{p}^{10} \subset \cD^{10}$ and if $\cD_{-\mu}=TM$ for some positive integer $\mu$.
From now on,
{\it any CR manifold  is assumed to be regular}. 
\begin{definition}\cite{BER, KZ}  A CR manifold is {\it  $k$-nondegenerate} if   $\cF_p\! \neq\! 0$ for all $-1 \leq$ $p$ $\leq k-2$  and $\cF_{k - 1} \!\!=\! 0$.
\end{definition}
\begin{definition}
\label{acore}
An {\it abstract core} is a finite dimensional $\bZ$-graded real vector space $\gm=\bigoplus_{p\in\bZ}\gm^{p}$ endowed with 
\begin{itemize}
\item[(i)] a complex structure 
\vskip0.01cm
$$
J:\bigoplus_{p\geq -1}\gm^p\longrightarrow\bigoplus_{p\geq -1}\gm^p
$$
compatible with the grading, i.e., with $J(\gm^p)\subset\gm^p$ for all $p\geq -1$. We denote by $\wh\gm^p=\gm^{p(10)}\oplus\gm^{p(01)}$ the corresponding decomposition into $J$-holomorphic and $J$-antiholomorphic parts of $\wh\gm^p=\gm^p\otimes\bC$ for all $p\geq -1$;
\item[(ii)] a bracket of $\bZ$-graded Lie algebras
\vskip0.01cm
$$
[\cdot,\cdot]:\gm_-\wedge\gm_-\longrightarrow\gm_-\;\;,\qquad \gm_-=\bigoplus_{p<0}\gm^p\;\;,
$$
satisfying $[Jv,Jw]=[v,w]$ for all $v,w\in\gm^{-1}$ and nondegenerate and fundamental in Tanaka's sense;
\item[(iii)] an injective and $\bC$-linear map for any $p\geq 0$ 
\vskip0.15cm\par\noindent
$$
\;\;\;\;\;\;\;\;\;\;L^{p+2}: \gm^{p(10)}  \longrightarrow \gm^{p-1(10)}\otimes (\gm^{-1(01)})^*\bigcap\wh\gm^{-2}\otimes S^{p+2}(\gm^{-1(01)})^*
$$
\vskip0.1cm\par\noindent
where $\gm^{-1(10)}$ is understood as space of maps from $\gm^{-1(01)}$ to  $\wh\gm^{-2}$
using bracket (ii).
\end{itemize}
\vskip0.15cm\par\noindent
The core is of {\it depth} d$(\gm)=\mu$ (resp. {\it height} ht$(\gm)=k-2$) if $\gm^p=0$ for all $p<-\mu$ (resp. all $p>k-2$).
A {\it morphism}  of cores $\gm$ and $\gm'$ is a morphism $\varphi:\gm\longrightarrow \gm'$ of real vector spaces such that
\begin{itemize}
\item[(i)] $\varphi(\gm^p)\subset\gm'^p$ for all $p\in\bZ$;
\item[(ii)] $\varphi|_{\gm^p}\circ J=J'\circ\varphi|_{\gm^p}$ for all $p\geq -1$;
\item[(iii)] $\varphi|_{\gm_-}:\gm_-\longrightarrow \gm'_-$ is a morphism of Lie algebras;
\item[(iv)] $\varphi^*(L'^{p+2}\circ\varphi(v))=\varphi\circ L^{p+2}(v)$ for all $v\in\gm^{p(10)}$, $p\geq 0$, where $\varphi$ is extended by $\bC$-linearity. 
\end{itemize}
It is an {\it immersion} (resp. an {\it isomorphism}) if it is injective (resp. bijective).
\end{definition}
Note that any immersion $\varphi$ with $\varphi(\gm^{-1})=\gm'^{-1}$ is fully determined by its action on $\gm_-$, by property (iv). 
The following result recasts the Tanaka and Freeman sequences, and the higher order Levi forms, in the form suitable for our purposes.
\begin{lemma}
\label{lemmino}
Let $(M,\cD,\cJ)$ be a $k$-nondegenerate CR manifold with $\cD_{-\mu}=TM$. For every $x\in M$, the $\bZ$-graded vector space
$
\gm_x=\bigoplus_{p\in \bZ}\gm_x^{p}
$
defined by
\begin{equation*}
\begin{split}
\gm_x^{p}&=\frac{\cD_{p}|_x}{\cD_{p+1}|_x}\;\qquad\qquad\;\,\text{for all}\; p\leq -2\,,\\
\gm_x^{p}&=\frac{\Re(\cF_p)|_x}{\Re(\cF_{p+1})|_x}\;\;\;\qquad\text{for all}\; p\geq -1\,,
\end{split}
\end{equation*}
has the natural structure of an abstract core of depth $\mu$ and height $k-2$.
\end{lemma}
\begin{proof}
The depth and height follow directly from definitions
while point (i) of Definition \ref{acore} from the fact that 
$\cJ$ preserves the real part $\Re(\cF_p)$ of $\cF_p$ for every $p\geq -1$.
The proof of (ii) 
follows the same lines of \cite[\S 1.2]{Ta1} once we observed that
$
[\underline\cF_0, \underline\cD_{p}]\subset\underline\cD_{p}
$
for every $p\leq -1$. Now for every $p\geq 0$ we have $[\underline\cF_p^{10},\underline\cD^{01}]\subset\underline\cF_{p-1}^{10}\oplus\underline\cD^{01}$ and $[\underline\cF_{p}^{10},\underline\cF_{0}^{01}]\subset \underline\cF_{p}^{10}\oplus\underline\cF_0^{01}$ (see \cite[Appendix]{KZ}) and the higher order Levi form induces  a well-defined bilinear map 
$$
\cL^{p+2}_{x}: \gm_x^{p(10)} \times \gm^{-1(01)}_x \longrightarrow \frac{\cF_{p-1}^{10}|_x\oplus\cD^{01}|_x}{\,\,\,\cF_p^{10}|_x\oplus\cD^{01}|_x}\simeq \gm_x^{p-1(10)}
$$
at any $x\in M$. The corresponding map from
$
\gm_x^{p(10)}$ to $\gm_x^{p-1(10)}\otimes (\gm_x^{-1(01)})^*
$
is injective and $\bC$-linear and, by a direct induction on $p\geq 0$,
takes values in the subspace $\wh{\gm}^{-2}_x\otimes S^{p+2}(\gm_x^{-1(01)})^*$ of $\gm_x^{-1(10)}\bigotimes^{p+1}(\gm_x^{-1(01)})^*$. 
\end{proof}
By Definition \ref{acore} and Lemma \ref{lemmino}, the core is an invariant of a CR manifold which generalizes the usual notion of a fundamental algebra associated with a nondegenerate CR manifold (\cite{Ta1}). In contrast with this case, it does not posses any structure of a Lie algebra and the problem of
constructing $k$-nondegenerate CR manifolds with a given core at all points is more delicate. To do so, we first need to recall the notions of homogeneous CR manifolds and their associated CR algebras.
\smallskip\par
\subsection{Homogeneous CR manifolds and CR algebras}
\label{sec:2.2}
Let $(M,\cD,\cJ)$ be a CR manifold that is locally homogeneous around a point $x\in M$ under a (possibly infinite-dimensional) Lie algebra $\gg$ of infinitesimal CR automorphisms. Transitivity of the action amounts to
 $T_xM=\left\{Z|_x\;|%\!\!\!\!\!\!\!\!\!\!\!\phantom{C^{C^C}}
\;Z\in \gg\right\}$ while
$$
\gq=\left\{Z\in\wh\gg\;\;|\!\!\!\!\!\!\!\!\!\!\phantom{C^{C^C}}Z|_x\in \cD^{10}|_x\right\}
$$
is a complex Lie subalgebra of $\wh\gg=\gg\otimes\bC$, by the integrability condition of CR manifolds.   %\eqref{intcond} 
In the terminology of \cite{MeNa3} the pair $(\gg,\gq)$ is an {\it abstract CR algebra}, i.e., it enjoys:
\begin{itemize}
\item[(i)] $\gg$ is a real Lie algebra,
\item[(ii)] $\gq$ is a complex subalgebra of $\wh{\gg}$,
\item[(iii)] the quotient $\gg/\gg_{o}$ is finite-dimensional, where $$\gg_{o}=\gg\cap\gq=\Re(\gq\cap\overline\gq)=\left\{Z\in\gg\,|\!\!\!\!\!\!\!\!\!\!\!\phantom{C^{C^C}}Z|_x=0\right\}\ $$
is the real isotropy algebra at $x$.
\end{itemize}
Conversely any abstract CR algebra determines a unique (germ of) locally homogeneous CR manifold $(M,\cD,\cJ)$ with 
\begin{equation}
\label{itangente}
\begin{split}
T_xM&\simeq \gg/\gg_{o}\;,
\\
\cD^{10}|_x&\simeq\gq/\gq\cap\bar\gq\;,
\end{split}
\end{equation}
see \cite{MeNa3} and also \cite[\S 4]{Fel} for more details. The associated core at $x\in M$ can also be directly recovered from the CR algebra. First note that the Freeman bundles are locally homogeneous bundles with fiber $\cF_{p}^{10}|_{x}\simeq \gq_{p}/\gq\cap\overline\gq$ where
\begin{align}
\label{bologna}
\notag\gq_{p}&=\left\{Z\in \wh{\gg}\,|\!\!\!\!\!\!\!\!\!\!\!\phantom{C^{C^C}} Z|_x\in \cF_{p}^{10}|_x\right\}\\
&=\left\{Z\in\gq_{p-1}\,|\!\!\!\!\!\!\!\!\!\!\!\phantom{C^{C^C}}[Z,\overline\gq]\subset \gq_{p-1}+\overline\gq\right\}
\end{align}
and $\gq_{-1}=\gq\supset \gq_0\supset \cdots\supset \gq_{p-1}\supset \gq_p\supset \gq_{p+1}\supset\cdots
\supset \gq\cap\overline\gq$ is the associated sequence of complex subalgebras 
of $\gq$ (see \cite{Fel}). The homogeneous CR manifold is $k$-nondegenerate precisely when $\gq_{k-1}=\gq\cap\overline\gq$ and $\gq_{k-2}\neq\gq\cap\overline\gq$. Similarly the fibers $\cD_{p}|_{x}\simeq \gg_{p}/\gg_o$ of the Tanaka bundles correspond   to the sequence
 $
\gg_{-\mu}\supset\cdots\supset\gg_{p-1}\supset\gg_{p}\supset\gg_{p+1}\supset\cdots\supset\gg_{-2}\supset\gg_{-1}
$ of subspaces
\beq
\label{prato}
\gg_{-1}=\Re(\gq+\bar\gq)\;,\qquad\gg_{p}=\gg_{p+1}+[\gg_{-1},\gg_{p+1}]\;.
\eeq
It is not difficult to see that $[\gg_p,\gg_q]\subset\gg_{p+q}$ for all $p,q\leq -1$.
It follows from these observations that
\begin{align}
\label{corecore}
&\gm^{p}_x\simeq\gg_{p}/\gg_{p+1}\qquad\qquad\qquad\text{for all}\; p\leq -2\,,\\
  \begin{split}
  \label{corecore2}
	&\gm_x^p\simeq \Re(\frac{\gq_p+\overline\gq_p}{\gq_{p+1}+\overline\gq_{p+1}})
	\\
    &\gm^{p(10)}_x\simeq\gq_p/\gq_{p+1}
  \end{split}\qquad\,\text{for all}\; p\geq -1\,,
\end{align}
as (real or complex) vector spaces and %$J$ acts with eigenvalue $+i$ (resp. $-i$) on $\gq_p$ (resp. $\overline\gq_p$) and
that the Lie bracket of $\gm_-$ is induced by the Lie bracket of $\gg$. Finally 
$[\gq_p,\overline\gq_0]\subset \gq_p+\overline\gq_0$ 
for all $p\geq 0$ by a direct induction and the Lie bracket of $\gg$ also induces an operation on the quotients
$$[\gm^{p(10)}_x, \gm^{-1(01)}_x]\subset \frac{\gq_{p-1}+\overline\gq}{\;\;\;\,\gq_p+\overline\gq}\simeq \gm^{p-1(10)}_x$$
and, in turn, the required immersion $L^{p+2}$ of Definition \ref{acore}. %\eqref{prolstrange}.

In view of \S \ref{sinizia}  it is convenient to relax  the definition of a CR algebra, including those pairs which satisfy (i) and (ii), but not necessarily (iii). Informally speaking we are considering infinite-dimensional locally homogeneous CR manifolds but we will not attempt to rigorously define  such objects. Note however that the r.h.s. of \eqref{itangente} and \eqref{bologna}-\eqref{corecore2} still make sense and  
%whereas the Lie bracket in $\gg$ induces well-defined corresponding immersions
%$$
%L^{p+2}:\gq_p/\gq_{p+1}\rightarrow \gq_{p-1}/\gq_{p}\otimes (\bar\gq/\bar\gq_{0})^*\cap (\gg_{-2}/\gg_{-1})^\bC\otimes S^{p+2}(\bar\gq/\bar\gq_{0})^*\ .
%$$
we may say that $(\gg,\gq)$ is holomorphically nondegenerate when
$
\bigcap_{p\geq -1} \gq_{p}=\gq\cap\overline\gq
$ (this corresponds to usual $k$-nondegeneracy of some order $k$ whenever $\dim(\gg/\gg_o)<+\infty$). 

In \S\ref{univ} and \S\ref{cisiamo} we consider the opposite problem of associating a CR algebra
to an abstract core. {\it We restrict to CR manifolds} $(M,\cD,\cJ)$ {\it of hypersurface type and therefore to cores $\gm$ of depth d$(\gm)=2$ and $\dim(\gm^{-2})=1$.}
We first describe in \S\ref{univ} a holomorphically nondegenerate CR algebra $(\gc,\gu)$ %(but not $k$-nondegenerate, for any $k$)
that is universal, in the sense that
any abstract core $\gm$ has a natural immersion $\varphi:\gm\longrightarrow\mathfrak{M}$ into the core $\mathfrak{M}$ of $(\gc,\gu)$, see Proposition \ref{autoval}. 
This universal property will be crucial to constructing appropriate Lie subalgebras of $(\gc,\gu)$ and their associated homogeneous CR manifolds.
\section{The universal CR algebra}\label{univ}
\label{sinizia}
\setcounter{section}{3}
\setcounter{equation}{0}
The universal CR algebra $(\gc,\gu)$ is given by the real infinite-dimensional contact algebra $\gc$ and an appropriate complex subalgebra $\gu$ of its complexification;   
we first study the structure of contact algebras over $\bK=\bC$, $\bR$ simultaneously.
%\subsection{The infinite-dimensional contact algebra $\gc$ over $\mathbb K=\bC$ and $\bR$}
%\subsection{The universal holomorphically nondegenerate CR algebra}
%\hfill\vskip0.1cm\par
Let $\gc_-=\gc^{-2}\oplus\gc^{-1}$ be the Heisenberg algebra of degree $n$, the fundamental $\bZ$-graded Lie algebra over $\mathbb K$ such that:
\begin{itemize}
\item[(i)] $\dim \gc^{-2}=1$,
\item[(ii)] $\dim\gc^{-1}=2n$,
\item[(iii)] $\gc_-$ is nondegenerate (if $v\in\gc^{-1}$ satisfies $[v,\gc^{-1}]=0$ then $v=0$).
\end{itemize} 
It is well-known that the maximal transitive prolongation of $\gc_-$ is an infinite-dimensional simple Lie algebra with the grading \
\beq
\label{contatto}
\gc=\bigoplus_{p\geq -2}\gc^{p}\,,
\eeq usually referred to as the {\it contact algebra} of degree $n$ \cite{MT}. We identify $\gc^{-2}$ with the ground field using the isomorphism $\pi_{\gc^{-2}}:\gc^{-2}\longrightarrow \bK$ associated with a basis $\{e^{-2}\}$ of $\gc^{-2}$ and denote by
$B:\gc^{-1}\wedge \gc^{-1}\longrightarrow\bK$ the symplectic form
\beq
\label{simple}
[v,w]=B(v,w)e^{-2}\;,\qquad v,w\in\gc^{-1}\;.
\eeq
For all $p\geq -2$, we denote the subspace of $\gc^p$ which acts trivially on $\gc^{-2}$ by
$$\gk^p=\left\{X\in\gc^p\,|\phantom{C^{C^C}}\!\!\!\!\!\!\!\!\!\![X,\gc^{-2}]=0\right\}$$ 
and note that $\gk^{-2}=\gc^{-2}$, $\gk^{-1}=\gc^{-1}$ and
$$
[\gk^p,\gc^{-1}]\subset \gk^{p-1}\;,\qquad[\gk^{p},\gk^{q}]\subset\gk^{p+q}\;,
$$
for all $p,q \geq -1$, since \eqref{contatto} is a $\mathbb Z$-graded Lie algebra. 

Let $E$ be the element of $\gc^0$ which satisfies
$[E,X]=pX$ for every $X\in\gc^p$, it is called the {\it grading element}. The $0$-degree part of the contact algebra has a direct sum decomposition
$\gc^0=\gk^0\oplus\bK E$, with $\gk^0=\sp(\gc^{-1})=\sp(\gc^{-1},B)$. We consider the usual identification $\sp(\gc^{-1})\simeq S^2(\gc^{-1})$ where 
%given by
%$$
%v_1\odot v_2\mapsto v_1\otimes B(v_2,\cdot)+v_2\otimes B(v_1,\cdot)\quad\text{for any}\;\; v_1,v_2\in\gc^{-1}\,,
%$$
%Usiamo convenzione che v\odot w=v\otimes w+w\otimes v e simile per i prodotti simmetrici più alti (otteniamo n! termini). Inoltre c^{-1} e c^{-1*} sono identificati con w\mapsto B(w,\cdot) %In questo senso i prolungamenti è meglio scriverli nella forma V\otimes S^iV^* e non S^iV^*\otimes V
the bracket of an element $v_1\odot v_2\in \gk^0$ with $w\in\gc^{-1}$ takes the form
\beq
\label{identificationsymmetric}
[v_1\odot v_2, w]=v_1 B(v_{2},w)+v_2 B(v_1,w)\ .
\eeq
%Setting $[\gk^{-1},\gk^{-1}]=0$, one can easily check that 
%\beq
%\label{serve}
%\bigoplus_{p\geq -1}\gk^{p}
%\eeq
%is a $\bZ$-graded Lie algebra and that $\gk^{p}$ for $p\geq 0$ is naturally isomorphic to the $p$-th Cartan prolongation of $\gk^0$. 
%Questo corrisponde a scegliere il bracket operatoriale e non usuale sullo spazio dei campi vettoriali, cioè ad usare [Z,U]=-ZU+UZ.
It is proved in \cite{MT, SS} that there exist similar identifications  $\gk^{p}\simeq S^{p+2}(\gc^{-1})$ which are $\gk^0$-equivariant and such that the bracket of
$
X=v_1\odot \cdots \odot v_{p+2}\in\gk^{p}$ with $Y=w_1\odot\cdots\odot w_{q+2}\in\gk^{q}
$ is
$$
[X,Y]=\sum_{i=1}^{p+2}\sum_{j=1}^{q+2}v_1\odot\cdots\odot \wh{v}_i\odot \cdots \odot v_{p+2}\odot w_1\odot\cdots \odot\wh{w}_{j}\odot\cdots \odot w_{q+2} B(v_i,w_j)\,,
$$
where $\wh{v}$ indicates that the vector $v\in\gc^{-1}$ is omitted. The following result was also proved by Morimoto and Tanaka.
\begin{proposition}\cite{MT}
\label{MoriTan}
Let $e^{-2}$ be a basis of $\gc^{-2}$. Then $\ad(e^{-2})$ is a surjective and $\gk^0$-equivariant map from $\gc^{p}$ to $\gc^{p-2}$ with kernel $\gk^p$, for every $p\geq 0$. Moreover there exists a unique $\gk^0$-module $\xi^p$ which is isomorphic to $\gc^{p-2}$ and complementary to $\gk^p$ in $\gc^p$.
\end{proposition}
Their proof relied on the existence of appropriate, but not necessarily $\gk^0$-equivariant, immersions $\mu^p:\gc^{p-2}\longrightarrow\gc^p$. We now give an improved version of this result and an explicit description of the decomposition $\gc^p=\gk^p\oplus\xi^{p}$. 
More precisely we give maps $\mu^p$ which are $\gk^0$-equivariant, together with conditions which guarantee their uniqueness, and explicitly describe the Lie brackets between the different irreducible $\gk^0$-submodules of $\gc$ (this description will be relevant only in \S \ref{exexamplesII} and it can be skipped on a first reading). To formulate the result, we set 
\begin{equation*}
%\label{simba2000}
\begin{split}
\mu^{-1}=0\;\;,\qquad\mu^{p|p+2}=\Id_{\gk^p}\qquad\text{and, for any}\;0\leq i\leq [p/2]\,,\\
\mu^{p|p-2i}=\mu^{p}\circ\mu^{p-2}\circ\cdots\circ\mu^{p-2i}|_{\gk^{p-2i-2}}:\gk^{p-2i-2}\longrightarrow \gc^{p}\;.
\end{split}
\end{equation*}
We adopt the convention that the binomial coefficient $\binom{-1}{0}=1$ while $\binom{k-1}{k}$ is trivial for any integer $k\geq 1$.
\begin{proposition}
\label{noioso}
%\vskip0.2cm\par\noindent
There exists a $\gk^0$-equivariant immersion
$
%\label{mucca}
\mu^p:\gc^{p-2}\longrightarrow\gc^p
$ with image $\xi^p=\Im(\mu^p)$
and a decomposition $\gc^{p}=\gk^p\oplus\xi^p$ of $\gk^0$-modules for all $p\geq 0$. 
%\beq
%\label{rotten}
%\;.
%\eeq
If
$
\pi_{\gk^p}:\gc^p\longrightarrow\gk^p$ and $\pi_{\xi^p}:\gc^p\longrightarrow\xi^p
$
are the corresponding projection operators, there is a unique choice of $\mu^p$'s such that
the following conditions are satisfied
$$
\begin{dcases}
\mu^{p}(X)e^{-2}=X\,\;\;\;\;\;\;\;\;\;\;\;\;\;\;\;\;\qquad\qquad\qquad\qquad\qquad\qquad\qquad\qquad\qquad\,(\rm{C}1)\\
\mu^{p}(X)v= \mu^{p-1}[X,v]+\frac{1}{2}\pi_{\gk^{p-2}}(X)\odot v\in\xi^{p-1}\oplus\gk^{p-1} \;\text{for all}\;v\in\gc^{-1}\;\;(\rm{C}2)
\end{dcases}
$$
for all $X\in\gc^{p-2}$. In this case each map $\mu^{p|p-2i}$ %defined in \eqref{simba2000} 
is a $\gk^0$-equivariant immersion and
$\gc^p$ decomposes into irreducible inequivalent $\gk^0$-modules as follows:
\begin{equation}
\label{gradop}
\begin{split}
\gc^p\simeq\!\!\!\!\!\! \bigoplus_{-1\leq i\leq [p/2]}\!\!\!\!\!\! S^{p-2i}(\gc^{-1})\;\;\;,\qquad S^{p-2i}(\gc^{-1})\simeq\Im(\mu^{p|p-2i})\qquad{with}\\
\!\!\!\!\gk^p\simeq S^{p+2}(\gc^{-1})\qquad\text{and}\qquad\xi^p\simeq\!\!\!\!\!  \bigoplus_{0\leq i\leq [p/2]}\!\!\!\!\!S^{p-2i}(\gc^{-1})\,.
\end{split}
\end{equation}
Furthermore:
\vskip0.2cm\par
(i) The bracket of $X\in\gk^p$  with $\mu^{q|q-2j}(Y)\in\xi^q$ is in $\gc^{p+q}$ and
%$$
%X=v_1\odot\cdots\odot v_{p+2}\in \gk^p\phantom{ccccccccccccccccccccccccccccci}
%$$
%and
%$$
%Y=\mu^{q,q-2j}(w_1\odot\cdots\odot w_{q-2j})\in S^{q-2j}(q)\subset \xi^q\;\;\;\;\;(0\leq j\leq[q/2])
%$$
\vskip0.4cm\par\noindent
\beq
\label{cash}
[X,\mu^{q|q-2j}(Y)]=\underbrace{\frac{p}{2}\mu^{p+q|p+q+2-2j}(X\odot Y)}_{\text{elem. of}\; S^{p+q+2-2j}(\gc^{-1})}+\underbrace{\mu^{p+q|p+q-2j}[X,Y]}_{\text{elem. of}\; S^{p+q-2j}(\gc^{-1})}
\eeq
\vskip0.4cm\par\noindent
where $p\geq -1$, $q\geq 0$ and $0\leq j\leq[q/2]$.
\vskip0.3cm\par
$(ii)$ The bracket of $\mu^{p|p-2i}(X)\in \xi^p$ and $\mu^{q|q-2j}(Y)\in\xi^q$ is in $\xi^{p+q}$ and
\vskip0.4cm\par\noindent
\beq
\label{june}
\begin{split}
[\mu^{p|p-2i}(X),\mu^{q|q-2j}(Y)]=& \alpha(p,i;q,j)\underbrace{\mu^{p+q|p+q-2i-2j}(X\odot Y)}_{\text{elem. of}\;S^{p+q-2i-2j}(\gc^{-1})}\\
&\;\;+\beta(i;j)\underbrace{\mu^{p+q|p+q-2i-2j-2}[X,Y]}^{\phantom{C^{C^C}}}_{\text{elem. of}\;S^{p+q-2i-2j-2}(\gc^{-1})}\,,
\end{split}
\eeq
\vskip0.4cm\par\noindent
where $p,q\geq 0$, $0\leq i\leq [p/2]$, $0\leq j\leq [q/2]$ and
the coefficients $\alpha$ and $\beta$ are respectively given by
$$
\alpha(p,i;q,j)=\frac{p-2i-2}{2}\{\sum_{k=0}^{j}\binom{i+k-1}{k} (j+1-k)\}
$$
$$
\phantom{cccccccccci}-\frac{q-2j-2}{2}\{\sum_{k=0}^{i}\binom{j+k-1}{k} (i+1-k)\}
$$
and
$$
\beta(i;j)=\sum_{k=0}^{j}\binom{i+k-1}{k} (j+1-k)+\sum_{k=0}^{i}\binom{j+k-1}{k} (i+1-k)\ .
$$
%while $\alpha(p,0;q,j)=\frac{p-2}{2}(j+1)-\frac{q-2j-2}{2}$ and $\beta(0,j)=j+2$ for any $j\geq 0$ 
\end{proposition}
\begin{proof}
It is a straightforward but rather long modification of original arguments of Morimoto
and Tanaka in \cite{MT} combined with repeated induction arguments. The details are omitted for the sake of brevity.
%Postponed to Appendix \ref{appendice}. 
\end{proof}
%\bigskip\par
From now on we will always denote by  \eqref{gradop} the decomposition of $\gc$ determined by the $\gk^0$-equivariant maps $\mu^{p|p-2i}$ described in Proposition \ref{noioso}. Unless there is a risk of confusion, we omit the immersion $\mu^{p|p-2i}$ and tacitly identify its image $\Im(\mu^{p|p-2i})$ in $\gc^p$ with %corresponding irreducible $\gk^0$-module 
$S^{p-2i}(\gc^{-1})$. 
%Having in mind that one of our main goals is to determine appropriate subalgebras of contact algebras,

Let now $\gc$ be the {\it real} contact algebra and $\wh\gc=\gc\otimes\bC$. We say that a complex structure $\mathfrak{J}$ on $\gc^{-1}$ has signature $\operatorname{sgn}(\mathfrak{J})=(r,s)$, $n=r+s$, if $\mathfrak{J}\in\gk^0=\sp(\gc^{-1})$ and 
the associated Hermitian product has signature $(r,s)$. We fix $\mathfrak{J}$ once and for all
and identify the space of all complex structures of signature $(r,s)$ with $\Sp_{2n}(\bR)/\mathrm{U}(r,s)$.
%$g(v,w)=B(\mathfrak{J}v,w)$ 
% in definite signature they are parameterized by Sp_2n /U(n)
Clearly $\hat{\gc}^{-1}$ decomposes into $\hat{\gc}^{-1}=\gc^{-1(10)}\oplus\gc^{-1(01)}$ with 
\begin{align*}
\gc^{-1(10)}&=\left\{\phantom{C^{C^C}}\!\!\!\!\!\!\!\!\!\!\!\! v-i\gJ v\,|\, v\in\gc^{-1}\right\}\;,\\
\gc^{-1(01)}&=\overline{\gc^{-1(01)}}=\left\{\phantom{C^{C^C}}\!\!\!\!\!\!\!\!\!\!\!\! v+i\gJ v\,|\, v\in\gc^{-1}\right\}\;.
\end{align*} 
%(the adjoint action of $\wh\gk^0$ on each irreducible component in \eqref{gradop} is the natural extension of the tautological representation of $\wh\gk^0$ to the symmetric algebra of $\wh\gc^{-1}$).
Similarly we have the following.
\begin{definition}
We denote by $S^{\ell,p-2j-\ell}=S^{\ell}(\gc^{-1(10)})\otimes S^{p-2j-\ell}(\gc^{-1(01)})$ the $\ad(\mathfrak J)$-eigenspace in $S^{p-2j}(\hat\gc^{-1})$ given by $\left.\ad(\gJ)\phantom{\frac{}{\partial}}\!\!\!\!\!\right|_{S^{\ell,p-2j-\ell}}=i(2\ell-p+2j)\Id$.
%\beq
%\label{cenaboh}
%\begin{split}
%\left.\ad(\gJ)\phantom{\frac{}{\partial}}\!\!\!\!\!\right|_{S^{\ell,p-2j-\ell}}=i(2\ell-p+2j)\Id\;.
%\end{split}
%\eeq
\end{definition}
It follows that each component $S^{p-2j}(\hat\gc^{-1})$ of $\hat\gc^p$ decomposes into
\beq
\label{cenaboh}
S^{p-2j}(\hat\gc^{-1})=\bigoplus_{0\leq\ell\leq p-2j}S^{\ell,p-2j-\ell}
%S^{p-2j}(2\ell-p+2j)
\eeq
and there is also a decomposition of the whole $\wh\gc$ into $\ad(\mathfrak{J})$-eigenspaces, with purely imaginary eigenvalues. We are now ready to describe the universal CR algebra $(\gc,\gu)$.
\begin{proposition}
\label{autoval}\label{bed}
%Since a copy $S^{r}(\gc^{-1})$ of the $r$-th symmetric power of $\gc^{-1}$ is present in several components of the contact algebra, we will indicate the copy contained in $\gc^k$ (if there is any) with the symbol $S^{r}(k)$, omitting the reference to the vector space $\gc^{-1}$. 
Let $\displaystyle\mathfrak{u}=\bigoplus_{p\geq -1}\gu^{p}$ be the $\bZ$-graded subspace of $\wh\gc$ with component
\beq
\label{subsub}%S^{p+2}(2\ell-p-2)
\gu^p=\bigoplus_{1\leq \ell\leq p+2} S^{\ell,p+2-\ell}\oplus\bigoplus_{0\leq j\leq [p/2]} S^{p-2j}(\hat\gc^{-1})%\subset\hat\gc^p
\eeq
\vskip0.1cm\par\noindent
given by the direct sum of all $\ad(\gJ)$-eigenspaces in $\hat\gc^p$ of non-minimal eigenvalue \!$-i(p+2)$. Then $\gu$ is a Lie subalgebra of $\wh\gc$ and $(\gc,\gu)$ a holomorphically nondegenerate CR algebra. Moreover:
\begin{itemize}
\item[(i)] the real isotropy algebra  is nonnegatively $\bZ$-graded $\displaystyle\gc_{o}=\bigoplus_{p\geq 0}\gc_{o}^{p}$ with component
\begin{align*}
\gc_{o}^p=\Re(&\gu^p\cap\bar\gu^p)\;,\qquad\text{where}\\
&\gu^p\cap\bar\gu^p=\bigoplus_{1\leq \ell\leq p+1} S^{\ell,p+2-\ell}\oplus\bigoplus_{0\leq j\leq [p/2]} S^{p-2j}(\hat\gc^{-1})%\subset\hat\gc^p
\end{align*}
is the direct sum of all $\ad(\mathfrak J)$-eigenspaces in $\hat\gc^p$ of non-extremal eigenvalues $\pm i(p+2)$;
\item[(ii)] the components of the core $\displaystyle\mathfrak{M}=\bigoplus_{p\geq -2}\mathfrak{M}^p$ of $(\gc,\gu)$ are given by
\begin{equation*}
%\label{eccoci}
\mathfrak{M}^{p}=\begin{dcases}
\gc^{p}\quad\;\;\;\;\;\;\;\;\;\;\;\;\;\;\;\;\;\;\qquad\quad\text{for}\;\;p=-2,-1\,,
%\,\qquad\qquad\qquad\qquad\qquad,
\\
\Re(\mathfrak{M}^{p(10)}\oplus\mathfrak{M}^{p(01)})\quad\text{for all}\;\;p\geq 0\,,
\end{dcases}
\end{equation*}
where $\mathfrak{M}^{p(10)}=S^{p+2,0}$ and 
$\mathfrak{M}^{p(01)}=\overline{\mathfrak{M}^{p(10)}}=S^{0,p+2}$
%\vskip0.3cm\par\noindent
%\begin{align*}
%\mathfrak{M}^{p(10)}&=S^{p+2,0}\\
%\mathfrak{M}^{p(01)}&=\overline{\mathfrak{M}^{p(10)}}=S^{0,p+2}
%\end{align*}
are the $\ad(\mathfrak J)$-eigenspaces in $\hat\gc^p$ of extremal eigenvalues $\pm i(p+2)$;
\item[(iii)] any abstract core $\displaystyle\gm$ of hypersurface type admits a natural
immersion $\varphi:\gm\longrightarrow\mathfrak{M}$.
\end{itemize} 
\end{proposition}
\begin{proof}
By definition \eqref{subsub}, decomposition \eqref{cenaboh} and $\gk^0$-equivariance, it follows $[\gu^p,\gu^q]\subset\gu^{p+q}$ and $(\gc,\gu)$ is a CR algebra. 
The claim on $\gc_o$ is also clear as $\gc_o=\Re(\gu\cap\overline\gu)$ by definition, conjugation in $\wh\gc$ with respect to the real form $\gc$ is compatible with the grading and
$\gu^{-1}\cap\bar\gu^{-1}=0$. Now a direct induction shows that the $p$-th term \eqref{bologna} of the Freeman sequence of $(\gc,\gu)$  is given by
\beq
\label{dreamtheater}
\gu_p=\bigoplus_{0\leq j\leq p-1}(\gu^j\cap\overline\gu^j)\oplus\;\;\bigoplus_{j\geq p }\gu^{j}
\eeq
so that $\displaystyle\bigcap_{p\geq -1}\gu_p=\gu\cap\overline\gu$ and $(\gc,\gu)$ is holomorphically nondegenerate. 

Claim (ii) follows from \eqref{dreamtheater}, the decomposition
$
\wh\gc=(\gu\cap\overline\gu)\bigoplus\wh{\mathfrak{M}}
$, the fact that $\mathfrak{M}^{p(10)}$ can be iteratively characterized by
$[\mathfrak{M}^{p(10)},\gc^{-1(01)}]\subset \mathfrak{M}^{p-1(10)}$ 
for any $p\geq 0$  and, finally, the identifications \eqref{corecore}-\eqref{corecore2}.
We also note that in this case the immersion
\begin{equation*}
\begin{split}
%\label{full}
\mathfrak{L}^{p+2}:\;&\mathfrak{M}^{p(10)}\longrightarrow\mathfrak{M}^{p-1(10)}\otimes (\mathfrak{M}^{-1(01)})^*\cap\wh{\mathfrak{M}}^{-2}\otimes S^{p+2}(\mathfrak{M}^{-1(01)})^* 
\end{split}
\end{equation*}
of (iii) of Definition \ref{acore} is  surjective and hence an isomorphism, for all $p\geq 0$. To prove (iii) fix an abstract core $\displaystyle\gm=\bigoplus\gm^p$
and an identification $\varphi_-:\gm_-\longrightarrow\gc_-$ of $\bZ$-graded Lie algebras  with $\varphi_-|_{\gm^{-1}}\circ J=\mathfrak{J}\circ\varphi_-|_{\gm^{-1}}$. By %\eqref{full}, 
the observation after Definition \ref{acore} and an induction argument
we get a unique morphism of cores $\varphi:\gm\longrightarrow \mathfrak{M}$ that is injective and
satisfies
\begin{multline*}
\varphi|_{\gm_-}:=\varphi_-\;,\qquad\varphi|_{\gm^{p(10)}}:=(\mathfrak{L}^{p+2})^{-1}\circ \varphi|_{\gm_{<p}}\circ L^{p+2}\qquad\text{and}\\
\varphi|_{\gm^{p-1(10)}}\circ J=(\frac{1}{p+1}\mathfrak{J})\circ\varphi|_{\gm^{p-1(10)}}\;,
\end{multline*}
where $\displaystyle\gm_{<p}=\bigoplus_{r<p}\gm^{r}$ for all $p\geq 0$. This concludes the proof.
%
%By $(iii)$ of Definition \ref{acore} and a 
%\beq
%\label{prolstrangeII}
%L^{p+2}(v)\in \gm^{p-1(10)}\otimes (\gm^{-1(01)})^*\cap\wh\gm^{-2}\otimes S^{p+2}(\gm^{-1(01)})^* 
%\eeq
%for any $v\in\gm^{p(10)}$
%\beq
%\label{otvol}
%L^{2}(f)\in\gc^{-1(10)}\otimes(\gc^{-1(01)})^*\ .
%\eeq
%\begin{itemize}
%\item[(ii)] $\varphi^*(L'^{p+2}\circ\varphi(v))=\varphi\circ L^{p+2}(v)$ for all $v\in\gm^{p(10)}$, $p\geq 0$, where $\varphi$ is extended by $\bC$-linearity. 
%\end{itemize}
%The proof is by induction on $p\geq 0$. 
%Now, extend this action to the whole $\wh\gc^{<0}=\wh\gc^{-2}\oplus\wh\gc^{-1}$ by trivially setting
%$$
%L^{2}(f)(\gc^{-1(10)})=0\quad\text{and}\quad
%L^{2}(f)(\wh\gc^{-2})=0\ .
%$$
%Since \eqref{otvol} takes its values in $\wh\gc^{-2}\otimes S^2(\gc^{-1(01)})^*$, it turns out that the above extension is a $0$-degree derivation of $\wh\gc^{<0}$, and thus an element of 
%$$\wh\gk^0=S^{2}(\wh\gc^{-1})=\mathfrak{M}^{0(10)}\oplus S^{2}(0)\oplus \mathfrak{M}^{0(01)}\ .$$
%It is clear that $L^{2}(f)\in\mathfrak{M}^{0(10)}$. 
%To apply the induction procedure to an element $f\in\gm^{p(10)}$ with $p>0$, one proceeds similarly and trivially extend to $\wh\gc^{<0}=\wh\gc^{-2}\oplus\wh\gc^{-1}$ each 
%$$L^{p+2}(f)\in  \gm^{p-1(10)}\otimes(\gc^{-1(01)})^*\ .$$ 
%one gets that $L^{p+2}(f)\in \mathfrak{M}^{p(10)}$ and that equation \eqref{brachetto} holds as well. 
%is then a simple reformulation of how the image of the map in $(d)$ of Lemma \ref{lemmino} is described.
\end{proof}
%\smallskip\par
%By Proposition \ref{bed} 
We have just seen that any abstract core $\displaystyle\gm=\bigoplus\gm^p$
is, up to isomorphism, always realized as a subspace of $\mathfrak M$ with components
\beq
\label{brachetto}
\gm^{p}=
\begin{dcases}
0\;\;\,\,\;\;\;\;\;\;\;\;\;\;\;\;\,\,\,\,\qquad\qquad\qquad\qquad\qquad\text{for all}\;\;p<-2\,,\\
\gc^{p}\,\;\;\;\;\;\;\;\;\;\;\;\;\;\;\;\;\qquad\qquad\qquad\qquad\qquad\text{for}\;\;p=-2,-1\,,\\
\Re(\gm^{p(10)}\oplus\overline{\gm^{p(10)}})\,\;\;\;\;\;\;\;\;\;\;\;\;\;\;\;\;\;\,\qquad\text{for all}\;\;p\geq 0\;,
\end{dcases}
\eeq
where $\gm^{p(10)}\subset\mathfrak{M}^{p(10)}$ satisfies $[\gm^{p(10)},\gc^{-1(01)}]\subset \gm^{p-1(10)}$
for all $p\geq 1$. In particular
\beq
\label{dimensione}
\dim_{\bC}\gm^{p(10)}\leq 
%\dim_{\bC} \mathfrak{M}^{p(10)}=\dim_{\bC} S^{p+2}(p+2)=
\binom{n+p+1}{p+2}
\eeq
and two cores $\gm$, $\gm'$ of the same signature $\operatorname{sgn}(\mathfrak J)=(r,s)$ are isomorphic if and only if $\gm'=g\cdot \gm$ for some element $g\in \Aut(\gc,\mathfrak{J})$  
of the group 
\begin{align*}
\Aut(\gc,\mathfrak{J})&=\left\{
g:\gc\longrightarrow\gc
\;\;\text{automorphism of}\;\;\bZ\text{-graded Lie algebras}
\!\!\!\!\!\!\!\!\!\!\!\!\phantom{C^{C^C}}\right.
\\
&\phantom{cccccccccccccccc}
\left.\text{such that}\;\;
g\circ\ad(\mathfrak J)=\ad(\mathfrak J)\circ g
\!\!\!\!\!\!\!\!\!\!\!\!\phantom{C^{C^C}}\right\}\simeq\mathrm{CU}(r,s)\;.
\end{align*}
We have used here that the Lie algebra $\mathfrak{cu}(r,s)=\bR\oplus\mathfrak{u}(r,s)$ of $\Aut(\gc,\mathfrak{J})$ is naturally 
%$Lie(\Aut(\gc,\mathfrak{J}))$ 
identifiable with the subalgebra $\bR E\oplus\Re(S^{1,1})$ of $\gc^0$ and
that any $0$-degree automorphism of a fundamental Lie algebra can be canonically prolonged to a unique automorphism of the maximal
transitive prolongation.
%-\Id belongs to U(2) of course.
%$[J(\cdot),\cdot]:S^2\gm^{-1}\rightarrow \gm^{-2}$.
\begin{definition}\label{Knondeg}
A CR manifold $(M,\cD,\cJ)$ is called {\it strongly regular of type $\gm$} if the associated cores $\gm_x\simeq \gm$ at all points $x\in M$. 
\end{definition}
The main aim of \S\ref{cisiamo} %\S\ref{exexamples} 
is to constructing (strongly regular) homogeneous CR manifolds of a given type $\gm$. 
\vskip0.3cm
\par
\section{Homogeneous models for \texorpdfstring{$k$}{k}-nondegenerate CR manifolds }
\label{cisiamo}
\setcounter{section}{4}
\setcounter{equation}{0}
\setcounter{table}{0}
\medskip\par
\subsection{Main definitions, results and first examples}
\label{sec:4.1}
%We now introduce some subalgebras $(\gg,\gq)$ of the universal CR algebra $(\gc,\gu)$ and then consider their corresponding homogeneous CR manifolds $M=G/G_o$.
We recall that any abstract core $\gm=\bigoplus_{p\in\bZ}\gm^{p}$ is 
identified with a finite dimensional subspace \eqref{brachetto} of the universal CR algebra $(\gc,\gu)$. 
%and, in particular,
%fully determined by the components $\gm^{p(10)}$, $p\geq 0$,
\begin{definition}
\label{ukip}
A {\it model} of type $\gm$
is the datum of a $\bZ$-graded Lie subalgebra 
$$
%\beq
%\label{hommod}
\gg=\bigoplus_{p\in\bZ}\gg^{p}
$$ 
of $\gc$ satisfying
\begin{itemize}
\item[(i)] $\gg^{p}=\gc^{p}$ for all $p<0$;
\item[(ii)] the grading element $E$ belongs to $\gg^{0}$;
\item[(iii)] $\wh\gg^p=(\wh\gg^p\cap\gu^p)+(\wh\gg^p\cap\bar\gu^p)$ for all $p\geq 0$;
\item[(iv)] the natural  $\Aut(\gc,\mathfrak{J})$-equivariant projection %$\ad(\mathfrak{J})$- 
$\pi_{\mathfrak{M}^{p(10)}}:\wh\gc^p\longrightarrow \mathfrak{M}^{p(10)}
$ of $\wh\gc^p$ onto the $\ad(\mathfrak{J})$-eigenspace
 %$\mathfrak{M}^{p(10)}$
of maximum eigenvalue 
satisfies  $$\pi_{\mathfrak{M}^{p(10)}}(\wh\gg^p\cap\gu^p)=\mathfrak{m}^{p(10)}$$
for all $p\geq 0$. 
\end {itemize}
Two models $\gg$ and $\gg'$ of  type  $\gm$ and respectively $\gm'$  are called {\it isomorphic} if $\gg'=g\cdot\gg$ for some $g\in \Aut(\gc,\mathfrak{J})$; if this is the case then $\gm'=g\cdot\gm$ too. A model not contained in any other model is called {\it maximal}. 
\end{definition}
We emphasize that we do not require $\mathfrak J\in\gg^0$ in general. We say that $\gg$ has the {\it property $(\mathfrak J)$} when $\mathfrak J\in\gg^0$ and note that in this special case $\wh\gg$ decomposes into $\ad(\mathfrak J)$-eigenspaces, with the abstract core $\gm$ actually contained in $\gg$.
%by (iii)-(iv) of Definition \ref{ukip}.
This property is a direct generalization of the  ``property (J)'' introduced and studied in \cite{MeNa1} for Levi-Tanaka algebras, see also \cite{LN}. 
We anticipate that we will encounter in Theorem \ref{thmstrange} a model that does not satisfy property $(\mathfrak J)$: this fact says that the core $\gm$ \emph{cannot} always be realized simply as a subspace of $\gg$ and it is also the main reason that motivated weaker property (iv) in Definition \ref{ukip}.

The following is the first main result on homogeneous CR manifolds. We remark that  a CR subalgebra $(\gg,\gq)$ of a holomorphically nondegenerate CR algebra $(\gc,\gu)$ is not necessarily holomorphically nondegenerate (see \cite{MeNa3}), in our setting this stems from
(iii) and (iv), Definition \ref{ukip}. 
%$\wh\gg^p=\gq^p+\overline\gq^p$, $p\geq 0$, 
%for the associated CR algebra $(\gg,\gq)$
%and 
\begin{theorem}
\label{primoteorema}
Let $\displaystyle\gg=\bigoplus\gg^{p}$ be a model of type $\gm$. Then:
\begin{itemize}
\item[(i)] $\gq=\wh\gg\cap \gu$ is a complex Lie subalgebra of $\wh\gg$
with the compatible grading $\gq=\bigoplus_{p\geq -1}\gq^p$ with components $\gq^p=\wh\gg^{p}\cap\gu^{p}$;
\item[(ii)] the pair $(\gg,\gq)$ is a CR algebra with real isotropy algebra $\gg_o=\gg\cap\gu$
that is  $\bZ$-graded in nonnegative degrees;
\item[(iii)]  %If $\gm^{p}=0$ for any $p>k-2$,
$\gg$ is finite dimensional and the associated (germ of) locally homogeneous CR manifold $M=G/G_o$ with algebra of infinitesimal CR automorphisms $(\gg,\gq)$
%$M=G/G_o$ with the Lie algebras $Lie(G)=\gg$ and $Lie(G_o)=\gg_o$ 
is of hypersurface type and strongly regular of type $\gm$
(in particular it is $k$-nondegenerate, $k=\operatorname{ht}(\gm)+2$). 
\end{itemize}
\end{theorem}
\begin{proof}
Point (i) follows immediately as $\gq$ is the intersection of two $\bZ$-graded Lie subalgebras of $\wh\gc$ and therefore a $\bZ$-graded Lie subalgebra of $\wh\gc$.
The real isotropy algebra of the pair $(\gg,\gq)$
is $\bZ$-graded in nonnegative degrees,
$$
\gg_o=\gg\cap\gq=\gg\cap\gu=\bigoplus_{p\geq 0}\gg^p\cap\gu^p\,,
$$
and $\dim(\gg/\gg_o)<+\infty$ by (iii)-(iv) of Definition \ref{ukip} and since $\gm$ is finite dimensional by definition. This shows that $(\gg,\gq)$ is a CR algebra.
%the height $\operatorname{ht}(\gm)$ of $\gm$ is finite by definition. 

To prove (iii) we first show that the $q$-th term, $q\geq 0$, of the Freeman sequence of $(\gg,\gq)$ is given by
\beq
\label{chebarba}
\gq_q=\bigoplus_{p\geq q}\gq^p\oplus\bigoplus_{0\leq p\leq q-1}\gq^p\cap\overline\gq^p\;.
\eeq
%$\gq^p\cap\overline\gq^p=\wh\gg^p\cap\gu^p\cap\bar\gu^p$.
%\bigoplus_{p\geq q}(\wh\gg^p\cap\gu^p)\oplus\bigoplus_{0\leq p\leq q-1}(\wh\gg^p\cap\gu^p\cap\bar\gu^p)\,,
First note that each term $\gq_q$ is $\bZ$-graded, by its very definition  \eqref{bologna} and the fact that 
$\bar\gq$ is $\bZ$-graded. 
Conditions (i) and (iii) of Definition \ref{ukip} imply then
\begin{align*}
\gq_{0}&=\left\{X\in\gq\,|\phantom{C^{C^C}}\!\!\!\!\!\!\!\!\!\!\![X,\bar\gq]\subset\gq+\bar\gq \right\}\\
&=\left\{X\in\gq\,|\phantom{C^{C^C}}\!\!\!\!\!\!\!\!\!\!\![X,\bar\gq]\subset\bigoplus_{p\geq -1}\wh\gg^{p} \right\}\\
&\subset \left\{X\in\gq\,|\phantom{C^{C^C}}\!\!\!\!\!\!\!\!\![X,\gc^{-1(01)}]\subset \bigoplus_{p\geq -1}\wh\gg^{p} \right\}\;,
\end{align*}
and hence $\displaystyle\gq_0\subset\gq\cap\bigoplus_{p\geq 0}\wh\gg^{p}=\bigoplus_{p\geq 0}\gq^p$ by the nondegeneracy of $\gc_-=\gc^{-2}\oplus\gc^{-1}$; the opposite inclusion is immediate, hence $\displaystyle\gq_0=\bigoplus_{p\geq 0}\gq^p$. 

Assume now that \eqref{chebarba} holds for any 
%nonnegative integer 
$q$ strictly smaller than a positive integer $r+1$; we want to show \eqref{chebarba} for $q=r+1$ too. 
The induction hypothesis and condition (iii) imply
\begin{align*}
\gq_{r+1}&=\left\{X\in \gq_r\,|\phantom{C^{C^C}}\!\!\!\!\!\!\!\!\!\!\![X,\bar\gq]\subset   \gq_{r}+\bar\gq\right\}\\
&=\left\{X\in \bigoplus_{p\geq r}\gq^p\oplus\bigoplus_{0\leq p\leq r-1}\gq^p\cap\overline\gq^p\,|\phantom{C^{C^C}}\!\!\!\!\!\!\!\!\!\!\![X,\bar\gq]\subset   \bigoplus_{p\geq r}\gq^p+\bar\gq\right\}\\
&=\left\{X\in \bigoplus_{p\geq r}\gq^p\oplus\bigoplus_{0\leq p\leq r-1}\gq^p\cap\overline\gq^p\,|\phantom{C^{C^C}}\!\!\!\!\!\!\!\!\!\!\![X,\bar\gq]\subset   \bigoplus_{p\geq r}\wh\gg^p+\bar\gq\right\}\
\end{align*}%(\wh\gg^p\cap\gu^p\cap\bar\gu^p)
and
$\displaystyle\bigoplus_{p\geq r+1}\gq^p\oplus\bigoplus_{0\leq p\leq r}\gq^p\cap\overline\gq^p   \subset\gq_{r+1}$. We now prove the opposite inclusion.

By condition (iv) the natural projection $\pi_{\mathfrak{M}^{r(10)}}:\wh\gc^r\longrightarrow \mathfrak{M}^{r(10)}$ 
%The projection onto the $\ad(\mathfrak{J})$-eigenspace $\mathfrak{M}^{r(10)}$ of maximal eigenvalue in $\wh\gc^r$ 
yields  an identification 
$%\label{serve}
\gq^r/ \gq^r\cap\overline\gq^r\simeq \mathfrak{m}^{r(10)}
$  and any $X\in\gq^r$ has a unique decomposition $X=X^{10}+X_o$ where $X^{10}=\pi_{\mathfrak{M}^{r(10)}}(X)\in\gm^{r(10)}$ and $X_o\in \gu^r\cap\bar\gu^r$. By Proposition \ref{bed} and the definition of core we have for all $v\in\gc^{-1(01)}$
\begin{equation*} 
\begin{split}
&[X,v]=[X^{10},v]+[X_o,v]\qquad\text{where}\\
&[X^{10},v]\in\gm^{r-1(10)}\qquad\text{and}\qquad [X_o,v]\in\overline\gu^{r-1}
\end{split}
\end{equation*}
and $X^{10}=0$ if and only if $[X^{10},\gc^{-1(01)}]=0$. This implies $$\gq_{r+1}\subset\displaystyle\bigoplus_{p\geq r+1}\gq^p\oplus\bigoplus_{0\leq p\leq r}\gq^p\cap\overline\gq^p$$ and hence \eqref{chebarba} for all $q\geq 0$. 

The claims on the core and $k$-nondegeneracy follow from \eqref{brachetto} and the identifications
%\begin{align*}
%T_x^\bC M&\simeq\wh{\gg}/(\wh\gg\cap\gu\cap\bar\gu)=(\wh\gg\cap\gu)/(\wh\gg\cap\gu\cap\bar\gu)\oplus(\wh\gg\cap\bar\gu)/(\wh\gg\cap\gu\cap\bar\gu)\oplus\wh\gg^{-2}\\
%&\simeq\cD^{10}|_x\oplus\cD^{01}|_x\oplus T_x^\bC M/\cD^\bC|_x\ .
%\end{align*}
$$\gm_x^{q(10)}\simeq\frac{\;\;\cF_q^{10}|_x}{\cF_{q+1}^{10}|_x}\simeq\gq_{q}/\gq_{q+1}\simeq \gq^q/\gq^q\cap\overline\gq^q\simeq\gm^{q(10)}\;,$$
for all $q\geq 0$. In view of $k$-nondegeneracy and \cite[Prop. 3.1]{BRZ} (see also \cite[Remark, pag. 904]{Fel} for more details), it also follows that $\gg$ is finite-dimensional.

Finally, the (germ of) locally homogeneous CR manifold $M=G/G_o$ associated with the CR algebra $(\gg,\gq)$ is of hypersurface type since
\begin{align*}
T^\bC_xM&\simeq \wh\gg/\gq\cap\overline\gq\;,\\%=\bigoplus_{p\geq -2} \wh\gg^p/(\gg^p\cap\gu^p)
\cD^\bC|_x&\simeq \gq+\overline\gq/\gq\cap\overline\gq\;,
%\wh\gg\cap\gu/(\wh\gg\cap\gu\cap\bar\gu)=
%\bigoplus_{p\geq -1} \wh\gg^p\cap\gu^p/(\wh\gg^p\cap\gu^p\cap\bar\gu^p)
\end{align*}  
by \eqref{itangente} and $\wh\gg^p=\gq^p+\overline\gq^p$ for all $p\geq -1$; it is clearly strongly regular, by transitivity of the action of the Lie algebra $\gg$ of infinitesimal CR automorphisms. 
%and $\dim(T^\bC_x M/\cD^\bC|_x)=1$. 
 This proves (iii).
%In particular any CR manifold $M=G/G_o$ with $Lie(G)=\gg$, $Lie(\gg_o)=G_o$ is homogeneous under a finite-dimensional Lie group $G$ consisting of CR automorphisms, hence strongly regular of type $\gm$, $k$-nondegenerate and of hypersurface type.
\end{proof}
\begin{remark}
It would be interesting to look for a purely algebraic proof of the fact
that any model is finite-dimensional. In this paper we show the somehow weaker fact that most of the finite dimensional models $\gg$ that we determine are maximal
(see Theorem \ref{thmsimple}, %of \S \ref{exexamples} 
Theorem \ref{thmstrange}). %of \S \ref{exexamplesII}
\end{remark}
\begin{example}[{\it Levi-nondegenerate CR manifolds}]\label{example00}\hfill\medskip\par
A connected CR manifold $(M,\cD,\cJ)$ of dimension $2n+1$ with first Levi form $\cL^{1}:\cD^{10}\times \cD^{01}\longrightarrow T^\bC M/ \cD^\bC$  nondegenerate at all points is strongly regular of type $\gm$ with $\gm=\gm^{-2}\oplus\gm^{-1}$ given by the Heisenberg algebra of some signature $\sgn(\mathfrak{J})=(r,s)$, $n=r+s$. By the results of \cite{Ta2, Ta3, MeNa1} (see also \cite[\S 3]{Tu}), for any signature there exists a unique maximal model, it is a grading
$\displaystyle\gg=\bigoplus\gg^p$ of the simple Lie algebra $\gg=\su(r+1,s+1)$ with
$$
\gg^p=\begin{cases}
0\;\;\;\;\;\;\;\;\;\;\;\;\;\;\;\qquad\qquad\qquad\qquad\qquad&\text{for all}\;\;|p|>2\;,\\
\bR\,\;\;\;\;\;\;\;\;\;\;\;\;\;\;\;\;\qquad\qquad\qquad\qquad\qquad&\text{for}\;\;p=-2\;,\\
\bC^{n}\,\,\;\;\;\;\;\;\;\;\;\;\;\;\;\;\;\;\;\qquad\qquad\qquad\qquad\quad&\text{for}\;\;p=-1\;,\\
\mathfrak{u}(r,s)\oplus\bR E\;\;\;\;\;\;\;\;\;\;\;\;\;\;\qquad\qquad\qquad&\text{for}\;\;p=0\;,\\
(\bC^{n})^*\,\;\;\;\;\;\;\;\;\;\;\;\;\;\;\;\;\,\qquad\qquad\qquad\qquad&\text{for}\;\;p=1\;,\\
\bR^*\,\;\;\;\;\;\;\;\;\;\;\;\;\;\;\qquad\qquad\qquad\qquad\qquad&\text{for}\;\;p=2\;.\\
\end{cases}
$$
In other words $\wh\gg^0=S^{1,1}\oplus\bC E$, $\wh\gg^1=S^{1,0}\oplus S^{0,1}$ and $\wh\gg^2=S^0$ (this can also be   checked directly using
Proposition \ref{noioso}).
\end{example}
%\medskip\par
\begin{example}[{\it CR manifolds of dimension $5$}]\label{example01}\hfill\medskip\par
A $5$-dimensional connected CR manifold $(M,\cD,\cJ)$ with $k$-nondegenerate Levi form, $k\geq 2$, is actually $2$-nondegenerate at all points and strongly regular of type $\gm$ where $\sgn(\mathfrak{J})=(1,0)$ and
\begin{multline*}
\gm=\gm^{-2}\oplus\gm^{-1}\oplus\gm^0\qquad\text{with}\\
\gm^{-2}=\bR\;,\quad\gm^{-1}=\bC\;,\quad\gm^0=\mathfrak{M}^0=\Re(S^{2,0}\oplus S^{0,2})\;,
\end{multline*}
by a simple dimension argument applied to \eqref{dimensione} with $n=1$. By the results of \cite{FK1, FK2} there is a unique maximal model of type $\gm$, a grading
$\displaystyle\gg=\bigoplus\gg^p$ of the simple Lie algebra $\gg=\so(3,2)$ with 
$$
\gg^p=\begin{cases}
0\;\;\;\;\;\;\;\;\;\;\;\;\;\;\;\qquad\qquad\qquad\qquad\qquad&\text{for all}\;\;|p|>2\;,\\
\bR\,\;\;\;\;\;\;\;\;\;\;\;\;\;\;\;\;\qquad\qquad\qquad\qquad\qquad&\text{for}\;\;p=-2\;,\\
\bC\,\;\;\;\;\;\;\;\;\;\;\;\;\;\;\;\;\quad\qquad\qquad\qquad\qquad\quad&\text{for}\;\;p=-1\;,\\
\mathfrak{sp}_{2}(\bR)\oplus\bR E\,\;\;\;\;\;\;\;\;\;\;\;\;\;\qquad\qquad\qquad&\text{for}\;\;p=0\;,\\
\bC^*\,\,\,\;\;\;\;\;\;\;\;\;\;\;\;\;\;\;\;\;\quad\qquad\qquad\qquad\qquad&\text{for}\;\;p=1\;,\\
\bR^*\,\;\,\;\;\;\;\;\;\;\;\;\;\;\;\;\qquad\qquad\qquad\qquad\qquad&\text{for}\;\;p=2\;,\\
\end{cases}
$$
(see also \cite{SV} and the description in \cite[\S 3.2]{MS}). In other words 
\begin{multline*}
\wh\gg^0=\wh\gc^0=\wh{\mathfrak{M}}^0\oplus S^{1,1}\oplus \bC E\qquad\text{and}\\
\wh\gg^1=S^{1,0}\oplus S^{0,1}\;,\qquad\wh\gg^2=S^0\;.
\end{multline*}
We remark that $\gg^0$ not only contains the $0$-degree part $\mathfrak{u}(1)\oplus\bR E$ of the real isotropy algebra but also the real part of the direct sum $\wh{\mathfrak{M}}^0=S^{2,0}\bigoplus S^{0,2}$ of the two $\ad(\mathfrak J)$-eigenspaces in $\hat\gc^0$ 
of extremal eigenvalues $\pm 2i$.
\end{example}
%\medskip\par
The main aim of \S \ref{exexamples}  and \S \ref{exexamplesII} is to study $7$-dimensional models and their associated homogeneous CR manifolds. In view of Definition
\ref{ukip} and Theorem \ref{primoteorema} it is important to first classify the $7$-dimensional abstract cores up to isomorphism.
This is the content of \S \ref{sec:7cores}.
%\bigskip\par
\subsection{Classification of abstract cores in dimension \texorpdfstring{$7$}{7}}
\label{sec:7cores}
\setcounter{table}{0}
A simple %dimensional 
argument yields three main classes of abstract cores associated with strongly regular $7$-dimensional CR manifolds.
\begin{table}[H]
\begin{centering}
\makebox[\textwidth]{%
\begin{tabular}{|c|c|c|c|c|c|}
\hline
$\;\;\;\text{Class}^{\phantom{C^C}}$ &  $\gm^{-2}$  & $\gm^{-1}$ & $\gm^0$ & $\gm^{1}$ & $\gm^{p}\;(p> 1)$ \\
\hline
\hline
&&&&\\[-3mm]
(A) & $\bR$
& $\bC^3$ & $0$ & $0$ & $0$ 
\\
\hline
\hline
&&&&\\[-3mm]
(B) & $\bR$ & $\bC^2$ & $\bC$ & $0$ & $0$
\\
\hline
\hline
&&&&\\[-3mm]
(C)& $\bR$ & $\bC^{\phantom{1}}$  & $\bC$ & $\bC$ & $0$
\\
\hline
\end{tabular}}
\label{tab:1}
\caption{The abstract cores $\gm=\bigoplus\gm^p$ with $\dim(\gm)=7$.}
\end{centering}
\end{table}
%\begin{itemize}
%\item[(A)] $\gm^{-2}=\bR$, $\gm^{-1}=\bC^3$ and $\gm^{p}=(0)$ for all $p\geq 0$,
%\item[(B)] $\gm^{-2}=\bR$, $\gm^{-1}=\bC^2$, $\gm^{0}=\bC$ and $\gm^{p}=(0)$ for all $p\geq 1$,
%\item[(C)] $\gm^{-2}=\bR$, $\gm^{-1}=\bC$, $\gm^0=\bC$, $\gm^{1}=\bC$ and $\gm^{p}=(0)$ for all $p\geq 2$,
%\end{itemize}
Note that the case $\gm^{-2}=\bR$, $\gm^{-1}=\bC$, $\gm^0=\bC^2$ and $\gm^{p}=0$ for all $p>0$ is not permissible by \eqref{dimensione}.
Class (A) correspond to the Levi-nondegenerate case described in Example \ref{example00} and it is easy to see that there exists just one core in class (C), i.e., the ``$3$-nondegenerate'' core
\begin{multline}
\label{kore}
\gm^{-2}=\gc^{-2}=\left\langle e^{-2}\right\rangle\;,\qquad\gm^{-1}=\gc^{-1}=\left\langle e_1,e_2\right\rangle\qquad\text{and}\\
\gm^{0(10)}=\mathfrak{M}^{0(10)}\;,\qquad\gm^{1(10)}=\mathfrak{M}^{1(10)}\,.\quad
\end{multline}
%The strongly regular CR manifolds with core \eqref{kore} are $3$-nondegenerate.
The description of the equivalence classes of cores in (B) is more involved and splits into $\sgn(\mathfrak J)=(2,0)$ and  $\sgn(\mathfrak J)=(1,1)$. 
For our purposes it is convenient to consider a {\it complex-symplectic basis} $\left\{e_{1},\dots, e_{4}\phantom{C^{C^C}}\!\!\!\!\!\!\!\!\!\!\!\!\right\}$ of $\gc^{-1}\simeq\bR^4$, that is a real basis satisfying
\begin{equation}
\label{complexsymplectic}
\begin{cases}
B(e_i,e_{i+2})=-B(e_{i+2},e_i)=1\;\;&\text{for all}\;\;1\leq i\leq 2,\\
B(e_i,e_j)=0\;\;&\text{otherwise},\\
\mathfrak J(e_i)=-e_{i+2}\;,\;\;\mathfrak J(e_{i+2})=e_{i}\;\;&\text{for all}\;\; 1\leq i\leq r,\\
\mathfrak J(e_i)=e_{i+2}\;,\;\;\mathfrak J(e_{i+2})=-e_{i}\;\;&\text{for all}\;\; r+1\leq i\leq 2.
\end{cases}
\end{equation}
%and identify $\gk^0$ with the simple Lie algebra of matrices
%\beq
%\notag
%\gk^0\simeq \sp_4(\bR)=\left\{\begin{pmatrix} A & B \\ C & -A^T \end{pmatrix}\,|\, A, B, C\in\operatorname{Mat}(2,2;\bR) 
%,\, B=B^T,\, C=C^T\right\}\;.
%\eeq
We also record here, for later use in \S \ref{exexamples}, that another kind of natural basis can be considered when $\sgn(\mathfrak J)=(1,1)$. In this case
\begin{equation}
\label{complexWitt}
\begin{cases}
B(e_i,e_{i+2})=-B(e_{i+2},e_i)=1\;\;&\text{for all}\;\;1\leq i\leq 2,\\
B(e_i,e_j)=0\;\;&\text{otherwise},\\
\mathfrak J(e_1)=-e_{4}\;,\;\;&\mathfrak J(e_{4})=e_{1},\\
\mathfrak J(e_2)=-e_{3}\;,\;\;&\mathfrak J(e_{3})=e_{2},
\end{cases}
\end{equation}
and we call such a basis a {\it complex-Witt basis} of $\gc^{-1}$. It is easy to see that the associated vectors
\begin{equation}
\label{Witt-symplectic}
\begin{split}
e_1'&=\frac{e_1+e_2}{\sqrt{2}}\;,\qquad e_2'=\frac{e_1-e_2}{\sqrt{2}}\;,\\
e_3'&=\frac{e_3+e_4}{\sqrt{2}}\;,\qquad e_4'=\frac{e_3-e_4}{\sqrt{2}}\;,\\
\end{split}
\end{equation}
constitute a complex-symplectic basis $\left\{e'_{1},\dots, e'_{4}\phantom{C^{C^C}}\!\!\!\!\!\!\!\!\!\!\!\!\right\}$ of $\gc^{-1}$.
%We observe, for later use in \S \ref{exexamples} that
%\begin{equation}
%B(e_1^{10},e_1^{01})=-\frac{i}{2}\qquad\text{and}B(e_2^{10},e_2^{01})=\mp\frac{i}{2} 
%\end{equation}

To classify the abstract core in class (B) of  a fixed signature we first need some simple observations concerning the  action of the Lie group $\Aut(\gc,\mathfrak J)$ %with Lie algebra $Lie(\Aut(\gc,\mathfrak{J}))=\bR E\oplus\Re(S^{1,1})\subset\gc^0$ 
(recall the discussion above Definition \ref{Knondeg}). %We will see that an appropriate orbit space $\mathbb X$ describes the equivalence classes of cores.
First of all we note that any core $\gm=\gm^{-2}\oplus\gm^{-1}\oplus\gm^{0}=\gc^{-2}\oplus\gc^{-1}\oplus\gm^0$ is fully determined by the complex line $\gm^{0(10)}\subset\mathfrak M^{0(10)}\simeq S^2\bC^2$ given by the $\mathfrak{J}$-holomorphic part of $\wh\gm^0=\gm^{0(10)}\oplus\overline{\gm^{0(10)}}$. Furthermore the natural action of $\Aut(\gc,\mathfrak J)$ on $S^2\bC^2$ factors through the quotient $K_\sharp=\Aut(\gc,\mathfrak J)/\bZ_2$, namely to the Lie group
$K_\sharp=\bC^\times\cdot K$ 
associated with the connected and closed subgroup $K$ of $\SL_3(\bR)$ given by
\begin{align*}
K&=\mathrm{SO}_3(\bR)\simeq\mathrm{SU}(2)/\bZ_2 \;\;&\text{if}\;\;\sgn(\mathfrak J)=(2,0)\;,\\ 
K&=\mathrm{SO}^+(2,1)\simeq\SU(1,1)/\bZ_2\;\;&\text{if}\;\;\sgn(\mathfrak J)=(1,1)\;.
\end{align*} 
More explicitly we fix the basis $(\e_\alpha)$ of $\mathfrak{sl}_2(\bC)\simeq S^2\bC^2$  given by
\beq
\label{ortbasis}
\begin{split}
\e_1&=\begin{pmatrix} 0 & -1 \\ 1 & 0 \end{pmatrix}\simeq-\frac{1}{2}(e_1^{10}\odot e_1^{10}+e_2^{10}\odot e_2^{10})\;,\\
\e_2&=\begin{pmatrix} 0 & i \\ i & 0 \end{pmatrix}\simeq\frac{i}{2}(e_1^{10}\odot e_1^{10}-e_2^{10}\odot e_2^{10})\;,\\
\e_3&=\begin{pmatrix} i & 0 \\ 0 & -i \end{pmatrix}\simeq-i(e_1^{10}\odot e_2^{10})\;\qquad
\end{split}
\eeq
if $\sgn(\mathfrak{J})=(2,0)$ and
\beq
\label{ortbasis2}
\begin{split}
\e_1&=\begin{pmatrix} 0 & 1 \\ 1 & 0 \end{pmatrix}=\frac{1}{2}(e_1^{10}\odot e_1^{10}-e_2^{10}\odot e_2^{10})\;,\\
\e_2&=\begin{pmatrix} 0 & -i \\ i & 0 \end{pmatrix}=-\frac{i}{2}(e_1^{10}\odot e_1^{10}+e_2^{10}\odot e_2^{10})\;,\\
\e_3&=\begin{pmatrix} i & 0 \\ 0 & -i \end{pmatrix}=-i(e_1^{10}\odot e_2^{10})
\end{split}
\eeq
if $\sgn(\mathfrak J)=(1,1)$, so that $\bR^3=\operatorname{span}_{\bR}\left\{\e_1,\e_2,\e_3\right\}$ equals $\mathfrak{su}(2)$ %in \eqref{ortbasis} 
and $\mathfrak{su}(1,1)$, respectively. 
%in \eqref{ortbasis2} 
Using this identification, the action of $\Aut(\gc,\mathfrak J)$ on $S^2\bC^2$  factors to the representation of $K_\sharp$ on $V=\bC^3=\operatorname{span}_{\bC}\left\{\e_1,\e_2,\e_3\right\}$
%(the action of $K$ on $\sl_2(\bC)$ is given by the adjoint action of $\mathrm{SU}(2)$ 
%$\bR^3\simeq\mathfrak{su}(2)=\left\langle\e_1,\e_2,\e_3\right\rangle$
%$\bR^{2,1}\simeq\mathfrak{su}(1,1)=\left\langle\e_1,\e_2,\e_3\right\rangle$ 
\beq
\label{eq:barile}
\rho:K_\sharp\longrightarrow \GL(V)
\eeq 
given by the multiplications by the non-zero complex scalars and the $\bC$-linear extension to $V$ of the natural action of $K$ on $\bR^3$:
\begin{align}
\label{mori0o}
\notag(c,A)\cdot z&= c\cdot(Ax+iAy)\\
&=e^{i\vartheta}\cdot (\rho Ax+i\rho Ay)\qquad\qquad\\
\notag &=(\rho \cos\vartheta\cdot  Ax-\rho \sin\vartheta\cdot Ay)+ i(\rho \cos\vartheta\cdot Ay+\rho\sin\vartheta\cdot  Ax)\qquad
\end{align}
where $c=\rho e^{i\vartheta}\in\bC^\times$, $A\in K$ and $z=x+iy\in V$. 
It follows that two cores $\gm$ and $\gm'$ of the same signature are isomorphic if and only if corresponding
%there is $g\in\mathrm{U}(1,1)$ such that $g\cdot\gm^{0(10)}=\gm'^{0(10)}$ as $1$-dimensional lines in $V=S^{2,0}=S^2\bC^2$. Let $\rho:K_\sharp=\bC^*\cdot K\longrightarrow \GL(V)$
%be the representation determined by the isomorphism $K\simeq\mathrm{SU}(1,1)/\bZ_2$;
elements $z,z'\in V^\times$ (determined up to nonzero complex scalars) lie on the same $K_\sharp$-orbit. In other words we shall consider the projective plane $X=\mathbb P^{2}(\bC)$ with the natural projective representation of $K_\sharp$ and identify the orbit space
%we also denote the corresponding projective representation by $\rho_\pi:G_\sharp\longrightarrow\mathrm{PSL}(V)$. 
\begin{align*}
%\begin{split}
%\mathbb V=
%V^\times/K_\sharp=\left\{K_\sharp\cdot z\,|\phantom{C^{C^C}}\!\!\!\!\!\!\!\!\!\!\! z\in V,\,z\neq 0\right\}\\
%\text{and}\qquad 
\mathbb X&=X/K_\sharp\\
&=\left\{K_\sharp\cdot [z]\,|\phantom{C^{C^C}}\!\!\!\!\!\!\!\!\!\!\! [z]\in X\right\}
%\end{split}
\end{align*}
with $V^\times/K_\sharp$ via the canonical equivariant projection from $V^{\times}$ to $X$. 
\smallskip\par
Let now $N_\sharp\subset K_\sharp$ be the stabilizer of $[z]$ and $\gn_\sharp=Lie(N_\sharp)$ its associated Lie algebra.
It is a subalgebra of $Lie(K_\sharp)=\bC\oplus\mathfrak{su}(2)$ or $\bC\oplus\mathfrak{su}(1,1)$.
\begin{definition}
\label{prosciuttocotto}
An orbit $K_\sharp\cdot [z]$ is called of {\it type $N_\sharp$} if $K_\sharp\cdot[z]\simeq K_\sharp/N_\sharp$.
%$N_\sharp$ is the unique, up to conjugation, closed subgroup of $K_\sharp$ such that 
\end{definition} 
To determine the $K_\sharp$-orbits on $V$ it is also convenient, as an intermediate step in the proof of the following Proposition \ref{firstorbitprop} and Proposition \ref{secondorbitprop}, to consider the subgroup $\wt K=\bR^\times\cdot K$ of $K_\sharp$ and the restriction
\beq
\label{eq:rapres}
\wt\rho=\rho|_{\wt K}:\wt K\longrightarrow\GL(V)
\eeq 
of the representation \eqref{eq:barile}. Indeed $V\simeq\bR^3\oplus\bR^3$ as a $\wt K$-module, but not as a $K_\sharp$-module.
Moreover $\bC^\times\cdot \wt H\subset N_\sharp$  where $\wt H$ is the stabilizer of $z$ in $\wt K$ (we remark that this inclusion is in general proper, as there are elements $A\in K$ and $c\in\bC^\times$ with the property that $A\cdot z=c\cdot z$). 
\begin{definition}
\label{def:admissibility}
An orbit $K_\sharp\cdot [z]$ and the associated equivalence class of cores is called {\it admissible} if 
$\mathfrak{n}_\sharp\neq\bC$ (i.e., if the connected component of $N_\sharp$ is not as smallest as possible). 
\end{definition}
%$N_\sharp=\bC^*\cdot H$ only in the positive-definite case and even in this case it might reverse the vector 
%$G_\sharp\cdot z\simeq G_\sharp/G_z$
We now deal with the two signatures separately.
\smallskip\par\noindent
\subsubsection{The case $\mathbf{\sgn(\mathfrak J)=(2,0)}$.}\hfill\vskip0.1cm\par\setcounter{table}{1}
Let $\rho:K_\sharp=\mathbb C^\times\cdot K\longrightarrow \GL(V)$ be
given by the multiplications by $\bC^\times$ and the $\bC$-linear extension to $V$ of the natural action of $K=\mathrm{SO}_{3}(\bR)$ on the Euclidean vector space $(\bR^3,\left\langle\cdot,\cdot\right\rangle)$. %We first give an intermediate result.
%\smallskip\par
\begin{proposition}\label{firstorbitprop}
Every element $z\in V^\times$ is $K_\sharp$-related to one and only one of the canonical forms $\e_1+it \e_2$ for some $t\in [0,1]$.
%The map 
%\begin{multline*}
%\phi:I=[0,1]\longrightarrow X\quad\text{given by}\\
%\;\;\phi(t)=[\e_1+it \e_2]\;,\qquad (\e_\alpha)=\text{stand. basis of}\;\,V=\bC^3\;,\;\;
%\end{multline*}
%induces an identification $\mathbb X\simeq I$ and the orbit $G_\sharp\cdot \phi(t)$ corresponding to $t\in I$ is of type 
%$N_\sharp$ 
The corresponding $K_\sharp$-orbit is of type $N_\sharp$ with
$\mathfrak{n}_\sharp=\mathbb{C}\oplus \so_2(\bR)$ if $t=0,1$ and
$\mathfrak{n}_\sharp=\mathbb{C}$ otherwise. %if $0<t<1$.
\end{proposition}
\begin{proof} 
Postponed to Appendix \ref{appendice2}.
\end{proof}
%\medskip\par
%describe the moduli space $Mod_{(4,0)}(\gm)$ of cores
%$\gm=\bigoplus\gm^p$ up to isomorphisms.
We now directly apply this result to the classification of abstract cores.
\begin{theorem}
\label{firstcores}
Every $7$-dimensional abstract core $\gm=\bigoplus\gm^p$ of type (B) and $\sgn(\mathfrak J)=(2,0)$ is of the form
\begin{align*}
\gm^p&=0\,\;\;\text{for all}\;\;p\neq-2,-1,0\,,\\
\gm^{-2}&=\gc^{-2}=\left\langle\phantom{C^{C^C}}\!\!\!\!\!\!\!\!\!\!\!\! e^{-2}\right\rangle\,,\\
\gm^{-1}&=\gc^{-1}=\left\langle\phantom{C^{C^C}}\!\!\!\!\!\!\!\!\!\!\!\!  e_1,e_2,e_3,e_4\right\rangle\,,\\
\gm^{0}&=\Re(\gm^{0(10)}\oplus\overline{\gm^{0(10)}})\,,
\end{align*}
and isomorphic with one and only one of the canonical forms $\gm_t$ in Table \ref{tablecore20}, $t\in[0,1]$. %\ref{tablecore20}.
Any $7$-dimensional and strongly regular CR manifold $(M,\cD,\cJ)$ of type
$\gm_t$ is $2$-nondegenerate and it is not (even locally) CR diffeomorphic to any other $(M',\cD',\cJ')$ of type $\gm_{t'}$ if $t\neq t'$.\end{theorem} 
%The map $\phi:[0,1]\longrightarrow Mod_{(4,0)}(\gm)$ given by 
%$$\phi(t)=\text{\rm isom. class of}\;\;\gm_t=\bigoplus\gm_t^{p}$$ 
\begin{table}[H]
\begin{centering}
%\makebox[\textwidth]{%
\begin{tabular}{|c|c|c|}
\hline
$\begin{gathered}\\ \gm_t=\bigoplus\gm_t^{p}\\ t\in[0,1]\end{gathered}$ & $\gm_t^{0(10)}{}^{\phantom{C^C}}$ & $\gn_\sharp^{\phantom{C^{C^C}}}$  \\
\hline
\hline
&&\\[-1mm]
$t=0,1$ & $\begin{gathered}\\ \\ (1+t)e_1^{10}\odot e_1^{10}+(1-t)e_2^{10}\odot e_2^{10}\end{gathered}$
& $\bC\oplus\so_2(\bR)$
\\
\cline{1-1}\cline{3-3}
&&\\[3mm]
$\begin{gathered} t\in(0,1)\\ \\ \end{gathered}$ &  & $\begin{gathered}\bC\\ \\ \end{gathered}$  
\\
\hline
\end{tabular}
%}
\caption{The $7$-dimensional cores of signature $(2,0)$ up to isomorphism.}
\label{tablecore20}
\end{centering}
\end{table}
\begin{proof} 
%From definitions, the remark above Definition \ref{Knondeg} and the fact that the grading element $E$ of $\gc$ acts trivially on $\gc^0$, 
Two abstract cores $\gm$ and $\gm'$
%\begin{multline*}
%\gm=\gc^{-2}\oplus\gc^{-1}\oplus\Re(\gm^{0(10)}\oplus\overline{\gm^{0(10)}})\qquad\text{and}
%\\\gm'=\gc^{-2}\oplus\gc^{-1}\oplus\Re(\gm'^{0(10)}\oplus\overline{\gm'^{0(10)}})
%\end{multline*}
are equivalent if and only if the corresponding elements $z,z'\in V^\times$ lie on the same $K_\sharp$-orbit. %of the representation \eqref{eq:barile}. 
The first part of the theorem follows then from 
%the fact that \eqref{ortbasis} is an orthonormal basis
%for the action of $K$ on $\mathfrak{su}(2)$ 
Proposition \ref{firstorbitprop} and
%with $\mathbb X\simeq Mod_{(4,0)}(\gm)$;
the second part is clear as $\operatorname{ht}(\gm)=0$ and the core is an invariant of a CR manifold.
\end{proof}
%\medskip\par
%\vskip0.2cm\par\noindent
\subsubsection{The case $\mathbf{\sgn(\mathfrak J)=(1,1)}$.}\hfill\vskip0.1cm\par\setcounter{table}{2}
In this case \eqref{eq:barile} %$\rho:K_\sharp=\bC^\times\cdot K\longrightarrow \GL(V)$ 
is determined by the action of the connected Lorentz group $K=\mathrm{SO}^+(2,1)$ on the pseudo-Euclidean space $\bR^{2,1}$ with the orthonormal basis $(\e_\alpha)$ such that $\left\langle \e_1,\e_1\right\rangle=\left\langle \e_2,\e_2\right\rangle=-\left\langle \e_3,\e_3\right\rangle=1$. 
We recall that the orbits of $K$ on $\bR^{2,1}\backslash\{0\}$ split into five different types (see e.g. \cite{GMS, Sh}; the arguments are for $\bR^{3,1}$ but extend easily to our case too):
\beq
\label{Lororbits}
\begin{split}
S_{r>0}&=\left\{x\in\bR^{2,1}\;|\phantom{C^{C^C}}\!\!\!\!\!\!\!\!\!\!\!\!\left\langle x,x\right\rangle=r>0\right\}\;,\\
N^{+}&=\left\{x\in\bR^{2,1}\;|\phantom{C^{C^C}}\!\!\!\!\!\!\!\!\!\!\!\!\left\langle x,x\right\rangle=0\,,\,x_3>0\right\}\;,\\
N^{-}&=\left\{x\in\bR^{2,1}\;|\phantom{C^{C^C}}\!\!\!\!\!\!\!\!\!\!\!\!\left\langle x,x\right\rangle=0\,,\,x_3<0\right\}\;,\\
S_{r<0}^{+}&=\left\{x\in\bR^{2,1}\;|\phantom{C^{C^C}}\!\!\!\!\!\!\!\!\!\!\!\!\left\langle x,x\right\rangle=r<0\,,\,x_3>0\right\}\;,\\
S_{r<0}^{-}&=\left\{x\in\bR^{2,1}\;|\phantom{C^{C^C}}\!\!\!\!\!\!\!\!\!\!\!\!\left\langle x,x\right\rangle=r<0\,,\,x_3<0\right\}\;.\\
\end{split}
\eeq
Any orbit of the first three types is diffeomorphic to $S^1\times\bR$ and of the form $K/H$ where the stabilizer $H$ is non-compact and conjugated in $K$ to
the subgroup
$$
H\simeq\begin{dcases*} \text{the subgroup}\;\;\SO^{+}(1,1)\;\;\text{of boosts in the case}\;\;S_{r>0}\;,\\
 \text{the subgroup}\;\;\bR\;\;\text{of null rotations in the case}\;\;N^{\pm}\;.
\end{dcases*}
$$
For example, the stabilizer of $\e_1+\e_3$ is represented in the ``mixed basis'' $\{\e_1+\e_3,\e_2,\e_1-\e_3\}$ by the one-parameter group of null rotations 
$$\begin{pmatrix} 1 & -d/2 & -d^2/4 \\ 0 &1  & d \\ 0 &0 & 1 \end{pmatrix}\;,\qquad d\in\bR\;.$$
Any orbit $S^{\pm}_{r<0}$ is of the form $K/H\simeq \bR^2$ where $H\simeq\SO_{2}(\bR)$ is compact subgroup of ordinary rotations. The orbits \eqref{Lororbits} are surfaces of transitivity for $\mathrm{O}^+(2,1)$ too whereas those for the general Lorentz group $\mathrm{O}(2,1)$
are of only three types, namely $S_{r>0}$, $N=N^{+}\cup N^-$ and $S_{r<0}=S^+_{r<0}\cup S^-_{r<0}$. These three surfaces are surfaces of transitivity for $\mathrm{SO}(2,1)$ too. 

It follows
that the orbits of $\wt K=\bR^\times\cdot K$ on $\bR^{2,1}\backslash\{0\}$ are exactly three:
$$S_{>0}=\bigcup S_{r>0}\;,\qquad S_{<0}=\bigcup S_{r<0}\qquad\text{and}\;\;\;N\;,$$
with only $S_{>0}$ connected. The following table gives representatives for these orbits together with the associated
stabilizers $\wt H\subset \wt K$. 
\vskip0.1cm
\begin{table}[H]
\begin{centering}
%\makebox[\textwidth]{%
\begin{tabular}{|c|c|c|}
\hline
$\;\;\;\text{Orbit}^{\phantom{C^C}}$ & \text{Representative} & $\wt H^{\phantom{C^{C^C}}}$   \\
\hline
\hline
&&\\[-1mm]
$S_{>0}$ & $\e_1$
& $\SO^{+}(1,1)\cup \SO^{+}(1,1)\cdot\begin{psmallmatrix} 1 & 0 & 0 \\ 0 &1  &0 \\ 0 &0 & -1 \end{psmallmatrix}_{\phantom{C_{C_C}}}$
\\
\hline
\hline
&&\\[-1mm]
$N$& $\e_1+\e_3$  & $\bR^+\ltimes\bR_{\phantom{C_{C_C}}}$
\\
\hline
\hline
&&\\[-1mm]
$S_{<0}$& $\e_3$ & $\mathrm{SO}_2(\bR)_{\phantom{C_{C_C}}}$
\\
\hline
\end{tabular}
\label{tableorbit2}
%}
\caption{The orbits of $\wt K=\bR^\times\cdot\mathrm{SO}^{+}(2,1)$ on $\bR^{2,1}\backslash\{0\}$.}
\end{centering}
\end{table}
Here the subgroup $\bR^+\ltimes\bR$ is the semidirect product of the subgroup $\bR$ of null rotations and the  one parameter group of ``dilations'' given in the mixed basis by
\begin{equation*}
\bR^+=\left\{\begin{pmatrix} 1 & 0 & 0 \\ 0 & \rho  &0 \\ 0 &0 & \rho^2 \end{pmatrix}\,|\phantom{C^{C^C}}\!\!\!\!\!\!\!\!\!\! \rho>0\right\}\;.
\end{equation*} 
A deeper analysis of the representation \eqref{eq:barile} gives the following.
\begin{proposition}
\label{secondorbitprop} 
Every element $z\in V^\times$ is $K_\sharp$-related to one and only one of the canonical forms:
\begin{itemize}
\item[1.] $\e_1+it\e_2$ for $t\in [-1,1]$;
\item[2.] $\e_1+\e_3+i(t(\e_1+\e_3)+(\e_1-\e_3))$ for $t\in\bR$;
\item[3.] $\e_1+\e_3\pm i\e_2$;
\item[4.] $\e_3$;
\item[5.] $\e_1+\e_3$.
\end{itemize}
The corresponding $K_\sharp$-orbit is of type $N_\sharp$ with
\begin{itemize}
\item[--] $\mathfrak{n}_\sharp=\bC\oplus\so_2(\bR)$ for $\e_1\pm i\e_2$ and $\e_3$;
\item[--] $\mathfrak{n}_\sharp=\bC\oplus\so(1,1)$ for $\e_1$; 
\item[--] $\mathfrak{n}_\sharp=\bC\oplus(\bR\inplus\bR)$ for $\e_1+\e_3$;
\item[--] $\mathfrak{n}_\sharp=\bC$ otherwise.
\end{itemize}
%$N_\sharp=\bC^*\cdot\mathrm{O}_{2}(\bR)$ if $t=0$, $\bC^*\cdot\bZ_2$ if $0<t<1$ and $\bC^*\cdot SO_2(\bR)$ if $t=1$.
 %(\mathbb C^*\times \mathrm{SO}(n,\bR))/\{\pm\Id\}$ 
%$$
%\rho(c,A)(z)=c\cdot (Ax+iAy)\,,
%$$
%for any $c\in\mathbb C^*$, $A\in\mathrm{SO}(n,\bR)$ and $z=x+iy\in\bC^n$. 
\end{proposition}
\begin{proof} 
Postponed to Appendix \ref{appendice3}.
\end{proof}
%We can now describe the moduli space $Mod_{(2,2)}(\gm)$ of cores
%$\gm=\bigoplus\gm^p$ up to isomorphisms. 
In the following $\left\{e_{1},\dots, e_{4}\phantom{C^{C^C}}\!\!\!\!\!\!\!\!\!\!\!\!\right\}$ is, as usual in this section, a complex-symplectic basis of $\gc^{-1}$  (and not a complex-Witt basis).
\begin{theorem}
\label{secondcores}
Every $7$-dimensional abstract core $\gm=\bigoplus\gm^p$ of type (B) and $\sgn(\mathfrak J)=(1,1)$ is of the form
\begin{align*}
\gm^p&=0\;\;\text{for all}\;\;p\neq-2,-1,0\,,\\
\gm^{-2}&=\gc^{-2}=\left\langle\phantom{C^{C^C}}\!\!\!\!\!\!\!\!\!\!\!\! e^{-2}\right\rangle\,,\\
\gm^{-1}&=\gc^{-1}=\left\langle\phantom{C^{C^C}}\!\!\!\!\!\!\!\!\!\!\!\!  e_1,e_2,e_3,e_4\right\rangle\,,\\
\gm^{0}&=\Re(\gm^{0(10)}\oplus\overline{\gm^{0(10)}})\,
\end{align*}
and isomorphic with one and only one of the canonical forms $\gm_t$, $\wt\gm_t$, $\gm_\pm$, $\gm_{<0}$ and $\gm_{\rm null}$ in Table \ref{simbacacca}.
Any $7$-dimensional and strongly regular CR manifold $(M,\cD,\cJ)$ of type $\gm$ where $\gm$ is
a core in Table \ref{simbacacca} is $2$-nondegenerate and it is not (even locally) CR diffeomorphic to any other $(M',\cD',\cJ')$ whose core $\gm'$ is also in Table \ref{simbacacca} and different from $\gm$.
\end{theorem}
\begin{proof} 
It uses Proposition \ref{secondorbitprop} and it is completely analogous to that of Theorem \ref{firstcores}. 
\end{proof}
\begin{remark}
We remark that
Proposition \ref{secondorbitprop} and Theorem \ref{secondcores} are not exhaustive of the topological structure of the moduli space
$\mathbb X$ of cores $\gm$ of $\sgn(\mathfrak J)=(1,1)$ up to equivalence (this is due to the fact that $\mathrm{SU}(1,1)$ is not compact).
For instance, one can show  that $\gm_{<0}$ is a point at infinity of the $\wt\gm_t$'s whereas $\gm_{\rm null}$ is arbitrarily close to $\wt\gm_0$;
%$\sgn(\left\langle z,z\right\rangle)\in\left\{-1,0,+1\phantom{C^{C^C}}\!\!\!\!\!\!\!\!\!\!\!\!\right\}$ is constant on each
%$G_\sharp$-orbit.
we will not need this finer structure in this paper. 
\end{remark}
We see from Table \ref{tablecore20} and Table \ref{simbacacca} that there are precisely seven different classes of abstract cores $\gm=\bigoplus\gm^p$ with $\dim(\gm)=7$ and $\operatorname{ht}(\gm)=0$ which are admissible (that is, $\gn_\sharp$ properly contains $\bC$). The next section deals with the corresponding models.
\par\noindent
{\small
\begin{table}[H]
\begin{centering}
%\makebox[\textwidth]{%
\begin{tabular}{|c|c|c|}
\hline
$\begin{gathered}\\ \gm_t=\bigoplus\gm_t^{p}\\ t\in[-1,1]\end{gathered}$ & $\gm_t^{0(10)}{}^{\phantom{C^C}}$ & $\gn_\sharp^{\phantom{C^{C^C}}}$  \\
\hline
&&\\[3mm]
$\begin{gathered} t=\pm 1\\ \\ \end{gathered}$ & $$
& $\begin{gathered} \bC\oplus\so_2(\bR)\\ \\ \end{gathered}$
\\
\cline{1-1}\cline{3-3}
&&\\[3mm]
$\begin{gathered} t=0\\ \\ \end{gathered}$ & $\begin{gathered}(1+t)e_1^{10}\odot e_1^{10}+(t-1)e_2^{10}\odot e_2^{10}\end{gathered}$ & $\begin{gathered}\bC\oplus\so(1,1)\\ \\ \end{gathered}$  
\\
\cline{1-1}\cline{3-3}
&&\\[3mm]
$\begin{gathered} t\neq \pm 1,0\\ \\ \end{gathered}$ &  & $\begin{gathered}\bC\\ \\ \end{gathered}$  
\\
\hline
\hline
$\begin{gathered}\\ \wt\gm_t=\bigoplus\wt\gm_t^{p}\\ t\in\bR\end{gathered}$ & $\wt\gm_t^{0(10)}{}^{\phantom{C^C}}$ & $\gn_\sharp^{\phantom{C^{C^C}}}$  \\
\hline
&&\\[-1mm]
$t\in\bR$ & $\begin{gathered} \\ e_1^{10}\odot e_1^{10}-e_2^{10}\odot e_2^{10}+2(t-1)e_1^{10}\odot e_2^{10}\\
+i((1+t)e_1^{10}\odot e_1^{10}-(1+t)e_2^{10}\odot e_2^{10}-2e_1^{10}\odot e_2^{10})\\ \\
\end{gathered}$
& $\bC$
\\
\hline
\hline
$\begin{gathered}\\ \gm_\pm=\bigoplus\gm_\pm^{p}\\ \\ \end{gathered}$ & $\gm_\pm^{0(10)}{}^{\phantom{C^C}}$ & $\gn_\sharp^{\phantom{C^{C^C}}}$  \\
\hline
&&\\[-1mm]
$t=\pm 1$ & $\begin{gathered}\\ (1+t)e_1^{10}\odot e_1^{10}+(-1+t)e_2^{10}\odot e_2^{10}-2i e_1^{10}\odot e_2^{10}\\ \\ \end{gathered}$
& $\bC$
\\
\hline
\hline
$\begin{gathered}\\ \gm_{<0}=\bigoplus\gm_{<0}^{p}\\ \\ \end{gathered}$ & $\gm_{<0}^{0(10)}{}^{\phantom{C^C}}$ & $\gn_\sharp^{\phantom{C^{C^C}}}$  \\
\hline
&&\\[-1mm]
& $\begin{gathered}\\ e_1^{10}\odot e_2^{10}\\ \\ \end{gathered}$
& $\bC\oplus\so_2(\bR)$
\\
\hline
\hline
$\begin{gathered}\\ \gm_{\rm null}=\bigoplus\gm_{\rm null}^{p}\\ \\ \end{gathered}$ & $\gm_{\rm null}^{0(10)}{}^{\phantom{C^C}}$ & $\gn_\sharp^{\phantom{C^{C^C}}}$  \\
\hline
&&\\[-1mm]
& $\begin{gathered}\\ e_1^{10}\odot e_1^{10}-e_2^{10}\odot e_2^{10}-2i e_1^{10}\odot e_2^{10}\\ \\ \end{gathered}$
& $\bC\oplus(\bR\inplus\bR)$
\\
\hline
\end{tabular}
%}
\caption{The $7$-dimensional cores of signature $(1,1)$ up to isomorphism.}
\label{simbacacca}
\end{centering}
\end{table}}
%\begin{itemize}
%\item[1.] $\gm_t=\bigoplus\gm_t^{p}$ where $t\in [-1,1]$ and
%\begin{align*}
%\gm_{t}^{0(10)}&=\left\langle\phantom{C^{C^C}}\!\!\!\!\!\!\!\!\!\!\!\!\! (1+t)e_1^{10}\odot e_1^{10}+(t-1)e_2^{10}\odot e_2^{10}\right\rangle\,;\;\;\;\;\;\;\;\;
%\end{align*}
%\item[2.] $\gm_t=\bigoplus\gm_t^{p}$ where $t\in\bR$ and 
%\begin{align*}
%\;\;\;\;\;\;\;\;\;\;\;\;\;\;\;\;\;\;\;\gm_{t}^{0(10)}&=\left\langle\phantom{C^{C^C}}\!\!\!\!\!\!\!\!\!\!\!\!\! e_1^{10}\odot e_1^{10}-e_2^{10}\odot e_2^{10}+2(t-1)e_1^{10}\odot e_2^{10}\right.\\
%&\left.\phantom{C^{C^C}}\!\!\!\!\!\!\!\!\!\!\!\!\!\;\;\;\;\;+i((1+t)e_1^{10}\odot e_1^{10}-(1+t)e_2^{10}\odot e_2^{10}-2e_1^{10}\odot e_2^{10})\right\rangle\,;
%\end{align*}
%\item[3.] $\gm_{\pm}=\bigoplus\gm_{\pm}^{p}$ where 
%\begin{align*}
%\;\;\;\;\;\;\;\;\;\;\;\;\;\;\;\;\;\gm_{\pm}^{0(10)}&=\left\langle\phantom{C^{C^C}}\!\!\!\!\!\!\!\!\!\!\!\!\! (1\pm 1)e_1^{10}\odot e_1^{10}+(-1\pm 1)e_2^{10}\odot e_2^{10}-2i e_1^{10}\odot e_2^{10}\right\rangle\,;
%\end{align*}
%\item[4.] $\gm_{<0}=\bigoplus\gm_{<0}^{p}$ where 
%\begin{align*}
%\gm_{<0}^{0(10)}&=\left\langle\phantom{C^{C^C}}\!\!\!\!\!\!\!\!\!\!\!\!\! e_1^{10}\odot e_2^{10}\right\rangle\,;\;\;\;\;\;\;\;\;\;\;\;\;\;\;\;\;\;\;\;\;\;\;\;\;\;\;\;\;\;\;\;\;\;\;\;\;\;\;\;\;\;\;\;\;\;\,
%\end{align*}
%\item[5.] $\gm_{o}=\bigoplus\gm_{o}^{p}$ where 
%\begin{align*}
%\gm_{o}^{0(10)}&=\left\langle\phantom{C^{C^C}}\!\!\!\!\!\!\!\!\!\!\!\!\! e_1^{10}\odot e_1^{10}-e_2^{10}\odot e_2^{10}-2i e_1^{10}\odot e_2^{10}\right\rangle\,;\;\;\;\;\;\;\;\;\,
%\end{align*}
%\end{itemize} 
\medskip\par
\section{Models for \texorpdfstring{$7$}{7}-dimensional \texorpdfstring{$2$}{2}-nondegenerate CR manifolds}
\label{exexamples}
\setcounter{equation}{0}
\setcounter{section}{5}
\smallskip\par
The main aim of this section is to prove the following.
\begin{theorem}
\label{mainI}
For each of the seven admissible cores $\gm=\bigoplus\gm^p$ of $\dim(\gm)=7$, $\operatorname{ht}(\gm)=0$ there exists an associated model $\gg=\bigoplus\gg^p$ of type $\gm$ and 
a $7$-dimensional $2$-nondegenerate homogeneous CR manifold $M=G/G_o$ which is globally defined.
Homogeneous CR manifolds associated with different $\gm$ are not, even locally, CR diffeomorphic one to the other.
\end{theorem}
%\smallskip
To prove it, we first need to briefly recall the description of the $\bZ$-gradings of the semisimple  Lie algebras (see e.g. \cite{GOV, Y, D}).

Let $\gg$ be a complex semisimple Lie algebra. Fix a Cartan subalgebra $\mathfrak{h}\subset\gg$ and denote by $\Delta=\Delta(\gg,\mathfrak{h})$ the root system and by 
\begin{equation*}
%\label{rootspace}
\gg^\a=\left\{X\in\gg\;|\;[H,X]=\a(H)X\;\;\text{for all}\;\; H\in \gh\phantom{C^{C^C}}\!\!\!\!\!\!\!\!\!\!\!\right\}
\end{equation*}
the associated root space of $\a\in\Delta$.
Let
$\mathfrak{h}_{\bR}\subset\mathfrak{h}$ be the real subspace where all the roots are real valued; %Any integral weight $\lambda\in\mathfrak{c}_\R^*$ defines a gradation on $\g$ by setting:
any element $\lambda\in(\mathfrak{h}_\bR^*)^*\simeq\mathfrak{h}_\bR$ with $\lambda(\alpha)\in\mathbb{Z}$ for all $\alpha\in\Delta$ defines a grading $\displaystyle\gg=\bigoplus\gg^p$ on $\gg$ by setting:
\begin{align*}
\begin{cases}
\gg^0=\mathfrak{h}\oplus
\displaystyle\sum_{\substack{\alpha\in\Delta\\\lambda(\alpha)=0}}\gg^\alpha,\\
\gg^p=\displaystyle\sum_{\substack{\alpha\in\Delta\\\lambda(\alpha)=p}}\gg^\alpha, &\text{for all}\,\; p\in\bZ^\times,
\end{cases}
\end{align*}
and all possible gradings of $\gg$ are of this form, for some choice of $\mathfrak{h}$ and $\lambda$. We will refer to $\lambda(\alpha)$ as the \emph{degree} of the root $\alpha$.

There exists a set of positive roots $\Delta^+\subset\Delta$ such that $\lambda$ is dominant, i.e., $\lambda(\alpha)\geq0$ for all $\alpha\in\Delta^+$. 
The depth of $\gg$ is the degree of the maximal root and is also equal to the height of $\gg$.
Let $\Pi$ be the set of positive simple roots, which we identify with the nodes of the Dynkin diagram. A grading is fundamental if and only if $\lambda(\alpha)\in\{0,1\}$ for all $\alpha\in\Pi$. We denote a fundamental grading of a Lie algebra $\gg$ by marking with a cross the nodes of the Dynkin diagram of $\gg$ corresponding to simple roots $\alpha$ with $\lambda(\alpha)=1$.

The Lie subalgebra $\gg^0$ is reductive; the Dynkin diagram of its semisimple ideal is obtained from the Dynkin diagram of $\gg$ by removing all crossed nodes, and any line issuing from them.

A routine examination of the Dynkin diagrams of complex Lie algebras together with e.g. \cite[Prop. 3.2.4]{CS} implies that Table \ref{tabella} below lists all the complex graded semisimple Lie algebras $\displaystyle\gg=\bigoplus\gg^p$ which satisfy:
\begin{enumerate}
\item[(i)] the depth $\operatorname{d}(\gg)=2$,
\item[(ii)]
$\dim \gg^{-2}=1$,
\item[(iii)] 
$\dim \gg^{-1}=4$;
\item[(iv)]
$\gg_-$ is nondegenerate. %i.e. if $v\in\gg^{-1}$ satisfies $[v,\gg^{-1}]=(0)$ then $v=0$.
\end{enumerate}
Any such $\gg$ is a simple graded subalgebra of the complex contact algebra $\widehat\gc$ of degree $n=2$.
%We also note that $\gg$ is simple in all cases.
\par\noindent
%\vfill
\begin{centering}
\begin{table}[H]
\begin{tabular}{|c|c|c|c|c|}
\hline
\hline
$\gg$ & \text{Grading}  & $\gg^{-2^{\phantom{C^{C^C}}}}\!\!\!\!\!\!\!\!\!\!$ & $\gg^{-1}$ & $\gg^0$\\
\hline
\hline
&&&&\\[-3mm]
$A_{3}$&
\begin{tikzpicture}
%\node[root]   (1)                     {};
\node[xroot] (1) {}; %[right=of 1] {} edge [-] (1);
\node[root]   (2) [right=of 1] {} edge [-] (1);
\node[xroot]   (3) [right=of 2] {} edge [-] (2);
%\node[root]   (5) [right=of 4] {} edge [-] (4);
\end{tikzpicture}&
$\bC$ & $\bC^2\oplus(\bC^2)^*$ & $\ggl_2(\bC)\oplus\bC$
\\
\hline
&&&&\\[-3mm]
%$B_4$&
%\begin{tikzpicture}
%\node[root]   (1)                     {};
%\node[root] (2) [right=of 1] {} edge [-] (1);
%\node[root]   (3) [right=of 2] {} edge [-] (2);
%\node[xroot]   (4) [right=of 3] {} edge [rdoublearrow] (3);
%\end{tikzpicture}&
%$6$ & $4$ & $0$
%\\
%\hline

%$B_{\ell}$&
%\begin{tikzpicture}
%\node[root]   (1)                     {};
%\node[root] (2) [right=of 1] {} edge [-] (1);
%\node[root]   (3) [right=of 2] {} edge [-] (2);
%\node[xroot]   (4) [right=of 3] {} edge [-] (3);
%\node[]   (5) [right=of 4] {$\;\cdots\,$} edge [-] (4);
%%\node[root]   (6) [right=of 5] {} edge [-] (5);
%\node[root]   (7) [right=of 5] {} edge [rdoublearrow] (5);
%\end{tikzpicture}&
%$6$ & $8\ell-28$ & $B_{\ell-4}$
%\\
%\hline 

$C_{3}$&
\begin{tikzpicture}
\node[xroot]   (1)                     {};
\node[root] (2) [right=of 1] {} edge [-] (1);
%\node[]   (3) [right=of 2] {} edge [-] (2);
%\node[root]   (4) [right=of 3] {} edge [-] (3);
\node[root]   (3) [right=of 2] {} edge [doublearrow] (2);
\end{tikzpicture}&
$\bC$ & $\bC^4$ & $\mathfrak{sp}_4(\bC)\oplus\bC$
\\
\hline
&&&&\\[-3mm]
$G_{2}$&
\begin{tikzpicture}
\node[root]   (1)                     {};
\node[xroot] (2) [right=of 1] {} edge [doublearrow] (1) edge [-] (1);
%\node[]   (3) [right=of 2] {} edge [-] (2);
%\node[root]   (4) [right=of 3] {} edge [-] (3);
%\node[root]   (3) [right=of 3] {} edge [doublearrow] (2);
\end{tikzpicture}&
$\bC$ & $S^3\bC^2$ & $\mathfrak{sl}_2(\bC)\oplus \bC$
\\
\hline

\hline 
\hline
\end{tabular}
\caption{Some $\bZ$-gradings $\displaystyle \gg=\bigoplus\gg^p$ of complex simple Lie algebras with the associated representation of $\gg^0$ on $\displaystyle \gg_-$.}
\label{tabella}
\end{table}
\end{centering}
\par\noindent
We now recall the description of simple real graded Lie algebras. Let $\gg$ be a real simple Lie algebra. Fix a Cartan decomposition $\gg=\mathfrak{k}\oplus\mathfrak{p}$, a maximal abelian subspace $\gh_{\circ}\subset\mathfrak p$ and a maximal torus $\gh_\bullet$ in the centralizer of $\gh_\circ$ in $\mathfrak k$. Then $\gh=\gh_\bullet\oplus\gh_\circ$ is a maximally noncompact Cartan subalgebra of $\gg$.

Denote by $\Delta=\Delta(\wh\gg,\wh\gh)$ the root system of $\wh\gg$ and by $\gh_{\bR}=\mathrm{i}\gh_\bullet\oplus\gh_\circ\subset\wh\gh$ the real subspace where all the roots have real values. 
Conjugation $\sigma:\wh\gg\longrightarrow\wh\gg$ of $\wh\gg$ with respect to the real form $\gg$ leaves $\wh\gh$ invariant and induces an %isometric 
involution $\alpha\mapsto\bar\alpha$ on $\gh_\bR^*$, transforming roots into roots.
We say that a root $\alpha$ is compact if $\bar\alpha=-\alpha$ and denote by $\Delta_\bullet$ the set of compact roots.
There exists a set of positive roots $\Delta^+\subset\Delta$, with corresponding system of simple roots $\Pi$, and an involutive automorphism $\varepsilon\colon\Pi\to\Pi$ of the Dynkin diagram of $\wh\gg$ such that $\epsilon(\Pi\setminus\Delta_\bullet)\subseteq \Pi\setminus\Delta_\bullet$ and
\begin{align*}
\bar\alpha&=-\alpha\;\;\; \text{ for all }\alpha\in\Pi\cap\Delta_\bullet \ ,\\
\bar\alpha&=\varepsilon(\alpha)+\!\sum_{\beta\in\Pi\cap\Delta_\bullet}b_{\alpha,\beta}\beta\;\;\;\text{ for all }\alpha\in\Pi\setminus\Delta_\bullet\ .
\end{align*}
The Satake diagram of $\gg$ is the Dynkin diagram of $\wh\gg$ with the following additional information:
\begin{enumerate}
\item
nodes in $\Pi\cap\Delta_\bullet$ are painted black;
\item
if $\alpha\in\Pi\setminus\Delta_\bullet$ and $\varepsilon(\alpha)\neq\alpha$ then $\alpha$ and $\varepsilon(\alpha)$ are joined by a curved arrow.
\end{enumerate}
A list of Satake diagrams can be found in e.g. \cite{CS, GOV, O}.

Let $\lambda\in(\gh_\bR^*)^*\simeq\gh_\bR$ be %a dominant integral weight 
an element such that the induced grading on $\wh\gg$ is fundamental. Then the grading on $\wh\gg$ induces a grading on $\gg$ if and only if $\bar\lambda=\lambda$, or equivalently the following two conditions on the set 
$$\Phi=\left\{\alpha\in\Pi\mid\lambda(\alpha)=1\phantom{C^{C^C}}\!\!\!\!\!\!\!\!\!\!\!\!\right\}$$ are satisfied:
\begin{enumerate}
\item
$\Phi\cap\Delta_\bullet=\emptyset$;
\item
if $\alpha\in\Phi$ then $\varepsilon(\alpha)\in\Phi$.
\end{enumerate}
We indicate the grading on a real Lie algebra by crossing all nodes in $\Phi$. In the real case too the Lie subalgebra $\gg^0$ is reductive and the Satake diagram of its semisimple ideal is obtained from the Satake diagram of $\gg$. %by removing all crossed nodes, and any line issuing from them.

Table \ref{tabellona2} lists all the real graded simple Lie algebras
$\displaystyle\gg=\bigoplus\gg^p$  such that the induced grading on $\wh\gg$ is as in Table \ref{tabella}. 
\par\noindent
\begin{centering}
\begin{table}[H]
\begin{tabular}{|c|c|c|c|c|c|}
\hline
\hline
$\gg$ &  \text{Grading} & $\gg^{-2^{\phantom{C^{C^C}}}}\!\!\!\!\!\!\!\!\!\!$ & $\gg^{-1}$ & $\gg^0$\\
\hline
\hline
&&&&\\[-1mm]

$\sl_4(\bR)_{\phantom{C_{C_C}}}\!\!\!\!\!\!\!\!\!\!$&
\begin{tikzpicture}
\node[xroot]   (1)                     {};
\node[root] (2) [right=of 1] {} edge [-] (1);
\node[xroot]   (3) [right=of 2] {} edge [-] (2);
\end{tikzpicture}&
$\bR$ & $\bR^2\oplus(\bR^2)^*$ & $\ggl_2(\bR)\oplus\bR$
\\
\hline
&&&&\\[-1mm]

$\su(2,2)_{\phantom{C_{C_C}}}\!\!\!\!\!\!\!\!\!\!$&
\begin{tikzpicture}
\node[xroot]   (1)                     {}; 
\node[root] (2) [right=of 1] {} edge [-] (1);
\node[xroot]   (3) [right=of 2] {} edge [-] (2) edge [<->,out=145, in=35] (1);
\end{tikzpicture}&
$\bR$ & $\bR^2\oplus(\bR^2)^*$ & $\ggl_2(\bR)\oplus\bR$
\\
\hline

&&&&\\[-1mm]

$\su(1,3)_{\phantom{C_{C_C}}}\!\!\!\!\!\!\!\!\!\!$&
\begin{tikzpicture}
\node[xroot]   (1)                     {}; 
\node[broot] (2) [right=of 1] {} edge [-] (1);
\node[xroot]   (3) [right=of 2] {} edge [-] (2) edge [<->,out=145, in=35] (1);
\end{tikzpicture}&
$\bR$ & $\bR^2\oplus(\bR^2)^*$ & $\mathfrak{u}(2)\oplus\bR$
\\
\hline
&&&&\\[-1mm]

$\sp_{6}(\bR)_{\phantom{C_{C_C}}}\!\!\!\!\!\!\!\!\!\!$&
\begin{tikzpicture}
\node[xroot]   (1)                     {};
\node[root] (2) [right=of 1] {} edge [-] (1);
\node[root]   (3) [right=of 2] {} edge [doublearrow] (2);
\end{tikzpicture}&
$\bR$ & $\bR^4$ & $\sp_4(\bR)\oplus\bR$
\\
\hline

&&&&\\[-1mm]
$\begin{gathered} G_{2} \\ \text{split} \end{gathered}$&
\begin{tikzpicture}
\node[root]   (1)                     {};
\node[xroot] (2) [right=of 1] {} edge [doublearrow] (1) edge [-] (1);
%\node[]   (3) [right=of 2] {} edge [-] (2);
%\node[root]   (4) [right=of 3] {} edge [-] (3);
%\node[root]   (3) [right=of 3] {} edge [doublearrow] (2);
\end{tikzpicture}&
$\bR$ & $S^3\bR^2$ & $\sl_2(\bR)\oplus\bR$
\\
\hline
\hline
\end{tabular}
\caption{Some $\bZ$-gradings $\displaystyle \gg=\bigoplus\gg^p$ of real simple Lie algebras with the associated representation of $\gg^0$ on $\displaystyle \gg_-$.}
\label{tabellona2}
\end{table}
\end{centering}
\par\noindent
Let now $\gg=\bigoplus\gg^p$ be a model of type $\gm$ with $\dim(\gm)=7$ and $\operatorname{ht}(\gm)=0$ that in addition is a simple Lie algebra satisfying property $(\mathfrak J)$ (i.e., $\mathfrak J\in\gg^0$). By the discussion above the grading $\gg=\bigoplus\gg^p$ is necessarily as in Table \ref{tabellona2}. It is convenient to introduce a Cartan subalgebra which, in general, is not maximally noncompact. 
\begin{definition}
A Cartan subalgebra $\gh$ of $\gg$ is called {\it adapted}
if it contains the grading element $E$ of $\gg=\bigoplus\gg^p$ and $\mathfrak{J}$.
\end{definition}
Adapted Cartan subalgebras always exist since $E$ and $\mathfrak{J}$ are semisimple elements and $[E,\mathfrak J]=0$. Moreover any adapted Cartan subalgebra $\gh$ satisfies $\gh\subset\gg^0$ and decomposes into $\gh=\gh_{\bullet}\oplus\gh_{\circ}$ where
\begin{align*}
\gh_\bullet&=\left\{H\in\gh\;|\;\text{all eigenvalues of}\;\ad H\;\text{are purely imaginary}\phantom{C^{C^C}}\!\!\!\!\!\!\!\!\!\!\!\right\}\;,\\
\gh_\circ&=\left\{H\in\gh\;|\;\text{all eigenvalues of}\;\ad H\;\text{are real}\phantom{C^{C^C}}\!\!\!\!\!\!\!\!\!\!\!\right\}\;.
\end{align*}
Clearly $\mathfrak J\in\gh_\bullet$ and $E\in\gh_\circ$. %We again emphasize that $\gh$ {\it cannot be chosen} in general to be maximally noncompact. 
\emph{From now on $\gh$ is an adapted Cartan subalgebra.}

We fix a Hermitian form $h(t, s) = \overline{t}^t\cI\,s$ on $\bC^4$ 
and identify the associated special unitary Lie algebra $\gg$ with the Lie algebra of trace-free complex matrices satisfying $\overline{A}^t \cI + \cI A = 0$;
we follow the conventions of \cite{Su}:
\vskip-0.014cm\par\noindent
\begin{align*}
(i)\;\; \cI&= \left(
\begin{array}{c|ccc}1&0&0&0\\
\hline
0&-1&0&0\\
0&0&-1&0\\
0&0&0&-1\\
\end{array}
\right)\qquad\text{and}\\
\\
\gg&=\su(1,3)=\left\{\left(
\begin{array}{c|c} -\tr D & B\\
\hline
\overline{B}^{t}& D\\
\end{array}
\right)\;|\; B\in\operatorname {Mat}(1,3;\bC)\;\text{and}\;D\in\mathfrak{u}(3)\!\!\!\!\!\!\!\!\!\!\!\!\phantom{C^{C^C}}\right\}\;;\\
\\
(ii)\;\; 
\cI&= \left(
\begin{array}{cc|cc}1&0&0&0\\
0&1&0&0\\
\hline
0&0&-1&0\\
0&0&0&-1\\
\end{array}
\right)
\qquad\text{and}\\
\\
\gg&=\su(2,2)=\left\{
\left(
\begin{array}{c|c} A & B\\
\hline
\overline{B}^{t}& D\\
\end{array}
\right)
\;|\; B\in\operatorname {Mat}(2,2;\bC),\,A=\begin{pmatrix} u_1 & a \\ -\overline{a} & u_2 \end{pmatrix},\!\!\!\!\!\!\!\!\!\!\!\!\phantom{C^{C^C}}\right.
\\
&\phantom{ccccccccccccccccccc}
\left.
D=\begin{pmatrix} u_3 & d \\ -\overline{d} & u_4 \end{pmatrix}\;\text{where}\;u_i\in i\bR,\,\sum u_i=0 
\!\!\!\!\!\!\!\!\!\!\!\!\phantom{C^{C^C}}\right\}\;.
\end{align*}
%\beq \label{projcoor}Ê(t, s) = t^T\cI\,s\qquad  \text{with}\qquad \cI = \left(
%\begin{array}{cc|c|cc}0&0&0&0&1\\
%0&0&0&1&0\\
%\hline
%0&0&1&0&0\\
%\hline
%0&1&0&0&0\\
%1&0&0&0&0
%\end{array}
%\right)\,.\eeq
%The Lie algebra $ \so_{3,2}$  can be identified with the  Lie algebra of  real matrices such that  $A^T \cI + \cI A = 0$,  i.e.,  of the form 
%$$A =  \left(
%\begin{array}{c|c|c}
%\begin{matrix} a_1 & a_2 \\a_3 & a_4 \end{matrix} &  \begin{matrix}  a_5 \\a_6 \end{matrix} & \begin{matrix}  a_7 & 0 \\0 & -a_7 \end{matrix}\\
%\hline 
%\begin{matrix}   a_{8} &  a_{9} \end{matrix} & 0  & \begin{matrix}  -a_{6} & -a_5 \end{matrix}\\
 %\hline
%\begin{matrix} a_{10} & 0 \\0 & -a_{10} \end{matrix} &  \begin{matrix}  - a_{9}\\-a_{8} \end{matrix} & \begin{matrix}- a_4  & - a_2 \\- a_3 & - a_1 \end{matrix}
%\end{array}
%\right)\,,\qquad\text{for some}\, a_i \in \bR\,.$$
%
%
%the unique, up to conjugation, compact Cartan subalgebra $\gt_\bullet$ of $\gk^0$ is
%$$
%\gt_\bullet=\left\{\begin{pmatrix}0 & -D \\ D & 0\end{pmatrix}\,|\, D=\,\text{diagonal matrix of order}\,n\right\}\ 
%$$
%(in the conventions of \cite{Su}, this is the standard Cartan subalgebra of characteristic $(n,0)$).
%Note that for any two non-negative integers $r$ and $s$ with $n=r+s$, $\mathfrak{J}\in\gt_\bullet$ where
%\beq
%\label{complexstru}
%\gJ=\begin{pmatrix}
%0 & I_{r,s}\\
%-I_{r,s} & 0\\
%\end{pmatrix}\;,\;\;\;\text{where}\;\;\;\;I_{r,s}=
%\begin{pmatrix}
%\Id_r & 0 \\
%0 & -\Id_s \\
%\end{pmatrix}\,;
%\eeq
\begin{theorem}
\label{thmsimple}
Let $\gg$ be a model of type $\gm$, $\dim(\gm)=7$ and $\operatorname{ht}(\gm)=0$, which in addition is a semisimple real Lie algebra and satisfies property $(\mathfrak J)$. Then:
\begin{itemize}
\item[(i)] $\gg$ is simple and isomorphic to $\sl_4(\bR)$, $\su(1,3)$ or $\su(2,2)$ with the $\bZ$-grading $\gg=\bigoplus\gg^p$ given in Table \ref{tabellona2};
\item[(ii)] the complex structure $\mathfrak J$ is described in Table \ref{tabellona3} below;
\item[(iii)] there exists  always  an associated $7$-dimensional and $2$-nondegenerate homogeneous CR manifold $M=G/G_o$ which is globally defined (i.e., $G_o$ is closed in $G$). It is of type $\gm$ where the signature $\sgn(\mathfrak{J})$ of $\mathfrak J$ and the equivalence class $[\gm^{0(10)}]$ of $\gm^{0(10)}$ are (recall Tables \ref{tablecore20} and \ref{simbacacca}):
\par\noindent
\vskip-0.2cm
\begin{centering}
\begin{table}[H]
\begin{tabular}{|c|c|c|c|}
\hline
\hline
$\gg^{\phantom{C^{C^C}}}\!\!\!\!\!\!\!\!\!\!$ &  $\operatorname{sgn}(\mathfrak J)$ & $[\gm^{0(10)}]$ & $\gn_{\sharp_{\phantom{C_{C_C}}}}^{\phantom{C^{C^C}}}\!\!\!\!\!\!\!\!\!\!$ \\
\hline
\hline
&&&\\[-1mm]
$\sl_4(\bR)$ & $(1,1)_{\phantom{C_{C_C}}}^{\phantom{C^{C^C}}}\!\!\!\!\!\!\!\!\!\!$ & $e_1^{10}\odot e_1^{10}-e_2^{10}\odot e_2^{10}$
&
$\bC\oplus\so(1,1)$
\\
\hline 
&&&\\[-1mm]
$\su(1,3)$
 & $(1,1)_{\phantom{C_{C_C}}}^{\phantom{C^{C^C}}}\!\!\!\!\!\!\!\!\!\!$ & $e_1^{10}\odot e_2^{10}$
&
$\bC\oplus\so_2(\bR)$
\\
\hline 
&&&\\[-1mm]
$\su(2,2)$
 & $(2,0)_{\phantom{C_{C_C}}}^{\phantom{C^{C^C}}}\!\!\!\!\!\!\!\!\!\!$ & $e_1^{10}\odot e_1^{10}+e_2^{10}\odot e_2^{10}$
&
$\bC\oplus\so_2(\bR)$
\\
\hline 
\hline
\end{tabular}
\caption{}
\label{tabellatype}
\end{table}
\end{centering}
\item[(iv)] the adapted Cartan subalgebra $\gh$ of $\gg$ equals $\gn_\sharp$ in all cases (recall Definition \ref{prosciuttocotto});
\item[(v)] the model $\gg$ is maximal;
\item[(vi)] for completeness we also give the terms \eqref{bologna} of the Freeman sequence of the associated CR algebra $(\gg,\gq)$. They are $\gq=\gq_{-1}=\wh{\gg}^{2}+\wh{\gg}^{1}+\wh\gh+\gg^{-\alpha_2}+\gg^{-(\alpha_1+\alpha_2)}+\gg^{-(\alpha_2+\alpha_3)}$ and
$$\gq_{0}=\wh{\gg}^{2}\oplus\wh{\gg}^{1}\oplus\wh\gh\oplus\gg^{-\alpha_2}\;,\qquad\gq_{1}=\wh{\gg}^{2}\oplus\wh{\gg}^{1}\oplus\wh\gh=\gq\cap\overline\gq\;,$$
with $\alpha_i$ the simple root associated to the $i$-th node of the Dynkin diagram of $\sl_4(\bC)$.
\end{itemize} 
\end{theorem}
\!\!
\begin{sidewaystable}
\centering
\vskip15cm
%\hskip-2.6cm
\resizebox{\linewidth}{!}{
%{\footnotesize
\begin{tabular}{|c|c|c|c|c|c|c|c|c|c|c|c|}
\hline
$\begin{gathered} \\ \\ \\ \gg \end{gathered}$ & $\begin{gathered}\\ \\ \\ \sgn(\mathfrak J)\end{gathered}$ & $\begin{gathered}\\ \\ \\ E^2 \end{gathered}$ & $E_1^{10}$ & $E_2^{10}$ & $\begin{gathered}\\ \\ \\ E\end{gathered}$ & $\begin{gathered}\\ \\ \\ \mathfrak{J}\end{gathered}$ & $\begin{gathered}\\ \\ \\ H \end{gathered}$ & $e^{0(10)}$& $e_1^{10}$ & $e_2^{10}$ & $\begin{gathered}\\ \\ \\ e^{-2}\end{gathered}$  \\
\cline{4-5}\cline{9-11}
&&&&&&&&&&&\\[-3mm]
  & & & $\begin{gathered} \\ E_{1}^{01}\\ \\ \end{gathered}$ & $\begin{gathered}\\ E_2^{01} \\ \\ \end{gathered}$ & & & & $\begin{gathered}\\ e^{0(01)} \\ \\ \end{gathered}$& $\begin{gathered}\\ e_1^{01} \\ \\ \end{gathered}$ & $\begin{gathered}\\ e_2^{01} \\ \\ \end{gathered}$ & \\
\hline
\hline
&&&&&&&&&&&\\[-3mm]

$\begin{gathered} \\ \\ \\ 
\mathfrak{sl}_4(\bR) 
\end{gathered}$ &
$\begin{gathered} \\ \\ \\ 
(1,1) 
\end{gathered}$ & 
$\begin{gathered}\\ \\ \\ 
4\begin{psmallmatrix} 
0 & 0 & 0 & 0\\
0 & 0 & 0 & 0\\
0 & 0 & 0 & -1\\
0 & 0 & 0 & 0\\
\end{psmallmatrix}
\end{gathered}$ &
$\begin{gathered} 
\begin{psmallmatrix} 
0 & 0 & 0 & i\\
0 & 0 & 0 & 1\\
0 & 0 & 0 & 0\\
0 & 0 & 0 & 0\\
\end{psmallmatrix}
\\ \end{gathered}$ & 
$\begin{gathered} 
\begin{psmallmatrix} 
0 & 0 & 0 & 0\\
0 & 0 & 0 & 0\\
i & 1 & 0 & 0\\
0 & 0 & 0 & 0\\
\end{psmallmatrix}
\\ \end{gathered}$ & 
$\begin{gathered} \\ \\ \\ 
\begin{psmallmatrix} 
0 & 0 & 0 & 0\\
0 & 0 & 0 & 0\\
0 & 0 & 1 & 0\\
0 & 0 & 0 & -1\\
\end{psmallmatrix}
\end{gathered}$ & 
$\begin{gathered} \\ \\ \\ 
\begin{psmallmatrix} 
0 & -1 & 0 & 0\\
1 & 0 & 0 & 0\\
0 & 0 & 0 & 0\\
0 & 0 & 0 & 0\\
\end{psmallmatrix}
\end{gathered}$ & 
$\begin{gathered} \\ \\ \\ 
\frac{1}{2}\begin{psmallmatrix} 
1 & 0 & 0 & 0\\
0 & 1 & 0 & 0\\
0 & 0 & -1 & 0\\
0 & 0 & 0 & -1\\
\end{psmallmatrix}
\end{gathered}$ & 
$\begin{gathered} 
\begin{psmallmatrix} 
1 & -i & 0 & 0\\
-i & -1 & 0 & 0\\
0 & 0 & 0 & 0\\
0 & 0 & 0 & 0\\
\end{psmallmatrix}
\\ \end{gathered}$ & 
$\begin{gathered} 
\begin{psmallmatrix} 
0 & 0 & i & 0\\
0 & 0 & 1 & 0\\
0 & 0 & 0 & 0\\
0 & 0 & 0 & 0\\
\end{psmallmatrix}
\\ \end{gathered}$ & 
$\begin{gathered} 
\begin{psmallmatrix} 
0 & 0 & 0 & 0\\
0 & 0 & 0 & 0\\
0 & 0 & 0 & 0\\
-1 & i & 0 & 0\\
\end{psmallmatrix}
\\ \end{gathered}$ &  
$\begin{gathered} \\ \\ \\ 
4\begin{psmallmatrix} 
0 & 0 & 0 & 0\\
0 & 0 & 0 & 0\\
0 & 0 & 0 & 0\\
0 & 0 & -1 & 0\\
\end{psmallmatrix}
\end{gathered}$
\\
&&&&&&&&&&&\\[-3mm]

& & &
$\begin{gathered} 
\begin{psmallmatrix} 
0 & 0 & 0 & -i\\
0 & 0 & 0 & 1\\
0 & 0 & 0 & 0\\
0 & 0 & 0 & 0\\
\end{psmallmatrix}
\\ \\ \end{gathered}$ & 
$\begin{gathered} 
\begin{psmallmatrix} 
0 & 0 & 0 & 0\\
0 & 0 & 0 & 0\\
-i & 1 & 0 & 0\\
0 & 0 & 0 & 0\\
\end{psmallmatrix}
\\ \\ \end{gathered}$ & & & & 
$\begin{gathered} 
\begin{psmallmatrix} 
1 & i & 0 & 0\\
i & -1 & 0 & 0\\
0 & 0 & 0 & 0\\
0 & 0 & 0 & 0\\
\end{psmallmatrix}
\\ \\ \end{gathered}$ & 
$\begin{gathered} 
\begin{psmallmatrix} 
0 & 0 & -i & 0\\
0 & 0 & 1 & 0\\
0 & 0 & 0 & 0\\
0 & 0 & 0 & 0\\
\end{psmallmatrix}
\\ \\ \end{gathered}$ & 
$\begin{gathered} 
\begin{psmallmatrix} 
0 & 0 & 0 & 0\\
0 & 0 & 0 & 0\\
0 & 0 & 0 & 0\\
-1 & -i & 0 & 0\\
\end{psmallmatrix}
\\ \\ \end{gathered}$ & 
\\
\hline
&&&&&&&&&&&\\[-3mm]

$\begin{gathered} \\ \\ \\ 
\mathfrak{su}(1,3) 
\end{gathered}$ &
$\begin{gathered} \\ \\ \\ 
(1,1) 
\end{gathered}$ & 
$\begin{gathered}\\ \\ \\ 
2\begin{psmallmatrix} 
i & -i & 0 & 0\\
i & -i & 0 & 0\\
0 & 0 & 0 & 0\\
0 & 0 & 0 & 0\\
\end{psmallmatrix}
\end{gathered}$ &
$\begin{gathered} 
\begin{psmallmatrix} 
0 & 0 & 0 & 0\\
0 & 0 & 0 & 0\\
1 & -1 & 0 & 0\\
0 & 0 & 0 & 0\\
\end{psmallmatrix}
\\ \end{gathered}$ & 
$\begin{gathered} 
\begin{psmallmatrix} 
0 & 0 & 0 & 1\\
0 & 0 & 0 & 1\\
0 & 0 & 0 & 0\\
0 & 0 & 0 & 0\\
\end{psmallmatrix}
\\ \end{gathered}$ & 
$\begin{gathered} \\ \\ \\ 
\begin{psmallmatrix} 
0 & 1 & 0 & 0\\
1 & 0 & 0 & 0\\
0 & 0 & 0 & 0\\
0 & 0 & 0 & 0\\
\end{psmallmatrix}
\end{gathered}$ & 
$\begin{gathered} \\ \\ \\ 
\begin{psmallmatrix} 
0 & 0 & 0 & 0\\
0 & 0 & 0 & 0\\
0 & 0 & i & 0\\
0 & 0 & 0 & -i\\
\end{psmallmatrix}
\end{gathered}$ & 
$\begin{gathered} \\ \\ \\ 
\frac{1}{2}\begin{psmallmatrix} 
i & 0 & 0 & 0\\
0 & i & 0 & 0\\
0 & 0 & -i & 0\\
0 & 0 & 0 & -i\\
\end{psmallmatrix}
\end{gathered}$ & 
$\begin{gathered} 
\begin{psmallmatrix} 
0 & 0 & 0 & 0\\
0 & 0 & 0 & 0\\
0 & 0 & 0 & 1\\
0 & 0 & 0 & 0\\
\end{psmallmatrix}
\\ \end{gathered}$ & 
$\begin{gathered} 
\begin{psmallmatrix} 
0 & 0 & 0 & 1\\
0 & 0 & 0 & -1\\
0 & 0 & 0 & 0\\
0 & 0 & 0 & 0\\
\end{psmallmatrix}
\\ \end{gathered}$ & 
$\begin{gathered} 
\begin{psmallmatrix} 
0 & 0 & 0 & 0\\
0 & 0 & 0 & 0\\
1 & 1 & 0 & 0\\
0 & 0 & 0 & 0\\
\end{psmallmatrix}
\\ \end{gathered}$ &  
$\begin{gathered} \\ \\ \\ 
2\begin{psmallmatrix} 
i & i & 0 & 0\\
-i & -i & 0 & 0\\
0 & 0 & 0 & 0\\
0 & 0 & 0 & 0\\
\end{psmallmatrix}
\end{gathered}$
\\
&&&&&&&&&&&\\[-3mm]

& & &
$\begin{gathered} 
\begin{psmallmatrix} 
0 & 0 & 1 & 0\\
0 & 0 & 1 & 0\\
0 & 0 & 0 & 0\\
0 & 0 & 0 & 0\\
\end{psmallmatrix}
\\ \\ \end{gathered}$ & 
$\begin{gathered} 
\begin{psmallmatrix} 
0 & 0 & 0 & 0\\
0 & 0 & 0 & 0\\
0 & 0 & 0 & 0\\
1 & -1 & 0 & 0\\
\end{psmallmatrix}
\\ \\ \end{gathered}$ & & & & 
$\begin{gathered} 
\begin{psmallmatrix} 
0 & 0 & 0 & 0\\
0 & 0 & 0 & 0\\
0 & 0 & 0 & 0\\
0 & 0 & -1 & 0\\
\end{psmallmatrix}
\\ \\ \end{gathered}$ & 
$\begin{gathered} 
\begin{psmallmatrix} 
0 & 0 & 0 & 0\\
0 & 0 & 0 & 0\\
0 & 0 & 0 & 0\\
1 & 1 & 0 & 0\\
\end{psmallmatrix}
\\ \\ \end{gathered}$ & 
$\begin{gathered} 
\begin{psmallmatrix} 
0 & 0 & 1 & 0\\
0 & 0 & -1 & 0\\
0 & 0 & 0 & 0\\
0 & 0 & 0 & 0\\
\end{psmallmatrix}
\\ \\ \end{gathered}$ & 
\\
\hline
&&&&&&&&&&&\\[-3mm]

$\begin{gathered} \\ \\ \\ 
\mathfrak{su}(2,2) 
\end{gathered}$ &
$\begin{gathered} \\ \\ \\ 
(2,0) 
\end{gathered}$ & 
$\begin{gathered}\\ \\ \\ 
2\begin{psmallmatrix} 
i & 0 & -i & 0\\
0 & 0 & 0 & 0\\
i & 0 & -i & 0\\
0 & 0 & 0 & 0\\
\end{psmallmatrix}
\end{gathered}$ &
$\begin{gathered} 
\begin{psmallmatrix} 
0 & 0 & 0 & 0\\
1 & 0 & -1 & 0\\
0 & 0 & 0 & 0\\
0 & 0 & 0 & 0\\
\end{psmallmatrix}
\\ \end{gathered}$ & 
$\begin{gathered} 
\begin{psmallmatrix} 
0 & 0 & 0 & 1\\
0 & 0 & 0 & 0\\
0 & 0 & 0 & 1\\
0 & 0 & 0 & 0\\
\end{psmallmatrix}
\\ \end{gathered}$ & 
$\begin{gathered} \\ \\ \\ 
\begin{psmallmatrix} 
0 & 0 & 1 & 0\\
0 & 0 & 0 & 0\\
1 & 0 & 0 & 0\\
0 & 0 & 0 & 0\\
\end{psmallmatrix}
\end{gathered}$ & 
$\begin{gathered} \\ \\ \\ 
\begin{psmallmatrix} 
0 & 0 & 0 & 0\\
0 & i & 0 & 0\\
0 & 0 & 0 & 0\\
0 & 0 & 0 & -i\\
\end{psmallmatrix}
\end{gathered}$ & 
$\begin{gathered} \\ \\ \\ 
\frac{1}{2}\begin{psmallmatrix} 
i & 0 & 0 & 0\\
0 & -i & 0 & 0\\
0 & 0 & i & 0\\
0 & 0 & 0 & -i\\
\end{psmallmatrix}
\end{gathered}$ & 
$\begin{gathered} 
\begin{psmallmatrix} 
0 & 0 & 0 & 0\\
0 & 0 & 0 & 1\\
0 & 0 & 0 & 0\\
0 & 0 & 0 & 0\\
\end{psmallmatrix}
\\ \end{gathered}$ & 
$\begin{gathered} 
\begin{psmallmatrix} 
0 & 0 & 0 & 1\\
0 & 0 & 0 & 0\\
0 & 0 & 0 & -1\\
0 & 0 & 0 & 0\\
\end{psmallmatrix}
\\ \end{gathered}$ & 
$\begin{gathered} 
\begin{psmallmatrix} 
0 & 0 & 0 & 0\\
1 & 0 & 1 & 0\\
0 & 0 & 0 & 0\\
0 & 0 & 0 & 0\\
\end{psmallmatrix}
\\ \end{gathered}$ &  
$\begin{gathered} \\ \\ \\ 
2\begin{psmallmatrix} 
i & 0 & i & 0\\
0 & 0 & 0 & 0\\
-i & 0 & -i & 0\\
0 & 0 & 0 & 0\\
\end{psmallmatrix}
\end{gathered}$
\\
&&&&&&&&&&&\\[-3mm]

& & &
$\begin{gathered} 
\begin{psmallmatrix} 
0 & -1 & 0 & 0\\
0 & 0 & 0 & 0\\
0 & -1 & 0 & 0\\
0 & 0 & 0 & 0\\
\end{psmallmatrix}
\\ \\ \end{gathered}$ & 
$\begin{gathered} 
\begin{psmallmatrix} 
0 & 0 & 0 & 0\\
0 & 0 & 0 & 0\\
0 & 0 & 0 & 0\\
1 & 0 & -1 & 0\\
\end{psmallmatrix}
\\ \\ \end{gathered}$ & & & & 
$\begin{gathered} 
\begin{psmallmatrix} 
0 & 0 & 0 & 0\\
0 & 0 & 0 & 0\\
0 & 0 & 0 & 0\\
0 & 1 & 0 & 0\\
\end{psmallmatrix}
\\ \\ \end{gathered}$ & 
$\begin{gathered} 
\begin{psmallmatrix} 
0 & 0 & 0 & 0\\
0 & 0 & 0 & 0\\
0 & 0 & 0 & 0\\
1 & 0 & 1 & 0\\
\end{psmallmatrix}
\\ \\ \end{gathered}$ & 
$\begin{gathered} 
\begin{psmallmatrix} 
0 & -1 & 0 & 0\\
0 & 0 & 0 & 0\\
0 & 1 & 0 & 0\\
0 & 0 & 0 & 0\\
\end{psmallmatrix}
\\ \\ \end{gathered}$ & 
\\
\hline
\end{tabular}}
%}
\vskip0.7cm
\caption{The explicit decomposition of $\gg$ in terms of an adapted Cartan subalgebra $\gh=\left\{E,\mathfrak{J},H\right\}$, a basis $\left\{e_1,\ldots,e_4\right\}$ of $\gg^{-1}$ which is complex-symplectic for $\gg=\su(1,3)$, $\su(2,2)$ and complex-Witt for $\gg=\sl_4(\bR)$, %the associated elements $e_{i}^{10}=\frac{1}{2}(e_{i}-i\ad(\gJ)e_{i})\in\gg^{-1(10)}$, 
and, finally, a generator $e^{0(10)}$ of $\gm^{0(10)}$ and a basis of elements ``$E$'s'' of the positive part $\gg^{2}\bigoplus\gg^1$ of $\gg$.}
\label{tabellona3}
\end{sidewaystable}
%\begin{remark}
%Table \ref{tabellona3} gives
%\end{remark}
\begin{proof}
We first show by contradiction that $\gg$ is not $\sp_6(\bR)$ or $G_2$ split.

If $\gg=\sp_6(\bR)$ then $\gg^0=\gc^0\supset\Re(S^{2,0}\oplus S^{0,2})$ and $\dim(\gm)\geq 11$, a contradiction. Let $\gg$ be $G_2$ split and $\gh=\gh_\bullet\oplus\gh_\circ$ an adapted Cartan subalgebra of $\gg$ with associated root system $\Delta=\Delta(\wh\gg,\wh\gh)$;
each root space $\gg^\a\subset\wh\gg$ is in particular  $\mathfrak J$-stable.
Let also $\Delta^+\subset \Delta$ be a set of positive roots such that $\l$ is dominant and $\Pi=\left\{\a_1,\a_2\right\}$ the associated system of simple roots satisfying $\l(\a_1)=0$ and $\l(\a_2)=1$ (recall Table \ref{tabella}). 

The roots $\a\in\Delta$ with $\l(\a)=-1$ are $-\a_2,-\a_1-\a_2$, $-3\a_1-\a_2$, $-2\a_1-\a_2$
and the unique root with $\l(\a)=-2$ is $\a=-3\a_1-2\a_2$. By nondegeneracy of $\wh\gg_-$, two possibilities may occur in $\wh\gg^{-1}=\gg^{-1(10)}\oplus\gg^{-1(01)}$: 
%up to conjugation: %they are displayed in Table \ref{tabellaposs}.
\par\noindent
%\vfill
\begin{centering}
\begin{table}[H]
\begin{tabular}{|c|c|c|}
\hline
\hline
$\gg^{-1(10)^{\phantom{C^{C^C}}}}\!\!\!\!\!\!\!\!\!\!$  & $\gg^{-1(01)^{\phantom{C^{C^C}}}}\!\!\!\!\!\!\!\!\!\!$ & $\wh\gg^{0}$ \\
\hline
\hline
&&\\[-3mm]
$\begin{gathered}\\ \gg^{-\a_2}\oplus\gg^{-(\a_1+\a_2)}\\ \\ \end{gathered}$&
$\gg^{-(3\a_1+\a_2)}\oplus\gg^{-(2\a_1+\a_2)}$ & $\wh\gh\oplus\gg^{\a_1}\oplus\gg^{-\a_1}$
\\
\hline
&&\\[-3mm]
$\begin{gathered}\\ \gg^{-\a_2}\oplus\gg^{-(2\a_1+\a_2)}\\ \\  \end{gathered}$
&
$\gg^{-(3\a_1+\a_2)}\oplus\gg^{-(\a_1+\a_2)}$ & $\wh\gh\oplus\gg^{\a_1}\oplus\gg^{-\a_1}$
\\
\hline 
\hline
\end{tabular}
\caption{}
\label{tabellaposs}
\end{table}
\end{centering}
%\begin{itemize}
%\item $\a(\mathfrak{J})=i$ for $\a=-\a_2$ and $\a=-\a_1-\a_2$ whereas $\alpha(\mathfrak J)=-i$ for $\a=-3\a_1-\a_2$ and $\a=-2\a_1-\a_2$;
%\item $\a(\mathfrak{J})=i$ for $\a=-\a_2$ and $\a=-2\a_1-\a_2$ whereas $\alpha(\mathfrak J)=-i$ for $\a=-\a_1-\a_2$ and $\a=-3\a_1-\a_2$.
%\end{itemize}
\par\noindent
A direct inspection of the adjoint action of $\gg^{\a_1}$ on $\gg^{-1(10)}$ and $\gg^{-1(01)}$ yields  that
$\gg^{\a_1}$ is not $\mathfrak J$-stable in both cases, a contradiction.
\smallskip\par
Let now $\gg$ be $\sl_4(\bR)$, $\su(1,3)$ or $\su(2,2)$, $\gh=\gh_\bullet\oplus\gh_\circ$ an adapted Cartan subalgebra with root system $\Delta$ and associated system $\Pi=\left\{\a_1,\a_2,\a_3\right\}$ of simple roots satisfying $\l(\a_1)=\l(\a_3)=1$ and $\l(\a_2)=0$. 

Since $\l(\a_1+\a_2+\a_3)=2$, two possibilities occur in $\wh\gg^{-1}=\gg^{-1(10)}\oplus\gg^{-1(01)}$: %up to conjugation:
\par\noindent 
\begin{centering}
\begin{table}[H]
\begin{tabular}{|c|c|c|}
\hline
\hline
$\gg^{-1(10)^{\phantom{C^{C^C}}}}\!\!\!\!\!\!\!\!\!\!$  & $\gg^{-1(01)^{\phantom{C^{C^C}}}}\!\!\!\!\!\!\!\!\!\!$ & $\wh\gg^{0}$ \\
\hline
\hline
&&\\[-3mm]
$\begin{gathered}\\ \gg^{-(\a_1+\a_2)}\oplus\gg^{-(\a_2+\a_3)}\\ \\ \end{gathered}$
&
$\gg^{-\a_1}\oplus\gg^{-\a_3}$
 & $\wh\gh\oplus\gg^{\a_2}\oplus\gg^{-\a_2}$
\\
\hline
&&\\[-3mm]
$\begin{gathered}\\ \gg^{-\a_3}\oplus\gg^{-(\a_2+\a_3)}\\ \\ \end{gathered}$
&
$\gg^{-\a_1}\oplus\gg^{-(\a_1+\a_2)}$
 & $\wh\gh\oplus\gg^{\a_2}\oplus\gg^{-\a_2}$
\\
\hline 
\hline
\end{tabular}
\caption{}
\label{tabellaposs2}
\end{table}
\end{centering}
\par\noindent
%\begin{itemize}
%\item $\a(\mathfrak{J})=i$ for $\a=-\a_1$ and $\a=-\a_3$ whereas $\alpha(\mathfrak J)=-i$ for $\a=-\a_1-\a_2$ and $\a=-\a_2-\a_3$;
%\item $\a(\mathfrak{J})=i$ for $\a=-\a_1$ and $\a=-\a_1-\a_2$ whereas $\alpha(\mathfrak J)=-i$ for $\a=-\a_2-\a_3$ and $\a=-\a_3$.
%\end{itemize}
In the second case $\wh\gg^0=S^{1,1}\oplus\bC E$ and hence $\dim(\gm)=5$, a contradiction. 

In the first case $\gg^{-\a_2}\subset\mathfrak{M}^{0(10)}$ and $\gg^{\a_2}\subset\mathfrak{M}^{0(01)}$ and there is no contradiction. To check that this case does actually occur we need to explicitly identify the adapted Cartan subalgebras of $\su(1,3)$, $\su(2,2)$ and $\sl_4(\bR)$; we rely upon the classification, up to conjugation, of the Cartan subalgebras of the simple real Lie algebras given in \cite{Su}.

If $\gg=\su(1,3)$ there are two Cartan subalgebras, the compact one (clearly not adapted) and the subalgebra $\gh=\gh_{\bullet}\oplus\gh_{\circ}$ with $\dim \gh_{\bullet}=2$, $\dim \gh_{\circ}=1$:
\begin{align}
\notag
\begin{split}
\gh_{\bullet}&=\left\{H=\begin{pmatrix} u_1 & 0 & 0 & 0\\
0 & u_1 & 0 & 0\\
0 & 0 & u_3 & 0\\
0 & 0 & 0 & u_4\\
\end{pmatrix}\,|\,u_i\in i\bR\;\;\text{and}\;\; \tr H=0\right\}\;,\\
\gh_{\circ}&=
\left\{H=\begin{pmatrix} 0 & h_1 & 0 & 0\\
h_1 & 0 & 0 & 0\\
0 & 0 & 0 & 0\\
0 & 0 & 0 & 0\\
\end{pmatrix}\,|\,h_1\in \bR\right\}\;.
\end{split}
\end{align}
Let $\left\{E,\mathfrak J,H\right\}$ be the basis of $\gh$ given in Table \ref{tabellona3}; 
$E\in\gh_\circ$ and coincides with the grading element of the grading of $\su(1,3)$ in Table \ref{tabellona2}, 
the set $\left\{\mathfrak J,H\right\}$ is a basis of $\gh_\bullet$. Moreover $\ad(\mathfrak J)|_{\gg^{-1}}$ and $\ad(H)|_{\gg^{-1}}$ are two linearly independent and commuting complex structures on $\gg^{-1}$ but the eigenspace decomposition of $\ad(H)|_{\gg^{-1}}$ is as in the second case of Table \ref{tabellaposs2} and hence not permissible.
%hence, it follows that $\mathfrak{J}=\pm H_1$ or $\mathfrak{J}=\pm H_2$. 

It follows from this observation, Table \ref{tabellona3} and the decomposition of $\su(1,3)$ under the adjoint action of $\mathfrak{J}$ that
\beq
\label{scelta}
\begin{split}
&\gg^{-\alpha_1}=\left\langle e^{01}_1\right\rangle\;,\qquad\gg^{-\alpha_3}=\left\langle e^{01}_2\right\rangle\;,\\
&\gg^{-(\alpha_1+\a_2)}=\left\langle e^{10}_2\right\rangle\;,\qquad\gg^{-(\a_2+\alpha_3)}=\left\langle e^{10}_1\right\rangle\;,\\
&\gg^{-\alpha_2}=\gm^{0(10)}=\left\langle e^{0(10)}\right\rangle\;,\qquad \gg^{\alpha_2}=\gm^{0(01)}=\left\langle e^{0(01)}\right\rangle\;,
\end{split}
\eeq
and $\gg=\bigoplus\gg^p$ is a model of type $\gm$. %with $\dim(\gm)=7$ and $\operatorname{ht}(\gm)=0$.
%\smallskip\par

If $\gg=\su(2,2)$ there are three Cartan subalgebras, the compact one (not adapted), the one with
$\dim\gh_\bullet=1$ and $\dim\gh_\circ=2$ and finally the Cartan subalgebra
$\gh=\gh_{\bullet}\oplus\gh_{\circ}$ with $\dim \gh_{\bullet}=2$, $\dim \gh_{\circ}=1$:
\begin{align}
\label{sopra}
\begin{split}
\gh_{\bullet}&=\left\{H=\begin{pmatrix} u_1 & 0 & 0 & 0\\
0 & u_2 & 0 & 0\\
0 & 0 & u_1 & 0\\
0 & 0 & 0 & u_4\\
\end{pmatrix}\,|\,u_i\in i\bR\;\;\text{and}\;\; \tr H=0\right\}
\;,\\
\gh_{\circ}&=
\left\{H=\begin{pmatrix} 0 & 0 & h_1 & 0\\
0 & 0 & 0 & 0\\
h_1 & 0 & 0 & 0\\
0 & 0 & 0 & 0\\
\end{pmatrix}\,|\,h_1\in \bR\right\}
\;.
\end{split}
\end{align}
The Cartan subalgebra with $\dim\gh_\bullet=1$ is not adapted, since in that case $\gh_\bullet$ is generated by a semisimple element with eigenvalues $0,\pm i$. The subalgebra \eqref{sopra} is adapted and gives rise to a model, using \eqref{scelta}, Table \ref{tabellona3} and an argument analogous to the previous case.

If $\gg=\sl_{4}(\bR)$ there are three Cartan subalgebras, the vectorial one (not adapted), the one with $\dim\gh_\bullet=2$ and $\dim\gh_\circ=1$ and finally the Cartan subalgebra $\gh=\gh_{\bullet}\oplus\gh_{\circ}$ with $\dim\gh_\bullet=1$ and $\dim\gh_\circ=2$:
\begin{align}
\label{aggiuntina}
\begin{split}
\gh_{\bullet}&=\left\{H=\begin{pmatrix} 0 & - h_2 & 0 & 0\\
h_2 & 0 & 0 & 0\\
0 & 0 & 0 & 0\\
0 & 0 & 0 & 0\\
\end{pmatrix}\,|\, h_2\in \bR\right\}
\;,\\
\gh_{\circ}&=
\left\{H=\begin{pmatrix} h_1 & 0 & 0 & 0\\
0 & h_1 & 0 & 0\\
0 & 0 & h_3 & 0\\
0 & 0 & 0 & h_4\\
\end{pmatrix}\,|\,h_i\in\bR\;\text{and}\; \operatorname{tr}H=0\right\}
\;.
\end{split}
\end{align}
%Fix a basis of $\gc$ given by
%$$
%J:=\begin{pmatrix} 0 & - 1 & 0 & 0\\
%1 & 0 & 0 & 0\\
%0 & 0 & 0 & 0\\
%0 & 0 & 0 & 0\\
%\end{pmatrix}\;,\; E:=\begin{pmatrix} 0 & 0 & 0 & 0\\
%0 & 0 & 0 & 0\\
%0 & 0 & 1 & 0\\
%0 & 0 & 0 & -1\\
%\end{pmatrix}\;,\; T:=\begin{pmatrix} 1 & 0 & 0 & 0\\
%0 & 1 & 0 & 0\\
%0 & 0 & -1 & 0\\
%0 & 0 & 0 & -1\\
%\end{pmatrix}\ .
%$$
%As before $E$ is the grading element of a unique gradation of $\wh{\gg}=\sum_{-2\leq p\leq 2} \wh{\gg}^{p}$ with basis
%Fix a basis of $\gc$ given by
%$$
%J:=\begin{pmatrix} 0 & 0 & 0 & 0\\
%0 & 0 & 0 & 0\\
%0 & 0 & i & 0\\
%0 & 0 & 0 & -i\\
%\end{pmatrix}\;,\; E:=\begin{pmatrix} 0 & 1 & 0 & 0\\
%1 & 0 & 0 & 0\\
%0 & 0 & 0 & 0\\
%0 & 0 & 0 & 0\\
%\end{pmatrix}\;,\; T:=\begin{pmatrix} i & 0 & 0 & 0\\
%0 & i & 0 & 0\\
%0 & 0 & -i & 0\\
%0 & 0 & 0 & -i\\
%\end{pmatrix}\ .
%$$
The Cartan with $\dim\gh_\circ=1$ is not adapted, as in that case $\gh_\circ$ is generated by a semisimple element with eigenvalues $0,\pm 1$. The Cartan \eqref{aggiuntina} is adapted. To see this, 
note that the set $\left\{E, H\right\}$ as in Table \ref{tabellona3}
is a basis of $\gh_\circ$ consisting of grading elements of two inequivalent gradings. The element corresponding to the grading of Table \ref{tabellona2} is $E$. We can therefore argue as in previous cases once we note that in this case the decomposition of $\sl_4(\bR)$ under the adjoint action of the generator  $\mathfrak{J}$ of $\gh_\bullet$ is 
\beq
\label{scelta2}
\begin{split}
&\gg^{-\alpha_1}=\left\langle e^{01}_1\right\rangle\;,\qquad\gg^{-\alpha_3}=\left\langle e^{01}_2\right\rangle\;,\\
&\gg^{-(\alpha_1+\a_2)}=\left\langle e^{10}_1\right\rangle\;,\qquad\gg^{-(\a_2+\alpha_3)}=\left\langle e^{10}_2\right\rangle\;,\\
&\gg^{-\alpha_2}=\gm^{0(10)}=\left\langle e^{0(10)}\right\rangle\;,\qquad \gg^{\alpha_2}=\gm^{0(01)}=\left\langle e^{0(01)}\right\rangle\;,
\end{split}
\eeq
where $\left\{e_1,\ldots,e_4\right\}$ is a \emph{complex-Witt basis}. This proves (i) and (ii).

The second part of (iii) follows from Theorem \ref{primoteorema} and  $\sgn(\mathfrak J)$ from \eqref{simple} and \eqref{complexsymplectic}-\eqref{complexWitt}. A direct inspection of the adjoint action of $e^{0(10)}$ on $\gg^{-1(01)}$ and equation \eqref{identificationsymmetric} 
give the component $\gm^{0(10)}$ of the core for $\su(2,2)$ and $\su(1,3)$; %of $\gk^0$ with $S^{2}(\gg^{-1})$) 
in the first case we need to note
$$\gm^{0(10)}=\left\langle e_1^{10}\odot e_2^{10}\right\rangle\simeq\left\langle e_1^{10}\odot e_1^{10}+e_2^{10}\odot e_2^{10}\right\rangle\;.$$ 
%as equivalent orbits in $\mathbb P^2(\bC)$. 
If $\gg=\sl_4(\bR)$ then $\gm^{0(10)}=\left\langle (e'_1)^{10}\odot (e'_1)^{10}-(e'_2)^{10}\odot (e'_2)^{10}\right\rangle$ 
w.r.t. the basis \eqref{Witt-symplectic} associated to the complex-Witt basis of Table \ref{tabellona3}. The Lie algebra $\gn_\sharp$ of the stabilizer of $\gm^{0(10)}$ is given in Theorem \ref{firstcores} and Theorem \ref{secondcores}. 

The first part of (iii) follows from the following argument. Note that the complex subalgebra $\gq=\wh\gg\cap\gu$ of the CR algebra $(\gg,\gq)$ associated with the model $\gg$ is
\begin{align*}
\gq&=\wh{\gg}^{2}+\wh{\gg}^{1}+\wh\gh+\gg^{-\alpha_2}+\gg^{-(\alpha_1+\alpha_2)}+\gg^{-(\alpha_2+\alpha_3)}\;.
%&=\wh{\gg}^{2}+\wh{\gg}^{1}+\gc^{\bC}+\left\langle e^{0(10)}\right\rangle+\left\langle e^{-1(10)}_2\right\rangle+\left\langle e^{-1(10)}_1\right\rangle\ .
\end{align*}
%In this case
%\begin{align*}
%\gq&=\wh{\gg}^{2}+\wh{\gg}^{1}+\gc^{\bC}+\wh{\gg}_{\alpha_2}+\wh{\gg}_{\alpha_1+\alpha_2}+\wh{\gg}_{\alpha_2+\alpha_3}\\
%&=\wh{\gg}^{2}+\wh{\gg}^{1}+\gc^{\bC}+\left\langle e^{0(10)}\right\rangle+\left\langle e^{-1(10)}_1\right\rangle+\left\langle e^{-1(10)}_2\right\rangle
%\end{align*}
This is the maximal $11$-dimensional parabolic subalgebra of $\wh\gg$
which corresponds, for an appropriate choice of system of simple roots, to the nonnegatively graded part of the $\bZ$-grading of $\wh\gg$ associated with the simply crossed Dynkin diagram 
{\tiny $\begin{tikzpicture}
%\node[root]   (1)                     {};
\node[root] (1) {}; %[right=of 1] {} edge [-] (1);
\node[xroot]   (2) [right=of 1] {} edge [-] (1);
\node[root]   (3) [right=of 2] {} edge [-] (2);
%\node[root]   (5) [right=of 4] {} edge [-] (4);
\end{tikzpicture}$} (a word of caution: the $\bZ$-grading of Table \ref{tabella} used in the construction of the models is different and associated with the $10$-dimensional parabolic subalgebra $\wh\gg_{\geq 0}=\wh\gg^0\oplus\wh\gg^1\oplus\wh\gg^2$). In particular $(\gg,\gq)$
is a \emph{parabolic CR algebra} in the sense of \cite{MeNa3} and hence there  exists always an associated  globally defined homogeneous CR manifold $M=G/G_o$.
This proves (iii).

Now $\gh=\gn_\sharp$ follows from $\dim\gh=\dim\gn_\sharp=3$ and the fact that the root space $\gm^{0(10)}$ is preserved by $\gh$. This proves (iv). 
%\medskip\par

We turn to (v). Let $\wh\gg_{\leq 0}=\wh\gg^{-2}\oplus\wh\gg^{-1}\oplus\wh\gg^{0}$ be the nonpositively graded graded part of $\wh\gg$; the adjoint action of $\wh\gg^0\simeq\ggl_2(\bC)\oplus\bC E$ on $\wh\gg_-=\wh\gg^{-2}\oplus\wh\gg^{-1}$ is given in Table \ref{tabella} where $\bC^{2}=\gg^{-\a_1}\oplus\gg^{-(\a_1+\a_2)}$ and $(\bC^2)^*=
\gg^{-\a_3}\oplus\gg^{-(\a_3+\a_2)}$. 
We first claim that the maximal prolongation $\wh\gg_{\infty}$ of $\wh\gg_{\leq 0}$ is finite-dimensional. This is a consequence of a deep theorem of Tanaka (see \cite[Theorem 11.1]{Ta1} and also \cite[Corollary 2, pag. 76]{Ta1}) based on some arguments of Serre on Spencer cohomology of Lie algebras (see \cite{GS}). In the form suitable for our purposes, this result says: 
\vskip0.2cm\par
{\it The maximal transitive prolongation $\wh\gg_{\infty}$ of a fundamental and transitive $\bZ$-graded Lie algebra $$\wh\gg_{\leq 0}=\bigoplus_{-2\leq p\leq 0}\wh\gg^{p}$$ is finite-dimensional if and only if the usual Cartan prolongation (in the sense of e.g. \cite[Chapter VII]{St}) of the linear Lie algebra 
$$
\mathfrak k=\left\{X\in\wh\gg^0\mid [X,\wh\gg^{-2}]=0^{\phantom{C^{C^C}}}\!\!\!\!\!\!\!\!\!\!\!\right\}\subset \mathfrak{gl}(\wh\gg^{-1})$$ is finite-dimensional.} 
\vskip0.2cm\par
In our case $\mathfrak k=\ggl_2(\bC)$. Consider the bilinear form $\beta$ on $\wh\gg^{-1}$ given by
$$
\beta(z+z^*,w+w^*)=z^*(w)+w^*(z)
$$
where $z,w\in\bC^2$ and $z^*, w^*\in(\bC^2)^*$. Straightforward computations show:
\begin{itemize}
\item[(i)] $\beta$ is symmetric and nondegenerate;
\item[(ii)] $\mathfrak k\subset\so(\wh\gg^{-1},\beta)$.
\end{itemize}
It follows that $\mathfrak k$ has a trivial Cartan prolongation and hence $\dim(\wh\gg_\infty)<+\infty$.

Let now $\gg'$ be another model with $\gg'\supseteq\gg$. Note that
$\gg'^0=\gg^0$ as $\gg'^0\supseteq\gg^0$, $\dim(\gm)=7$ and $\gn_\sharp\subset\gg^0$.
It follows that $\wh\gg'$ is a transitive prolongation of the same $\wh\gg'_{\leq 0}=\wh\gg_{\leq 0}$ and hence a subalgebra of $\wh\gg_\infty$. In particular it is finite-dimensional and $\wh\gg'=\wh\gg$ by \cite[Theorem 3.21]{MeNa1}. 

Finally (vi) follows by a direct computation using \eqref{chebarba} and $\overline\gq=\wh{\gg}^{2}\oplus\wh{\gg}^{1}\oplus\wh\gh\oplus\gg^{\alpha_2}\oplus\gg^{-\alpha_3}\oplus\gg^{-\alpha_1}$, we omit the details. The theorem is proved.
%&=\wh{\gg}^{2}+\wh{\gg}^{1}+\gc^{\bC}+\left\langle e^{0(01)}\right\rangle+\left\langle e^{-1(01)}_2\right\rangle+\left\langle e^{-1(01)}_1\right\rangle\ ,
%and
%\begin{align*}
%\overline\gq&=\wh{\gg}^{2}+\wh{\gg}^{1}+\gc^{\bC}+\wh{\gg}_{-\alpha_2}+\wh{\gg}_{\alpha_1}+\wh{\gg}_{\alpha_3}\\
%&=\wh{\gg}^{2}+\wh{\gg}^{1}+\gc^{\bC}+\left\langle e^{0(01)}\right\rangle+\left\langle e^{-1(01)}_1\right\rangle+\left\langle e^{-1(01)}_2\right\rangle\ .
%\end{align*}
\end{proof}
%\bigskip\par
%We now give some remarks.
%\bigskip\par
We remark that the models $\gg$ of Theorem \ref{thmsimple} have a unique up to conjugation admissible Cartan subalgebra $\gh$. It is maximally noncompact only in one case, namely for $\gg=\su(1,3)$.

Theorem \ref{mainI} follows from Theorem \ref{thmsimple} and the following Theorem \ref{mainII}. We recall that any core $\gm=\gm^{-2}\oplus\gm^{-1}\oplus\gm^{0}$, $\dim(\gm)=7$, is completely determined by the complex line $\gm^{0(10)}\subset\mathfrak M^{0(10)}\simeq S^2\bC^2$ given by $\wh\gm^0=\gm^{0(10)}\oplus\overline{\gm^{0(10)}}$ and that  the Lie algebra $\gn_\sharp$
of the stabilizer of $\gm^{0(10)}$ has already been described in \S\ref{sec:7cores}. %for the natural action of the conformal unitary group $\mathrm{CU}(r,s)$, $\sgn(\mathfrak J)=(r,s)$
\begin{theorem}
\label{mainII}
Let $\gm$ be a core, $\dim(\gm)=7$, $\operatorname{ht}(\gm)=0$, given by
$\gm=\gm_1$ in Table \ref{tablecore20} or 
$\gm_{t}$ with $t=\pm 1$ and $\gm_{\rm null}$ in Table \ref{simbacacca}.
Then the $\bZ$-graded subspace $\displaystyle\gg=\bigoplus\gg^p$ 
of $\gc$ with components
\begin{equation*}
\gg^{p}=
\begin{dcases}
0\;\;\;\;\;\;\;\;\,\,\,\,\qquad\qquad\qquad&\text{for all}\;\;p<-2\;\;\;\text{and}\;\;p>0\,,\\
\gc^{p}\,\;\;\;\;\;\;\;\;\;\;\;\;\qquad\qquad\qquad&\text{for}\;\;p=-2,-1\,,\\
\gn_\sharp\oplus\Re(\gm^{0(10)}\oplus\overline{\gm^{0(10)}})\,\;\;\;\;\;\;\;\;\;&\text{for}\;\;p=0\;,
\end{dcases}
\end{equation*}
is a model of type $\gm$ with $\dim(\gg)$, $\sgn(\mathfrak J)$, equivalence class $[\gm^{0(10)}]$ of $\gm^{0(10)}$ and Lie algebra $\gn_\sharp$ given by:
\par\noindent
%\vfill
\begin{centering}
\begin{table}[H]
\begin{tabular}{|c|c|c|c|}
\hline
\hline
$\dim(\gg)^{\phantom{C^{C^C}}}\!\!\!\!\!\!\!\!\!\!$ &  $\operatorname{sgn}(\mathfrak J)$ & $[\gm^{0(10)}]$ & $\gn_{\sharp_{\phantom{C_{C_C}}}}^{\phantom{C^{C^C}}}\!\!\!\!\!\!\!\!\!\!$ \\
\hline
\hline
&&&\\[-1mm]
$10$
 & $(2,0)_{\phantom{C_{C_C}}}^{\phantom{C^{C^C}}}\!\!\!\!\!\!\!\!\!\!$ & $e_1^{10}\odot e_1^{10}$
&
$\bC\oplus\so_2(\bR)$
\\
\hline
&&&\\[-1mm]
$10$
 & $(1,1)_{\phantom{C_{C_C}}}^{\phantom{C^{C^C}}}\!\!\!\!\!\!\!\!\!\!$ & $e_1^{10}\odot e_1^{10}$
&
$\bC\oplus\so_2(\bR)$
\\
\hline 
&&&\\[-1mm]
$10$
 & $(1,1)_{\phantom{C_{C_C}}}^{\phantom{C^{C^C}}}\!\!\!\!\!\!\!\!\!\!$ & $e_2^{10}\odot e_2^{10}$
&
$\bC\oplus\so_2(\bR)$
\\
\hline
&&&\\[-1mm]
$11$ & $(1,1)_{\phantom{C_{C_C}}}^{\phantom{C^{C^C}}}\!\!\!\!\!\!\!\!\!\!$ & $e_1^{10}\odot e_1^{10}-e_2^{10}\odot e_2^{10}-2ie_1^{10}\odot e_2^{10}$
&
$\bC\oplus(\bR\inplus\bR)$
\\
\hline 
\hline
\end{tabular}
\caption{}
\label{finaltabella}
\end{table}
\end{centering}
\par\noindent
In all cases there exists an associated homogeneous CR manifold $M=G/G_o$, $Lie(G)=\gg$, $Lie(G_o)=\gn_\sharp$, which is globally defined.
\end{theorem}
\begin{proof}
It is sufficient to prove that $\gg$ is a Lie subalgebra of $\gc$ or, equivalently, $\wh\gg$ a Lie subalgebra of $\wh\gc$. Since $\wh\gg$ is nonpositively $\bZ$-graded with $\wh\gg^p=\hat\gc^p$ for all $p<0$, what we really need to show is that $\wh\gg^0=\wh\gn_\sharp\oplus\gm^{0(10)}\oplus\overline{\gm^{0(10)}}$ is a Lie subalgebra of $\hat\gc^0$. First of all:
\begin{itemize}
\item[(i)] $[\wh\gn_\sharp,\wh\gn_\sharp]\subset\wh\gn_\sharp$ since $\gn_\sharp$ is a Lie algebra;
\item[(ii)] $[\wh\gn_\sharp,\gm^{0(10)}]\subset\gm^{0(10)}$ and $[\wh\gn_\sharp,\overline{\gm^{0(10)}}]\subset\overline{\gm^{0(10)}}$
by the definition of $\gn_\sharp$; %and since $\gm^{0(01)}=\overline{\gm^{0(10)}}$, $\overline{\wh\gn_\sharp}=\wh\gn_\sharp$; 
\item[(iii)] $[\gm^{0(10)},\gm^{0(10)}]=[\overline{\gm^{0(10)}},\overline{\gm^{0(10)}}]=0$ since $\dim_{\bC}(\gm^{0(10)})=1$.% and $\gm^{0(01)}$ are $1$-dimensional. %from a direct computation using Proposition \ref{noioso}.
\end{itemize}
Finally $[\gm^{0(10)},\overline{\gm^{0(10)}}]\subset\wh\gn_\sharp$ by an explicit computation which uses \eqref{complexsymplectic} and Proposition \ref{noioso}. We only give the details for the last case of Table \ref{finaltabella}, for which $[\gm^{0(10)},\overline{\gm^{0(10)}}]=0$ actually holds: set %Since $\gm^{0(10)}$ and $\gm^{0(01)}$ are $1$-dimensional 
\begin{align*}
X&=e_1^{10}\odot e_1^{10}-e_2^{10}\odot e_2^{10}-2ie_1^{10}\odot e_2^{10}\qquad\qquad\qquad
\end{align*}
and compute
\begin{align*}
[X, \overline{X}]&=-2ie_1^{10}\odot e_1^{01}+2e_1^{10}\odot e_2^{01}+2i e_2^{10}\odot e_2^{01}+2e_2^{10}\odot e_1^{01}\\
&\phantom{=}\;\,-2 e_2^{10}\odot e_1^{01}-2e_1^{10}\odot e_2^{01}-2ie_2^{10}\odot e_2^{01}+2i e_1^{10}\odot e_1^{01}\\
&=0\;.
\end{align*}
%This shows that $\wh\gg$ is a Lie algebra and $\gg$ a model. 
%\medskip\par
Table \ref{finaltabella} comes directly from Table \ref{tablecore20} and Table \ref{simbacacca}. Finally, we consider the simply connected Lie group $G$ with Lie algebra $Lie(G)=\gg$ and it is a direct task to see that the analytic subgroup $G_o$ of $G$ with Lie algebra $Lie(G_o)=\gn_\sharp$ is closed in $G$. We omit the details.
\end{proof}
%\medskip\par
\begin{remark}\hfill
\begin{itemize}
\item[(i)] Each of the models in Table \ref{finaltabella} satisfies property $(\mathfrak J)$ since $\mathfrak J\in\gn_\sharp$ by definition;
\item[(ii)] We do not know if the models in Table \ref{finaltabella} are maximal or not;
\item[(iii)] We believe that there are no models associated with cores of type (B) that are not admissible and checked 
the conjecture when $\sgn(\mathfrak J)=(2,0)$. 
It would be interesting to understand if $7$-dimensional $2$-nondegenerate strongly regular CR manifolds with non-admissible cores exist or not.
On this regard, we remark that for most CR-dimensions and CR-codimensions the assumption of strong regularity is quite restrictive.
\end{itemize}
\end{remark}
%By considering the analogous choice \eqref{scelta} of simple roots as in the case $\gg=\su_{1,3}$, one is lead exactly to the same arguments and conclusions.
%By considering the analogous choice \eqref{scelta} of simple roots as in the case $\gg=\su_{1,3}$, one
%By considering the choice of simple roots $\Delta=\{\alpha_1,\alpha_2,\alpha_3\}$ for which
%$$
%\wh{\gg}_{\alpha_1}=\left\langle e^{-1(01)}_1\right\rangle\,,\; \wh{\gg}_{\alpha_2}=\left\langle e^{0(10)}\right\rangle\,,\; \wh{\gg}_{\alpha_3}=\left\langle e^{-1(01)}_2\right\rangle\,,
%$$
%obtains a CR algebra $(\gg,\gq)=(\gsl_{4}(\bR),\gq)$. 
%For the discussions of the next sections, it is  quite useful
%to have  all Lie brackets between  elements of  the basis $\cB$ explicitly written down. Moreover, in place  of the   elements $E^{0}_j$ and $e^{\ell}_j$,  for $\ell = -1,0,1$ and $j = 1,2$,  it is   convenient to consider  the   elements in the complexification $(\so_{3,2})^\bC$ 
%\beq \label{definitionholelements}
%\begin{array}{llll}
%E^{0 (10)} &=  \frac{1}{2} \left(E^0_1 - i E^0_2\right)\,, &
%E^{0 (01)} &= \overline{E^{0 (10)}}
%\,,\\
%\,\\
%e^{\ell(10)} &=  \frac{1}{2} \left(e^{\ell}_1 - i e^{\ell}_2\right)
 %\,, &e^{\ell(01)} &= \overline{e^{\ell(10)}}\,,
 %\end{array} \quad \ell = -1,0,1\,,\eeq
%and  evaluate the Lie brackets between   such vectors and between  them and  the real  vectors  $E^2$,  $e^{-2}$ or $E^0_1$.  Here is the list  of such Lie brackets.
%\smallskip\par\noindent
\section{Models for \texorpdfstring{$7$}{7}-dimensional and \texorpdfstring{$3$}{3}-nondegenerate CR manifolds}\hfill\par\noindent
\setcounter{equation}{0}
\setcounter{section}{6}
\label{exexamplesII}
There is a unique abstract core $\gm=\bigoplus\gm^p$ 
associated to $3$-nondegenerate and $7$-dimensional CR manifolds. It is realized inside 
the real contact algebra $\gc=\bigoplus\gc^p$ of degree $n=1$ as
\begin{equation}
\label{kore2}
\begin{split}
\gm^{-2}&=\gc^{-2}=\left\langle e^{-2}\right\rangle\;,\qquad\gm^{-1}=\gc^{-1}=\left\langle e_1,e_2\right\rangle\;,\\
\gm^{0(10)}&=\mathfrak{M}^{0(10)}\;,\qquad\gm^{1(10)}=\mathfrak{M}^{1(10)}\;,
\end{split}
\end{equation}
where $\mathfrak{M}^{0(10)}$ and $\mathfrak{M}^{1(10)}$ are the $\ad(\mathfrak J)$-eigenspaces of maximal eigenvalue $2i$ in $\wh\gc^0$ and, respectively, $3i$ in $\wh\gc^1$. 
%An elementary dimensional argument says that the Freeman bundles associated to a necessarily satisfy 
%$$
%\rk_\bC\cD^\bC=6\;\;,\;\;\rk_\bC\cF_0=4\;\;,\;\;\rk_\bC\cF_1=2\;\;\text{and}\;\;\rk_\bC\cF_2=0\,.
%$$
%The degree $n=\frac{1}{2}\rk(\cD)/\rk(\Re(\cF_0))=1$ 
For notational convenience, we denote the elements of the basis $\left\{e_1^{10},e_1^{01}\right\}$ of $\wh\gc^{-1}=S^{1,0}\oplus S^{0,1}$ simply by
$
z=e_1^{10}$ and $\bar z=e_{1}^{01}$, drop the symbol $\odot$ in the expression of the symmetric products and identify each image $\Im(\mu^{p|0})\subset\gc^p$ with $\bR$ using $\mu^{p|0}(e^{-2})$ as basis (recall also the discussion after Proposition \ref{noioso}). For instance we write 
$$[z,\bar z]=-\frac{i}{2}\;,\qquad 
\mathfrak{J}=%-\frac{1}{2}(e_1^2+e_2^2)=
2z\bar z\;,\qquad %\in S^{1,1}\cap\gc^0$
E= -2\;,$$
where $E$ is the grading element.
\begin{theorem}
\label{thmstrange}
There exists a maximal model $\gg$ of type \eqref{kore2} and it is unique up to isomorphism. 
It is given by the $8$-dimensional $\bZ$-graded Lie subalgebra
\begin{equation*}
%\label{modellino}
\gg=\bigoplus_{p\in\bZ}\gg^{p}
\end{equation*} 
of the real contact algebra $\gc$ of degree $n=1$ with components
$$
\gg^p=\begin{cases}
0\;\;\;\;\;\;\;\;\;\;\;\;\;\;\;\;\;\;\;\;\;\;\;\;\;\;\;\;\;\;\;\;\;\;\;\;\;\;\;\;\;\;\;\;\;\;\;\;\;\;\;&\text{for all}\;\;p<-2\;\text{and}\;p>1\;,\\
\gc^{p}\,\;\;\;\;\;\;\;\;\;&\text{for}\;\;p=-2, -1\;,\\
\Re\left\langle E, M, \overline M\right\rangle\,\;\;\;\;\;&\text{for}\;\;p=0\;,\\
\Re\left\langle N, \overline N\right\rangle\,\,\,\;\;\;\;\;\;&\text{for}\;\;p=1\;,\\
%0\;\;\,\;\;\;\;\;\;\;\;&\text{for all}\;\;p>1\;,\\
\end{cases}
$$
where $M=z^2+z\overline z\in\mathfrak{M}^{0(10)}$ and  $N=z^3+2z^2\bar z+z\bar z^2-3iz-3i\bar z\in\mathfrak{M}^{1(10)}$. 
The associated terms of the Freeman sequence are
\begin{align*}
\gq_{-1}&=\gq=\left\langle z, E, M, N  \right\rangle\,,\quad
\gq_{0}=\left\langle E, M, N \right\rangle\,,\\
\gq_{1}&=\left\langle E, N  \right\rangle\,,\qquad
\gq_2=\gq\cap\overline\gq=\left\langle E\right\rangle\ .
\end{align*}
Moreover there is a $7$-dimensional $3$-nondegenerate homogeneous CR manifold $M=G/G_o$, $Lie(G)=\gg$, $Lie(G_o)=\mathbb R E$, which is  globally defined.
\end{theorem}
\begin{proof}
We first infer some necessary conditions assuming the existence of $\gg$. We split the proof in several steps.  
\smallskip\par\noindent
{\it Step 1. The Lie subalgebra $\gg^0$.}
\smallskip\par
We claim that $\gg$ does not satisfy property $(\mathfrak J)$ (i.e., $\mathfrak J\in\gg^0$).
 Indeed in that case $\gg^0=\gc^0$, by \eqref{kore2} and (ii)-(iv) of Definition \ref{ukip}, but any subalgebra with the same nonpositively graded part of $\gc$ is of the form (\cite[Proposition 3.2]{MT})
\begin{itemize}
\item[--] $\gg=\gc^{-2}\oplus\gc^{-1}\oplus\gc^{0}$,
\item[--] $\gg=\gc^{-2}\oplus\gc^{-1}\oplus\gc^{0}\oplus \mu^{1}(\gc^{-1})\oplus\mu^{2|0}(\gc^{-2})$,
\item[--] $\gg=\gc^{-2}\oplus\gc^{-1}\oplus\gc^{0}\oplus\bigoplus_{p>0}\gk^p$,
\item[--] $\gg=\gc$,
\end{itemize}
and this is a contradiction, since $\operatorname{ht}(\gm)=1$. Hence $\dim\wh\gg^0=3$ and there is a basis of $\wh\gg^0$ of the form
$
\{E, z^2+\alpha z\bar z, \bar z^2+\bar\alpha z\bar z\}
$ for some $\alpha\in\bC$, by (ii)-(iv) of Definition \ref{ukip}. A direct computation yields
$$
[z^2+\alpha z\bar z, \bar z^2+\bar\alpha z\bar z]=-i\bar\alpha z^2 -2i z \bar z-i\alpha \bar z^2
$$
but this bracket is again in $\wh\gg^0$ if and only if $\alpha=e^{i\vartheta}$ for some $\vartheta\in[0,2\pi)$. It follows that $\wh\gg^0$ equals the Borel subalgebra of $\wh\gc^0\simeq\ggl_{2}(\bC)$ given by
$$
\mathfrak{b}_{\vartheta}=\left\langle E, z^2+e^{i\vartheta}z\bar z, \bar z^2+e^{-i\vartheta}z\bar z\right\rangle\;.
$$
 We now see $\vartheta=0$, up to isomorphisms of models. In view of the observation before Definition \ref{Knondeg} it is enough to note that the $0$-degree automorphism
%of $\wh\gc_-=\wh\gc^{-2}\oplus\wh\gc^{-1}$ given by
\begin{multline*}
%\label{menomale}
T_{\vartheta}:\wh\gc_-\longrightarrow \wh\gc_-\;,\qquad\quad T_\vartheta(e^{-2})=e^{-2}\;,\\
T_\vartheta(z)=e^{-i\vartheta/2}z\;,\qquad T_\vartheta(\bar z)=e^{i\vartheta/2}\bar z\;,
\end{multline*}
satisfies the following properties:
\begin{itemize}
\item[(i)] $T_{\vartheta}$ is real, in the sense that $
\overline{T_{\vartheta}(X)}=T_{\vartheta}(\overline{X})$ for all $X\in\wh\gc_-$;
\item[(ii)] $T_{\vartheta}$ commutes with $\ad(\mathfrak J):\wh\gc_-\longrightarrow\wh\gc_-$;
\item[(iii)] the prolongation of $T_{\vartheta}$ to $\wh\gc$ sends $\gb_0$ onto $\gb_{\vartheta}$.
\end{itemize}
%and in turn exploits the direct consequence that the image $\gg_\vartheta=\Im T_\vartheta(\gg_0)$ of a model $\gg_0$ with $0$-degree part $\gg_0^0=\gb_0$ is also a model, with $\gg_\vartheta^0=\gb_\vartheta$.
From now on $\wh\gg^0$ is the Borel subalgebra $\gb=\gb_0=\{E,M, \overline M\}$ stabilizing the line spanned by $e_1=z+\bar z$.%we note $E\in\gu\cap\bar\gu$, $M\in\gu$ and $\overline M\in\bar\gu$. 
\smallskip\par\noindent
{\it Step 2. The space $\gg^1$ as a representation of $\gg^0$.}
\smallskip\par
We first note that the auxiliary space 
\begin{equation*}
%\label{anchequi}
\widetilde\gg^1=\left\{X\in\wh\gc^1\,|\,[X,\wh\gc^{-1}]\subset \wh\gg^0\phantom{C^{C^C}}\!\!\!\!\!\!\!\!\!\!\!\right\}
\end{equation*}
is a $\wh\gg^0$-module with $\wh\gg^1\subset\wt\gg^1$ as a submodule, as $\wh\gg$ is a $\bZ$-graded Lie algebra. Now
\begin{align*}
X\in\wh\gc^{1}&\simeq (S^{3,0}\oplus S^{2,1}\oplus S^{1,2}\oplus S^{0,3})\oplus(S^{1,0}\oplus S^{0,1})%\\
%&\phantom{\simeq}\;\; (S^{1,0}\oplus S^{0,1})
\end{align*}
decomposes into
%\begin{equation*}
%\label{elemento}
$X=\alpha_{30} z^3+\alpha_{21} z^2\bar z+ \alpha_{12}z\bar z^2  +  \alpha_{03} \bar z^3+\alpha_{10} z+\alpha_{01}\bar z$
%\end{equation*}
and a direct computation using Proposition \ref{noioso} shows
\begin{align*}
%\label{servee}
[X,z]&=\frac{1}{2}(\alpha_{10}+i \alpha_{21}) z^2
+(\frac{1}{2}\alpha_{01}+i \alpha_{12})z\bar z+
\frac{3}{2}i\alpha_{03} \bar z^2+\frac{1}{2}i\alpha_{01}\;,\\
%\label{servee2}
[X,\bar z]&=-\frac{3}{2}i \alpha_{30} z^2+(\frac{1}{2}\alpha_{10}-i \alpha_{21})z\bar z
+\frac{1}{2}(\alpha_{01}-i\alpha_{12}) \bar z^2-\frac{1}{2}i\alpha_{10}\;.
\end{align*}
It turns out that these brackets are in $\wh\gg^0$ if and only if the following linear system of  equations is satisfied:
\begin{align*}
\alpha_{10}-2i\alpha_{21}&=\alpha_{01}-i\alpha_{12}-3i\alpha_{30}\,,\\
\alpha_{01}+2i\alpha_{12}&=3i\alpha_{03}+\alpha_{10}+i\alpha_{21}\,.
\end{align*}
In other words
$$
\widetilde\gg^1=(z+\bar z)\odot \gb\bigoplus \left\langle z^2\bar z+z\bar z^2+\frac{i}{2}z  -\frac{i}{2}\bar z\right\rangle
$$
%To determine $\wh\gg^1$, 
%is a $\wh\gg^0$-submodule of $\widetilde\gg^1$ 
%which is at least two-dimensional. We now make this assertion more precise. 
with a basis of the form
$\displaystyle
\left\{N,\overline N, V, W\!\!\!\!\!\!\!\!\!\!\!\!\!\phantom{C^{C^C}}\right\}
$
where
\begin{align*}
V&=z^2\bar z+z\bar z^2+\frac{i}{2}z-\frac{i}{2}\bar z\;,\qquad\qquad\\
W&=z+\bar z\\
N&=(z+\bar z)\odot (z^2+z\bar z) -3i(z+\bar z)\\
&=z^3+2z^2\bar z+z\bar z^2-3i z -3i\bar z\;.
\end{align*}
The element $N$ is characterized by the following property: it is the unique non-trivial element in $\widetilde\gg^1\cap\gu$ which commutes with its conjugate (we omit the long but straightforward proof of this fact, based again on Proposition \ref{noioso}). We also note that $\pi_{\mathfrak M^{1(10)}}(V)=\pi_{\mathfrak M^{1(01)}}(V)=\pi_{\mathfrak M^{1(10)}}(W)=\pi_{\mathfrak M^{1(01)}}(W)=0$.

Now (iii)-(iv) of Definition \ref{ukip} say that $\wh\gg^1\cap\overline\gu$ is at least one-dimensional, including an element of the form
\begin{equation*}
%\label{forma}
\overline N_{\alpha\beta}=\overline N+\alpha V+\beta W\,,
\end{equation*}
for some $\a,\b\in\bC$. Clearly $N_{\alpha\beta}\in\wh\gg^1\cap\gu$ too. From $\ad(\mathfrak{J})$-equivariance and Proposition \ref{noioso}, we get for all $\delta,\gamma\in\bC$
\begin{align*}
\pi_{\mathfrak{M}^{2(10)}}[N_{\alpha\beta},\gamma V+\delta W]&=
\pi_{\mathfrak{M}^{2(10)}}[z^3,(\delta+\frac{i}{2}\gamma) z+\gamma z^2\bar z ]\\
&=(\frac{1}{2}\delta-\frac{5}{4}i\gamma) z^4\,,\\
\pi_{\mathfrak{M}^{2(01)}}[\overline N_{\alpha\beta},\gamma V+\delta W]&=
\pi_{\mathfrak{M}^{2(01)}}[\bar z^3,(\delta-\frac{i}{2}\gamma)\bar z+\gamma z\bar z^2]\\
&=(\frac{1}{2}\delta+\frac{5}{4}i\gamma) \bar z^4\;,
\end{align*}
where $\pi_{\mathfrak{M}^{2(10)}}:\wh\gc^2\longrightarrow \mathfrak{M}^{2(10)}$ and $\pi_{\mathfrak{M}^{2(01)}}:\wh\gc^2\longrightarrow \mathfrak{M}^{2(01)}$ 
are the projections 
onto the $\ad(\mathfrak J)$-eigenspaces of extremal eigenvalues $\pm 4i$ in $\wh\gc^2$. But $\operatorname{ht}(\gm)=1$ and therefore an element $\gamma V+ \delta W$ belongs to $\wh\gg^1$ if and only if $\gamma=\delta=0$.
In other words we obtained $\wh\gg^1=\left\langle N_{\alpha\beta},\overline N_{\alpha\beta}\right\rangle$. We will now see that $\alpha=\beta=0$. %$\pi_{\mathfrak{M}^{2(10)}}\wh\gg^2=\pi_{\mathfrak{M}^{2(01)}}\wh\gg^2=(0)$, %$\wh\gg^1\cap\gu^1\cap\bar\gu^1=(0)$

By $\ad(\mathfrak{J})$-equivariance and $[N,\overline N]=0$, we get
\begin{align*}
\pi_{\mathfrak{M}^{2(10)}}[N_{\alpha\beta},\overline N_{\alpha\beta}]&=
\pi_{\mathfrak{M}^{2(10)}}[N,\overline N_{\alpha\beta}]\\
&=\pi_{\mathfrak{M}^{2(10)}}[N,(\beta+\frac{i}{2}\alpha)z+\alpha z^2\bar z]\\
&=\pi_{\mathfrak{M}^{2(10)}}[z^3,(\beta+\frac{i}{2}\alpha)z+\alpha z^2\bar z]\\
&=(\frac{1}{2}\beta-\frac{5}{4}i\alpha) z^4\,,
\end{align*}
forcing $2\beta=5i\alpha$. On the other hand
\begin{align*} 
[2 M,\overline N_{\alpha\beta}]&=-2i(1+\alpha)z^3-i(7+3\alpha)z^2\bar z
-i(8+\alpha)z\bar z^2-3i\bar z^3\\
&\;\;\;\;+(3+2\alpha)\bar z+(3+\alpha)z
\end{align*}
and the condition $[2 M,\overline N_{\alpha\beta}]\in\wh\gg^1$ is equivalent to the following system of non-linear equations on $\a\in\bC$,
$$
2\alpha+\bar\alpha+\alpha\bar\alpha=0\;,\quad 2\alpha+6\bar\alpha+6\alpha\bar\alpha=0\;,\quad 2\alpha-4\bar\alpha-4\alpha\bar\alpha=0\,.
$$
The unique solution of this system is $\alpha=0$, hence $\b=0$ and $\wh\gg^1=\left\langle N,\overline N \right\rangle$.
%\label{uruguay}

It is not difficult check now that $\wh\gg:=\wh\gc^{-2}\oplus\wh\gc^{-1}\oplus\wh\gg^0\oplus\wh\gg^1$ is a complex Lie subalgebra of $\wh\gc$ and $\gg=\Re(\wh\gg)$ a model of type \eqref{kore2}. The uniqueness and maximality of $\gg$ are a consequence of the next step.
%\begin{itemize}
%\item[(i)] $\wh\gg^1$ is invariant under conjugation;
%\item[(ii)] $\wh\gg^1=(\wh\gg^1\cap\gu^1)+(\wh\gg^1\cap\bar\gu^1$);
%\item[(iii)] $\wh\gg^1\cap\gu^1$ projects non-trivially on $\mathfrak{M}^{1(10)}$.
%\end{itemize}
\smallskip\par\noindent
{\it Step 3. The space $\gg^p$ is trivial for all $p\geq 2$.}
\smallskip\par
We show that the auxiliary subspace $\widetilde\gg^2$ of $\wh\gc^2$ determined by the inclusions
\begin{equation*}
\begin{split}
%if $X\in\widetilde\gg^2$, then it projects trivially on $\mathfrak{M}^{2(10)}\oplus \mathfrak{M}^{2(01)}$, that is
%\label{bah}
\widetilde\gg^2&\subset (S^{3,1}\oplus S^{2,2}\oplus S^{1,3})\oplus(S^{2,0}\oplus S^{1,1}\oplus S^{0,2})\oplus S^{0,0}\subset \wh\gc^2\;,\\
%\item[(ii)] if $X\in\tilde\gg^2$, then 
[\widetilde\gg^2,\wh\gc^{-1}]&\subset \wh\gg^1\;,
\end{split}
\end{equation*}
is trivial. This clearly implies $\wh\gg^2=0$ and also $\wh\gg^p=0$ for all $p> 2$, by the transitivity of $\wh\gc$.
Now any $X\in\wt\gg^2$
%$X\in\wh\gc^2$ with 
%\begin{multline*}
%\pi_{\mathfrak{M}^{2(10)}}(X)=\pi_{\mathfrak{M}^{2(01)}}(X)=0\\
%\qquad\text{and}[X,\wh\gc^{-1}]\subset \wh\gg^1
%\phantom{cccccccccc}
%\end{multline*} 
satisfies
$[X,e^{-2}]\subset \wh\gg^0$ and it is therefore of the form
\begin{align*}
X&=\alpha_{31}z^3\bar z +\alpha_{22}z^2 \bar z^2+ \alpha_{13}z\bar z^3\\
&\;\;\;+\alpha_{20} z^2+(\alpha_{20}+\alpha_{02})z\bar z+\alpha_{02}\bar z^2+\alpha_{00}\,,
\end{align*}
for some $\alpha_{31},\ldots,\alpha_{00}$ in $\bC$; from this and Proposition \ref{noioso} we get
\begin{align*}
[X,z]&=(\frac{i}{2}\alpha_{31}+\frac{1}{2}\alpha_{20})z^3
+(i\alpha_{22}+\frac{1}{2}\alpha_{02}+\frac{1}{2}\alpha_{20})z^2\bar z \\
&\;\;\;\;+(\frac{3}{2}i\alpha_{13}+\frac{1}{2}\alpha_{02})z\bar z^2
+(\frac{i}{2}\alpha_{02}+\frac{i}{2}\alpha_{20}+\frac{1}{2}\alpha_{00})z+i\alpha_{02}\bar z\,,\\
[X,\bar z]&=(-\frac{3}{2}i\alpha_{31}+\frac{1}{2}\alpha_{20})z^2\bar z 
+(-i\alpha_{22}+\frac{1}{2}\alpha_{02}+\frac{1}{2}\alpha_{20})z\bar z^2 \\
&\;\;\;\;+(-\frac{i}{2}\alpha_{13}+\frac{1}{2}\alpha_{02})\bar z^3
-i\alpha_{20} z+(-\frac{i}{2}\alpha_{02}-\frac{i}{2}\alpha_{20}+\frac{1}{2}\alpha_{00})\bar z\,.
\end{align*}
The claim $X=0$ follows from the fact that conditions ``$[X,z]$ proportional to $N$'' and ``$[X,\bar z]$ proportional to $\overline N$'' are equivalent to a homogeneous linear system of six equations and six indeterminates which is nonsingular. 
\smallskip\par\noindent
{\it Step 4. The last claims.}
\smallskip\par
The terms of the Freeman sequence follow from a direct computation. The globally well-defined homogeneous CR manifold is the quotient of the simply connected Lie group $G$ with Lie algebra $Lie(G)=\gg$ by the closed subgroup $G_o$ with Lie algebra $Lie(G_o)=\mathbb R E$. We omit the details.
\end{proof}
\begin{appendix}
\section{Proof of Proposition \ref{firstorbitprop}.}
\label{appendice2}
\setcounter{equation}{0}
We first consider the reducible representation \eqref{eq:rapres} of $
\wt K$. The associated orbits $\wt K\cdot z$ split in two types, according to whether the real and imaginary components $x$ and $y$ of $z=x+iy$ are linearly dependent over $\bR$ or not. 

In the first case $\wt K\cdot z\simeq \wt K/\wt H$ where $\wt H\simeq \mathrm{O}_{2}(\bR)$ and a complete set of representatives for these orbits
is
\beq
\label{mori1o}
\left\{z_{t}=(1+it)\e_1\,|\!\!\!\!\!\!\!\!\!\!\phantom{C^{C^C}}t\in\bR\right\}\,\cup\, i\cdot \e_1\;.
\eeq
In the second case $\wt K\cdot z \simeq \wt K/\wt H$ where $\wt H\simeq\bZ_2$ and the representative set is parametrized by the upper half plane
$$\mathrm H=\left\{(t_1,t_2)\,|\!\!\!\!\!\!\!\!\!\!\phantom{C^{C^C}}t_2>0\right\}$$   
and explicitly given by
\beq
\label{mori2o}
\left\{z_{t_1,t_2}=(1+it_1)\e_1+(it_2)\e_2\,|\!\!\!\!\!\!\!\!\!\!\phantom{C^{C^C}} (t_1,t_2)\in \mathrm H\right\}\;.\quad\;
\eeq
%we similarly say that $G_o\cdot z$ is of {\it type $H$} where $H=G_z$ is the stabilizer of $z$.
Recall that the actions of $\wt K$ and $\bC^\times$ commute. In particular any $c=e^{i\vartheta}$ sends an orbit $\wt K\cdot z$ of type $\wt H$ onto an orbit of the same type. 

Using \eqref{mori1o} we check that all orbits of type $\wt H\simeq\mathrm O_{2}(\bR)$ are related and glue into a single $K_\sharp$-orbit, say $K_\sharp\cdot\e_1$. In particular any $e^{i\vartheta}$ acts on
the collection 
$$\displaystyle
\cU=V^\times\backslash K_\sharp\cdot\e_1=\!\!\!\bigcup_{(t_1,t_2)\in\mathrm H}\wt K\cdot z_{t_1,t_2}
$$ 
of all orbits of type $\wt H\simeq\bZ_2$.
More precisely any $z=x+iy\in\cU$ belongs to the orbit $\wt K\cdot z_{t_1,t_2}$ where
\begin{equation}
\label{enzooo}
t_1=\frac{\left\langle x,y\right\rangle}{\left\langle x,x\right\rangle}\;,\quad\qquad (t_2)^2=\frac{\left\langle y-t_1 x,y-t_1 x\right\rangle}{\left\langle x,x\right\rangle}\;.
\end{equation}
This follows from a check with $z=z_{t_1,t_2}$ and from the observation that the r.h.s. of the two identities are constants on the orbits. In other words there is an action of $S^1$
on $\mathrm H$ with the property that $e^{i\vartheta}\cdot(t_1,t_2)=(t_1',t_2')$ if and only if the associated orbits $\wt K\cdot z_{t_1,t_2}$ and $\wt K\cdot z_{t_1',t_2'}$ are 
%$S^1$-related (and therefore 
subsets of a $K_\sharp$-orbit. 

A computation using \eqref{enzooo} and $z'=e^{i\vartheta}\cdot z\in\wt K\cdot z_{t_1',t_2'}$ yields the explicit expression of this action
%\begin{align}
%\label{antanii}
%e^{i\vartheta}\cdot (x+iy)=(\cos(\vartheta)x-\sin(\vartheta)y)+i(\cos(\vartheta)y+\sin(\vartheta)x)\ .
%\end{align}
%To determine the unique representative of the $\bR^{>0}\cdot G_\sharp$-orbit associated to \eqref{antanii}, 
\vskip-0.3cm\par\noindent
\begin{align}\notag
\\
\label{stanco1o}
t_1'&=%\frac{\left\langle\cos(\vartheta)x-\sin(\vartheta)y,\cos(\vartheta)y+\sin(\vartheta)x\right\rangle}{\left\langle\cos(\vartheta)x-\sin(\vartheta)y,\cos(\vartheta)x-\sin(\vartheta)y\right\rangle}=
\frac{\frac{1}{2}(1-t_1^2-t_2^2)\sin 2\vartheta+t_1\cos 2\vartheta}{(t_1\sin\vartheta-\cos\vartheta)^2+(t_2\sin \vartheta)^2}\;,\\
\notag
\\
\label{stanco2o}
(t_2')^2&=\frac{(1+t_1^2)^2\sin^2 \vartheta+t_2^2(t_1\sin \vartheta+\cos \vartheta)^2}{(t_1\sin \vartheta-\cos \vartheta)^2+(t_2\sin \vartheta)^2}\;,
\end{align}
\vskip0.2cm\par\noindent 
%Equations \eqref{stanco1o}-\eqref{stanco2o} define by construction 
and a continuity argument in $\vartheta$ from $0$ to $\pi/2$ in \eqref{stanco1o} says that $(t_1,t_2)$ is always on the $S^1$-orbit of some $(t_1',t_2')$ with $t_1'=0$.
Exploiting \eqref{stanco1o}-\eqref{stanco2o} now with $t_1=t_1'=0$, we finally get $(0,1]$ as set of representatives %$\left\{(0,t)\,|\phantom{C^{C^C}}\!\!\!\!\!\!\!\!\!\!0<t\leq 1\right\}$ 
 for $\mathrm H/S^1$. 

Summarizing: for $t\in[0,1]$ the $K_\sharp$-orbits $K_\sharp\cdot z_{0,t}=K_\sharp\cdot (\e_1+it\e_2)$ are pairwise disjoint and their union is the entire $V^\times$. This is the first claim. 
Using \eqref{mori0o} one also checks that $N_\sharp$ is isomorphic to $\bC^\times\cdot\mathrm{O}_{2}(\bR)$ if $t=0$, $\bC^\times\cdot\bZ_2$ if $0<t<1$, $\bC^\times\cdot SO_2(\bR)$ if $t=1$. This readily implies the last claim.
%\qed
\qed
\medskip\par
\section{Proof of Proposition \ref{secondorbitprop}.}
\label{appendice3}
It is similar to Proposition \ref{firstorbitprop} and we only give the main steps. As before, we first split the orbits in two types. In the first case any orbit is equivalent to one of the orbits displayed in Table $3$ and a complete set of representatives is
\begin{multline}
\label{mori1}
\left\{z_{t}=(1+it)\e_1\phantom{C^{C^C}}\!\!\!\!\!\!\!\!\!\!\!\!\right\}\,\cup\, i\cdot \e_1\;,\qquad\left\{w_{t}=(1+it)\e_3\phantom{C^{C^C}}\!\!\!\!\!\!\!\!\!\!\!\!\right\}\,\cup\, i\cdot \e_3\\
\text{and}\quad\left\{u_{t}=(1+it)(\e_1+\e_3)\phantom{C^{C^C}}\!\!\!\!\!\!\!\!\!\!\!\!\right\}\,\cup\, i\cdot (\e_1+\e_3)\;,\qquad\;
\end{multline}
where $t\in\bR$. The associated stabilizers $\wt H$ are also in Table $3$. In the second case we first use $\wt K$ to fix $x$ equal to $\e_1$ if space-like (resp. $\e_1+\e_3$ if null, $\e_3$ if time-like) 
and then the stabilizer of $x$ in $\wt K$ to fix $y$. This gives the following six different types of representatives:
\begin{description}
\item[(i)] $\e_1+i(t_1\e_1+t_2\e_2)$ where $t_1\in\bR$ and $t_2\neq 0$;
\item[(ii)] $\e_1+\e_3+i(t(\e_1+\e_3)\pm(\e_1-\e_3))$ where $t\in\bR$;
\item[(iii)] $\e_1+\e_3\pm i\e_2$;
\item[(iv)] $\e_1+i(t_1\e_1+t_3\e_3)$ where $t_1\in\bR$ and $t_3>0$,
\item[(v)] $\e_1+i(t\e_1\pm(\e_2+\e_3))$ where $t\in\bR$,
\item[(vi)] $\e_3+i(t_1\e_1+t_3\e_3)$ where $t_1>0$ and $t_3\in\bR$.
\end{description}
The stabilizer $\wt H$ is always trivial, except in case \textbf{(i)} where $\wt H\simeq\bZ_2$.

Any $e^{i\vartheta}$ sends an orbit $\wt K\cdot z$ of type $\wt H$ onto an orbit of the same type. We first focus on representatives \eqref{mori1}.
All $\wt K$-orbits with $\wt H\simeq\SO^+(1,1)\cup\SO^+(1,1)\begin{psmallmatrix} 1 &0 & 0\\ 0& 1 & 0 \\ 0 & 0 & -1 \end{psmallmatrix}$ are related and glue into $K_\sharp\cdot\e_1$. Similarly those of type $\wt H\simeq\SO_2(\bR)$ (resp. $\wt H\simeq\bR^+\ltimes\bR$)
glue into $K_\sharp\cdot\e_3$ (resp. $K_\sharp\cdot(\e_1+\e_3)$).

The collection of all $\wt K$-orbits of type $\wt H\simeq\bZ_2$
is stable under any $e^{i\vartheta}$ and the representative of $\wt K\cdot e^{i\vartheta}\cdot(\e_1+i(t_1\e_1+t_2\e_2))$ is
$
(\e_1+i(t'_1\e_1+t'_2\e_2))
$
where
%\vskip-0.3cm\par\noindent
\begin{align}\notag
\\
\label{stanco1oo}
t_1'&=\frac{t_1(\cos^2\vartheta-\sin^2\vartheta)+\frac{1}{2}(1-t_1^2-t_2^2)\sin 2\vartheta}{\cos^2\vartheta+(t_1^2+t_2^2)\sin^2\vartheta-t_1\sin 2\vartheta}
\;,\\
\notag
\\
\label{stanco2oo}
t_2'&=\frac{t_2}{\cos^2\vartheta+(t_1^2+t_2^2)\sin^2\vartheta-t_1\sin 2\vartheta}\;.
\end{align}
\vskip0.2cm\par\noindent 
A continuity argument %from $\vartheta=0$ to $\vartheta=\pi/2$ in \eqref{stanco1oo} 
implies that $(t_1,t_2)$ is always  on the $S^1$-orbit of some $(t'_1,t'_2)$ with $t'_1=0$ and hence, using  \eqref{stanco1oo}-\eqref{stanco2oo} with $t_1=t_1'=0$, that 
$$\left\{\e_1+it\e_2\,|\,-1\leq t\leq 1, t\neq 0\phantom{C^{C^C}}\!\!\!\!\!\!\!\!\!\!\!\!\right\}$$ 
parametrizes the representatives of the $K_\sharp$-orbits $K_\sharp\cdot(\e_1+it\e_2)$. 

Finally the collection of $\wt K$-orbits with $\wt H=\left\{\mathrm{1}\right\}$, i.e., with representative as in \textbf{(ii)}-\textbf{(vi)} above, is stable under any $e^{i\vartheta}$. Applying appropriate $e^{i\vartheta}$ to any representative in \textbf{(iv)}-\textbf{(vi)} we can always reach an element $z=x+iy$ with a null real component $x$. In other words any orbit 
%and it can be equivalently described as the disjoint union
%$$\displaystyle
%\cU=\bigcup_{(t_1,t_2)\in\mathrm H}\wt G_o\cdot z_{t_1,t_2}
%$$ 
%where $\wt G_o=\bR^{+}\cdot G_o\simeq \mathrm{CO}_{3}(\bR)$; we are interested in the action of $e^{i\vartheta}$ on $\cU$.
with representative in \textbf{(iv)}-\textbf{(vi)} is $S^1$-related with at least one orbit with representative in \textbf{(ii)} or \textbf{(iii)}. 

A similar argument says that the representatives in \textbf{(iii)} are representatives of the associated $K_\sharp$-orbits too, whereas representatives of $K_\sharp$-orbits as in \textbf{(ii)} are given by $\e_1+\e_3+i(t(\e_1+\e_3)+(\e_1-\e_3))$, $t\in\bR$. This concludes the proof of the first claim.

Using \eqref{mori0o} we finally see
%and the fact that any $e^{i\vartheta}\in S^1$
%acts on $V\simeq\bR^{2,1}\oplus\bR^{2,1}\simeq\bR^{4,2}$ isometrically, one finally sees
$$
N_\sharp=\begin{cases}
\bC^\times\cdot(\SO^{+}(1,1)\cup \SO^{+}(1,1)\cdot\begin{psmallmatrix} 1 & 0 & 0 \\ 0 &1  &0 \\ 0 &0 & -1 \end{psmallmatrix})\;\;\;\text{for}\;\;\e_1,\\
\bC^\times\cdot\SO_2(\bR)\;\;\;\text{for}\;\;\e_1\pm i\e_2\;\;\text{and}\;\;\e_3,\\
\bC^\times\cdot (\bR^+\ltimes\bR)\;\;\;\text{for}\;\;\e_1+\e_3,\\
\bC^\times\cdot\bZ_2\;\;\;\text{for}\;\;\e_1+it\e_2\;\;\text{where}\;-1<t<1,\, t\neq 0,\\
\bC^\times\;\;\;\text{for}\;\;\e_1+\e_3\pm i\e_2\;\;\text{and all}\;\;\e_1+\e_3+i(t(\e_1+\e_3)+(\e_1-\e_3)),
\end{cases}
$$
which readily implies the last claim.
\setcounter{equation}{0}
\qed
\end{appendix}
\bigskip\par

\end{document}